\documentclass{amsart}

\newtheorem{theorem}{Theorem}[section]
\newtheorem{proping}{Proposition}[section]
\newtheorem{coring}{Corollary}[section]
\newtheorem{lemma}[theorem]{Lemma}
\DeclareMathAlphabet{\mathpzc}{OT1}{pzc}{m}{it}

\theoremstyle{definition}
\newtheorem{definition}[theorem]{Definition}

\theoremstyle{remark}
\newtheorem{remark}[theorem]{Remark}

\numberwithin{equation}{section}

\usepackage{txfonts, pxfonts}
\begin{document}

\title{A Construction of Constant Scalar Curvature Manifolds with Delaunay-type Ends}

\author{Almir Silva Santos}
\address{Universidade Federal de Sergipe
Departamento de Matem\'atica
\;Av. Marechal Rondon s/n
49100-000 S\~ao Crist\'ov\~ao, SE, Brazil}
\curraddr{Instituto de Matem\'atica Pura e Aplicada (IMPA), Estrada Dona Castorina 110, 22460-320, Rio de Janeiro, RJ, Brazil}
\email{arss@impa.br}
%\thanks{The author was supported in part by CNPq-Brazil.}

%\author{Author Two}
%\address{Mathematical Research Section, School of Mathematical Sciences,
%Australian National University, Canberra ACT 2601, Australia}

%    General info
%\subjclass[2000]{Primary 54C40, 14E20; Secondary 46E25, 20C20}

%\date{January 1, 2001 and, in revised form, June 22, 2001.}

%\dedicatory{This paper is dedicated to my family.}


\begin{abstract}
It has been showed by Byde \cite{AB} that it is possible to attach a Delaunay-type end to a compact nondegenerate manifold of positive constant scalar curvature, provided it is locally conformally flat in a neighborhood of the attaching point. The resulting manifold is noncompact with the same constant scalar curvature. The main goal of this paper is to generalize this result. We will construct a one-parameter family of solutions to the positive singular Yamabe problem for any compact non-degenerate manifold with Weyl tensor vanishing to sufficiently high order at the singular point. If the dimension is at most 5, no condition on the Weyl tensor is needed. We will use perturbation techniques and gluing methods.
\end{abstract}

\maketitle

%\section*{Preliminaries}

%% The correct journal style for \specialsection is all uppercase; a known bug
%% in amsart.cls prevents this, so input must be uppercase until it is fixed.
%\specialsection*{This is a Special Section Head}
%\specialsection*{THIS IS A SPECIAL SECTION HEAD}
%This is an example of a special section head%
%%%%%%%%%%%%%%%%%%%%%%%%%%%%%%%%%%%%%%%%%%%%%%%%%%%%%%%%%%%%%%%%%%%%%%%%
%\footnote{Here is an example of a footnote. Notice that this footnote
%text is running on so that it can stand as an example of how a footnote
%with separate paragraphs should be written.
%\par
%And here is the beginning of the second paragraph.}%
%%%%%%%%%%%%%%%%%%%%%%%%%%%%%%%%%%%%%%%%%%%%%%%%%%%%%%%%%%%%%%%%%%%%%%%%
%\tableofcontents
%%%%%%%%%%%%%%%%%%%%%%%%Introduction%%%%%%%%%%%%%%%%%%%%%%%%%%%%%%%%%%
%
%
\section{Introduction}
In 1960 Yamabe \cite{Y} claimed that every $n-$dimensional compact Riemannian manifold $M$, $n\geq 3$, has a conformal metric of constant scalar curvature. Unfortunately, in 1968, Trudinger discovered an error in the proof. In 1984 Schoen \cite{S1}, after the works of Yamabe \cite{Y}, Trudinger \cite{T} and Aubin \cite{A}, was able to complete the proof of {\it The Yamabe Problem}:
\begin{center}
\begin{minipage}{10cm}
Let $(M^n,g_0)$ be an $n-$dimensional compact Riemannian manifold (without boundary) of dimension $n\geq 3$. Find a metric conformal to $g_0$ with constant scalar curvature.
\end{minipage}
\end{center}
See \cite{LP} and \cite{SY} for excellent reviews of the problem.

It is then natural to ask whether every noncompact Riemannian  manifold of dimension $n\geq 3$ is conformally equivalent to a complete manifold with constant scalar curvature. For noncompact manifolds with a simple structure at infinity, this question may be studied by solving the so-called {\it singular Yamabe problem}:
%
%%%%%%%%%%%%%%%%%%%%%%%%%%%%%%%%%Minipage%%%%%%%%%%%%%%%%%%%%%%%%%%%%%%%%%%%%%%
%
\begin{center}
\begin{minipage}{10cm}
Given $(M,g_0)$ an $n-$dimensional compact Riemannian manifold of dimension $n\geq 3$ and a nonempty closed set $X$ in $M$, find a complete metric $g$ on $M\backslash X$ conformal to $g_0$ with constant scalar curvature.
\end{minipage}
\end{center}
In analytical terms, since we may write $g=u^{4/(n-2)}g_0$, this problem is equivalent to finding a positive function $u$ satisfying
%
%%%%%%%%%%%%%%%%%%%%%%%%%%%%%%%%%Equation%%%%%%%%%%%%%%%%%%%%%%%%%%%%%%%%%%%%%%
%
\begin{equation}\label{eq28}
\left\{
\begin{array}{l}
\displaystyle\Delta_{g_0}u-\frac{n-2}{4(n-1)}R_{g_0}u+ \frac{n-2}{4(n-1)}Ku^{\frac{n+2}{n-2}}=0\;\;\;\mbox{ on }\;\;\;M\backslash X\\
u(x)\rightarrow\infty \mbox{ as } x\rightarrow X
\end{array}\right.
\end{equation}
where $\Delta_{g_0}$ is the Laplace-Beltrami operator associated with the metric $g_0$, $R_{g_0}$ denotes the scalar curvature of the metric $g_0$, and $K$ is a constant. We remark that the metric $g$ will be complete if $u$ tends to infinity with a sufficiently fast rate.

The singular Yamabe problem has been extensively studied in recent years, and many existence results as well as obstructions to existence are known. This problem was considered initially in the negative case by Loewner and Nirenberg \cite{LN}, when $M$ is the sphere $\mathbb{S}^{n}$ with its standard metric. In the series of papers \cite{AM1}--\cite{AM3} Aviles and McOwen have studied the case when $M$ is arbitrary. For a solution to exist on a general $n-$dimensional compact Riemannian manifold $(M,g_0)$, the size of $X$ and the sign of $R_g$ must be related to one another: it is known that if a solution exists with $R_g<0$, then $\dim X>(n-2)/2$, while if a solution exists with $R_g\geq 0$, then $\dim X\leq (n-2)/2$ and in addition the first eigenvalue of the conformal Laplacian of $g_0$ must be nonnegative. Here the dimension is to be interpreted as Hausdorff dimension. Unfortunately, only partial converses to these statements are known. For example, Aviles and McOwen \cite{AM2} proved that when $X$ is a closed smooth submanifold of dimension $k$, a solution for (\ref{eq28}) exists with $R_g<0$ if and only if $k>(n-2)/2$. We direct the reader to the papers \cite{AM1}--\cite{AM3}, \cite{F}, \cite{FM}, \cite{LN}, \cite{MP2}--\cite{MP}, \cite{MPU}--\cite{MS}, \cite{P}, \cite{S}, \cite{SY1} and the references contained therein.

In the constant negative scalar curvature case, it is possible to use the maximum principle, and solutions are constructed using barriers regardless of the dimension of $X$. See \cite{AM1}--\cite{AM3}, \cite{F}, \cite{FM} for more details.

Much is known about the constant positive scalar curvature case. When $M$ is the round sphere $\mathbb{S}^n$ and $X$ is a single point, by a result of Caffarelli, Gidas, Spruck \cite{CGS}, it is known that there is no solution of (\ref{eq28}). See \cite{MPU} for a different proof. In the case where $M$ is the sphere with its standard metric, in 1988, R. Schoen \cite{S} constructed solutions with $R_g>0$ on the complement of certain sets of Hausdorff dimension less than $(n-2)/2$. In particular, he produced solutions to (\ref{eq28}) when $X$ is a finite set of points of at least two elements. Using a different method, later in 1999, Mazzeo and Pacard proved the following existence result:
\begin{theorem}[Mazzeo--Pacard,  \cite{MP}]
Suppose that $X=X'\cup X ''$ is a disjoint union of submanifolds in $\mathbb{S}^n$, where $X'=\{p_1,\ldots,p_k\}$ is a collection of points, and $X''=\cup_{j=1}^mX_j$ where $\dim X_j=k_j$. Suppose further that $0<k_j\leq(n-2)/2$ for each $j$, and either $k=0$ or $k\geq 2$. Then there exists a complete metric $g$ on $\mathbb{S}^n\backslash X$ conformal to the standard metric on $\mathbb{S}^n$, which has constant positive scalar curvature $n(n-1)$.
\end{theorem}

Also, it is known that if $X$ is a finite set of at least two elements, and $M=\mathbb{S}^n$, the moduli space of solutions has dimension equal to the cardinality of $X$ (see \cite{MPU}).

The first result for arbitrary compact Riemannian manifolds in the positive case appeared in 1996. Mazzeo and Pacard \cite{MP2} established the following result:
\begin{theorem}[Mazzeo--Pacard, \cite{MP2}]
Let $(M,g_0)$ be any $n-$dimensional compact Riemannian manifold with constant nonnegative scalar curvature. Let $X\subset M$ be any finite disjoint union of smooth submanifolds $X_i$ of dimensions $k_i$ with $0<k_i\leq (n-2)/2$. Then there is an infinite dimensional family of complete metrics on $M\backslash X$ conformal to $g_0$ with constant positive scalar curvature.
\end{theorem}

Their method does not apply to the case in which $X$ contains isolated points. If $X=\{p\}$, an existence result was obtained by Byde in 2003 under an extra assumption. It can be stated as follows:
\begin{theorem}[A. Byde, \cite{AB}]\label{teo05}
Let $(M,g_0)$ be any $n-$dimensional compact Riemannian manifold of constant scalar curvature $n(n-1)$, nondegenerate about 1, and let $p\in M$ be a point in a neighborhood of which $g_0$ is conformally flat. There is a constant $\varepsilon_0>0$ and a one-parameter family of complete metrics $g_\varepsilon$ on $M\backslash\{p\}$ defined for $\varepsilon\in(0,\varepsilon_0)$, conformal to $g_0$, with constant scalar curvature $n(n-1)$. Moreover, $g_\varepsilon\rightarrow g_0$ uniformly on compact sets in $M\backslash\{p\}$ as $\varepsilon\rightarrow 0$.
\end{theorem}

See \cite{AB}, \cite{MZ}, \cite{MP}, \cite{MPU} and \cite{MS} for more details about the positive singular Yamabe problem.

This work is concerned with the positive singular Yamabe problem in the case $X$ is a single point (or when $X$ is finite, more generally). Our main result is the construction of solutions to the singular Yamabe problem under a condition on the Weyl tensor. If the dimension is at most 5, no condition on the Weyl tensor is needed, as we will see below. We will use the gluing method, similar to that employed by Byde \cite{AB}, Jleli \cite{M}, Jleli and Pacard \cite{JP}, Kaabachi and Pacard \cite{KP}, Kapouleas \cite{K}, Mazzeo and Pacard \cite{MP1},\cite{MP}, Mazzeo, Pacard and Pollack \cite{MPP}, \cite{MPP1}, and other  authors. Our result generalizes the result of Byde, Theorem \ref{teo05}, and it reads as follows:
%
%%%%%%%%%%%%%%%%%%%%%%%%%%%%%%%%%Theorem%%%%%%%%%%%%%%%%%%%%%%%%%%%%%%%%%%%%%%
%
\\

\noindent{\bf Main Theorem:} {\it Let $(M^n,g_0)$ be an $n-$dimensional compact Riemannian manifold of scalar curvature $n(n-1)$, nondegenerate about 1, and let $p\in M$ with $\nabla_{g_0}^kW_{g_0}(p)=0$ for $k=0,\ldots,\left[\frac{n-6}{2}\right]$, where $W_{g_0}$ is the Weyl tensor of the metric $g_0$. Then, there exist a constant $\varepsilon_0>0$ and a one-parameter family of complete metrics $g_\varepsilon$ on $M\backslash\{p\}$ defined for $\varepsilon\in(0,\varepsilon_0)$, conformal to $g_0$, with scalar curvature $n(n-1)$. Moreover, each $g_\varepsilon$ is asymptotically Delaunay and $g_\varepsilon\rightarrow g_0$ uniformly on compact sets in $M\backslash\{p\}$ as $\varepsilon\rightarrow 0$.}\\

For the gluing procedure to work, there are two restrictions on the data $(M,g_0,X)$: non-degeneracy and the Weyl vanishing condition. The non-degeneracy is defined as follows (see \cite{AB}, \cite{KMS} and \cite{MPU1}):
%
%%%%%%%%%%%%%%%%%%%%%%%%%%%%%%%%%Definition%%%%%%%%%%%%%%%%%%%%%%%%%%%%%%%%%%%%%%
%
\begin{definition}\label{def3}
A metric $g$ is {\it nondegenerate} at $u\in C^{2,\alpha}(M)$ if the operator $L_g^u: C_{\mathcal{D}}^{2,\alpha}(M)\rightarrow C^{0,\alpha}(M)$ is surjective for some $\alpha\in(0,1)$, where 
$$L_g^u(v)= \Delta_gv- \frac{n-2}{4(n-1)}R_gv+\frac{n(n+2)}{4}u^{\frac{4}{n-2}}v,$$
$\Delta_g$ is the Laplace operator of the metric $g$ and $R_g$ is the scalar curvature of $g$. Here $C^{k,\alpha}(M)$ are the standard H\"older spaces on $M$, and the $\mathcal{D}$ subscript indicates the restriction to functions vanishing on the boundary of $M$ (if there is one).
\end{definition}

Although it is the surjectivity that is used in the nonlinear analysis, it is usually easier to check injectivity. This is a corollary of the non-degeneracy condition on $M$ in conjunction with self-adjointness. For example, it is clear that the round sphere $\mathbb{S}^n$ is degenerate because $L_{g_0}^1=\Delta_{g_0}+n$ annihilates the restrictions of linear functions on $\mathbb{R}^{n+1}$ to $\mathbb{S}^n$.

As it was already expected by Chru\'sciel and Pollack \cite{CP}, when $3\leq n\leq 5$ we do not need any hypothesis about the Weyl tensor, that is, in this case, (\ref{eq28}) has a solution for any nondegenerate compact manifold $M$ and $X=\{p\}$ with $p\in M$ arbitrary. We will show in Section \ref{sec05} that the product manifolds $\mathbb{S}^2(k_1)\times \mathbb{S}^2(k_2)$ and $\mathbb{S}^2(k_3)\times \mathbb{S}^2(k_4)$ are nondegenerate except for countably many values of $k_1/k_2$ and $k_3/k_4$. Therefore our Main Theorem applies to these manifolds. We notice that they are not locally conformally flat.

Byde proved his theorem assuming that $M$ is conformally flat in a neighborhood of $p$. With this assumption, the problem gets simplified since in the neighborhood of $p$ the metric is conformal to the standard metric of $\mathbb{R}^n$, and in this case it is possible to transfer the metric on $M\backslash\{p\}$ to cylindrical coordinates, where there is a family of well-known Delaunay-type solutions. In our case we only have that the Weyl tensor vanishes to sufficiently high order at $p$. Since the singular Yamabe problem is conformally invariant, we can work in conformal normal coordinates. In such coordinates it is more convenient to work with the Taylor expansion of the metric, instead of dealing with derivatives of the Weyl tensor. As indicated in \cite{KMS}, we get some simplifications. In fact, this assumption will be fundamental to solve the problem locally in Section \ref{sec02}. We will exploit the fact that the first term in the expansion of the scalar curvature, in conformal normal coordinate, is orthogonal to the low eigenmodes. Pollack \cite{P} has indicated that it would be possible to find solutions with one singular point with some Weyl vanishing condition, as opposed to the case of the round metric on $\mathbb{S}^n$.

The motivation for $\left[\frac{n-6}{2}\right]$ in the Main Theorem comes from the {\it Weyl Vanishing Conjecture} (see \cite{S2}). It states that if a sequence $v_i$ of solutions to the equation
$$\Delta_gv_i-\frac{n-2}{4(n-1)}R_gv_i+v_i^{\frac{n+2}{n-2}}=0$$
in a compact Riemannian manifold $(M,g)$, blows-up at $p\in M$, then one should have
$$\nabla^k_gW_g(p)=0\;\;\;\mbox{ for every }\;\;0\leq k\leq \left[\frac{n-6}{2}\right].$$
Here $W_g$ denotes the Weyl tensor of the metric $g$. This conjecture has been proved by Marques for $n\leq 7$ in \cite{FCM1}, Li and Zhang for $n\leq 9$ in \cite{LZ} and for $n\leq 11$ in \cite{LZ1}, and by Khuri, Marques and Schoen for $n\leq 24$ in \cite{KMS}. The Weyl Vanishing Conjecture was in fact one of the essential pieces of the program proposed by Schoen in \cite{S2} to establish compactness in high dimensions (see \cite{KMS}). In \cite{FCM2}, based on the works of Brendle \cite{B1} and Brendle and Marques \cite{BM1}, Marques constructs counterexamples for any $n\geq 25$.

The order $\left[\frac{n-6}{2}\right]$ comes up naturally in our method, but we do not know if it is the optimal one (see Remark \ref{remark02}.)

The Delaunay metrics form the local asymptotic models for isolated singularities of locally conformally flat constant positive scalar curvature metrics, see \cite{CGS} and \cite{KMPS}. In dimensions $3\leq n\leq 5$ this also holds in the non-conformally flat setting.  In \cite{FCM}, Marques proved that if $3\leq n\leq 5$ then every solution of the equation (\ref{eq28}) with a nonremovable isolated singularity is asymptotic to a Delaunay-type solutions. This motivates us to seek solutions that are asymptotic to Delaunay. We use a perturbation argument together with the fixed point method to find solutions close to a Delaunay-type solution in a small ball centered at $p$ with radius $r$. We also construct solutions in the complement of this ball. After that, we show that for small enough $r$ the two metrics can be made to have exactly matching Cauchy data. Therefore (via elliptic regularity theory) they match up to all orders. See \cite{JP} for an application of the method.

We will indicate in the end of this paper how to handle the case of more than one point. We prove:
\begin{theorem}
Let $(M^n,g_0)$ be an $n-$dimensional compact Riemannian manifold of scalar curvature $n(n-1)$, nondegenerate about 1. Let $\{p_1,\ldots,p_k\}$ a set of points in $M$ with $\nabla^j_{g_0}W_{g_0}(p_i)=0$ for $j=0,\ldots,\left[\frac{n-6}{2}\right]$ and $i=1,\ldots,k$, where $W_{g_0}$ is the Weyl tensor of the metric $g_0$. There exists a complete metric $g$ on $M\backslash\{p_1,\ldots,p_k\}$ conformal to $g_0$, with constant scalar curvature $n(n-1)$, obtained by attaching Delaunay-type ends to the points $p_1,\ldots,p_k$.
\end{theorem}
The organization of this paper is as follows.

In Section \ref{sec03} we record some notation that will be used throughout the paper. We review some results concerning the Delaunay-type solutions, as well as the function spaces on which the linearized operator will be defined. We will recall some results about the Poisson operator for the Laplace operator $\Delta$ defined in $B_r(0)\backslash\{0\}\subset\mathbb{R}^n$ and in $\mathbb{R}^n\backslash B_r(0)$. Finally, we will review some results concerning conformal normal coordinates and scalar curvature in these coordinates.

In Section \ref{sec02}, with the assumption on the Weyl tensor and using a fixed point argument we construct a family of constant scalar curvature metrics in a small ball centered at $p\in M$, which depends on $n+2$ parameters with prescribed Dirichlet data. Moreover, each element of this family is asymptotically Delaunay.

In Section \ref{sec04}, we use the non-degeneracy of the metric $g_0$ to find a right inverse for the operator $L_{g_0}^1$ in a suitable function space. After that, we use a fixed point argument to construct a family of constant scalar curvature metrics in the complement of a small ball centered at $p\in M$, which also depends on $n+2$ parameters with prescribed Dirichlet data. Each element of this family is a perturbation of the metric $g_0$.

In Section \ref{sec05}, we put the results obtained in previous sections together to find a solution for the positive singular Yamabe problem with only one singular point. Using a fixed point argument, we examine suitable choices of the parameter sets on each piece so that the Cauchy data can be made to match up to be $C^1$ at the boundary of the ball. The ellipticity of the constant scalar curvature equation then immediately implies that the glued solutions are smooth.

Finally, in Section \ref{sec16}, we briefly explain the changes that need to be made in order to deal with more than one singular point.\\

\noindent{\bf Acknowledgements.} The content of this paper is in the author's doctoral thesis at IMPA. The author is specially grateful to his advisor Prof. Fernando C. Marques for numerous mathematical conversations and constant encouragement. The author was partially supported by CNPq-Brazil.
%
%%%%%%%%%%%%%%%%%%%%%%%%%%%%%% CHAPTER %%%%%%%%%%%%%%%%%%%%%%%%%%%%%%%
%
\section{Preliminaries}\label{sec03}
%
%%%%%%%%%%%%%%%%%%%%%%%%%%%%%%%%%%%%Section%%%%%%%%%%%%%%%%%%%%%%%%%%%%%%%%%%%
%
%\section{Introduction}
In this section we record some notation and results that will be used frequently, throughout the rest of the work and sometimes without comment. 
%
%%%%%%%%%%%%%%%%%%%%%%%%%%%%%%%%%%%%Section%%%%%%%%%%%%%%%%%%%%%%%%%%%%%%%%%%%
%
\subsection{Notation}% and Constant Scalar Curvature Equation}
Let us denote by $\theta\mapsto e_j(\theta)$, for $j\in\mathbb{N}$, the eigenfunction of the Laplace operator on $\mathbb{S}^{n-1}$ with corresponding eigenvalue $\lambda_j$. That is,
$$\Delta_{\mathbb{S}^{n-1}}e_j+\lambda_j e_j=0.$$

These eigenfunctions are restrictions to $\mathbb{S}^{n-1}\subset\mathbb{R}^n$ of homogeneous harmonic polynomials in $\mathbb{R}^n$. We further assume that these eigenvalues are counted with multiplicity, namely $\lambda_0=0,\;\lambda_1=\cdots=\lambda_n=n-1,$ $\lambda_{n+1}=2n, \ldots$  and $\lambda_j\leq\lambda_{j+1},$
and that the eigenfunctions has $L^2-$norm equal to 1. The $i-$th eigenvalue counted without multiplicity is $i(i+n-2)$.

It will be necessary to divide the function space defined on $\mathbb{S}^{n-1}_r$, the sphere with radius $r>0$, into {\it high} and {\it low eigenmode} components.

If the eigenfunction decomposition of the function $\phi\in L^2(\mathbb{S}_r^{n-1})$ is given by
$$\phi(r\theta)=\sum_{j=0}^\infty \phi_j(r)e_j(\theta)\;\;\;\mbox{ where }\;\;\;\phi_j(r)=\int_{\mathbb{S}^{n-1}}\phi(r\cdot)e_j,$$
then we define the projection $\pi_r''$ onto the {\it high eigenmode} by the formula
$$\pi''_r(\phi)(r\theta):=\sum_{j=n+1}^\infty\phi_j(r)e_j(\theta).$$
The {\it low eigenmode} on $\mathbb{S}_r^{n-1}$ is spanned by the constant functions and the restrictions to $\mathbb{S}^{n-1}_r$  of linear functions on $\mathbb{R}^n$. We always will use the variable $\theta$ for points in $\mathbb{S}^{n-1}$, and use the expression $a\cdot\theta$ to denote the dot-product of a vector $a\in\mathbb{R}^n$ with $\theta$ considered as a unit vector in $\mathbb{R}^n$.

We will use the symbols $c$, $C$, with or without subscript, to denote various positive constants.
%
%%%%%%%%%%%%%%%%%%%%%%% Subsection %%%%%%%%%%%%%%%%%%%%%%%%%%%%%%%%%%%%%%%%%
%
\subsection{Constant scalar curvature equation}

It is well known that if the metric $g_0$ has scalar curvature $R_{g_0}$, and the metric $\overline{g}=u^{4/(n-2)}g_0$ has scalar curvature $R_{\overline{g}}$, then $u$ satisfies the equation
\begin{equation}\label{eq67}
\Delta_{g_0}u-\frac{n-2}{4(n-1)}R_{g_0}u+\frac{n-2}{4(n-1)}R_{\overline{g}} u^{\frac{n+2}{n-2}}=0,
\end{equation}
see \cite{LP} and \cite{SY}.

In this work we seek solutions to the singular Yamabe problem (\ref{eq28}) when $(M^n,g_0)$ is an $n-$dimensional compact nondegenerate Riemannian, manifold with constant scalar curvature $n(n-1)$, $X$ is a single point $\{p\}$, by using a method employed by \cite{AB}, \cite{M}, \cite{JP}, \cite{MP1}--\cite{MPP1}, \cite{MS} and others. Thus, we need to find a solution $u$ for the equation (\ref{eq67}) with $R_{\overline{g}}$ constant, requiring that $u$ tends to infinity on approach to $p$.

We introduce the quasi-linear mapping $H_{g}$,
\begin{equation}\label{eq50}
H_{g}(u)=\Delta_{g}u-\frac{n-2}{4(n-1)}R_{g}u+ \frac{n(n-2)}{4}|u|^{\frac{4}{n-2}}u,
\end{equation}
and seek functions $u$ that are close to a function $u_0$, so that $H_{g}(u_0+u)=0$, $u_0+u>0$ and $(u+u_0)(x)\rightarrow+\infty$ as $x\rightarrow p$. This is done by considering the linearization of $H_{g}$ about $u_0$,
\begin{equation}\label{eq49}
L_{g}^{u_0}(u) =\left.\frac{\partial}{\partial t}H_{g}(u_0+tu)\right|_{t=0}= \mathcal{L}_{g}u+\frac{n(n+2)}{4}u_0^{\frac{4}{n-2}}u,
\end{equation}
where
$$\mathcal{L}_{g}u=\Delta_{g}u-\frac{n-2}{4(n-1)}R_{g}u$$
is the Conformal Laplacian. The operator $\mathcal{L}_{g}$ obeys the following relation concerning conformal changes of the metric
\begin{equation*}\label{eq63}
\mathcal{L}_{v^{4/(n-2)}g}u=v^{-\frac{n+2}{n-2}}\mathcal{L}_{g}(vu).
\end{equation*}

The method of finding solutions to (\ref{eq28}) used in this work is to linearize about a function $u_0$, not necessarily a solution. Expanding $H_{g}$ about $u_0$ gives
$$H_{g}(u_0+u)=H_{g}(u_0)+L_{g}^{u_0}(u)+Q^{u_0}(u),$$
where the non-linear remainder term $Q^{u_0}(u)$ is independent of the metric, and given by
%
%%%%%%%%%%%%%%%%%%%%%%%%%% Equation %%%%%%%%%%%%%%%%%%%%%%%%%%%%%%%%%%%%%%%
%
\begin{equation}\label{eq44}
Q^{u_0}(u) = \displaystyle\frac{n(n+2)}{4}u\int_0^1\left(|u_0+tu|^{\frac{4}{n-2}}- u_0^{\frac{4}{n-2}}\right)dt.
\end{equation}

It is important to emphasize here that in this work $(M^n,{g_0})$ always will be a compact Riemannian manifold of dimension $n\geq 3$ with constant scalar curvature $n(n-1)$ and nondegenerate about the constant function 1. This implies that (\ref{eq50}) is equal to
$$H_{g}(u)=\Delta_gu-\frac{n(n-2)}{4}u+\frac{n(n-2)}{4}|u|^{\frac{4}{n-2}}u$$
and the operator $L_{g_0}^1:C^{2,\alpha}(M)\rightarrow C^{0,\alpha}(M)$ given by
\begin{equation}\label{eq29}
L_{g_0}^1(v)=\Delta_{g_0}v+nv,
\end{equation}
is surjective for some $\alpha\in(0,1)$, see Definition \ref{def3}.
%
%%%%%%%%%%%%%%%%%%%%%%%%%%%%%%%%%%%%Section%%%%%%%%%%%%%%%%%%%%%%%%%%%%%%%%%%%
%
\subsection{Delaunay-type solutions}\label{sec08}
In Section \ref{sec02} we will construct a family of singular solutions to the Yamabe Problem in the punctured ball of radius $r$ centered at $p$, $B_r(p)\backslash\{p\}\subset M$, conformal to the metric ${g_0}$, with prescribed high eigenmode boundary data at $\partial B_r(p)$. It is natural to require that the solution is asymptotic to a {\it Delaunay-type solution}, called by some authors {\it Fowler solutions}. We recall some well known facts about the Delaunay-type solutions that will be used extensively in the rest of the work. See \cite{MP} and \cite{MPU} for facts not proved here.

If $g=u^{\frac{4}{n-2}}\delta$ is a complete metric in $\mathbb{R}^n\backslash\{0\}$ with constant scalar curvature $R_g=n(n-1)$ conformal to the Euclidean standard metric $\delta$ on $\mathbb{R}^n$, then $u(x)\rightarrow\infty$ when $x\rightarrow 0$ and $u$ is a solution of the equation
\begin{equation}\label{eq01}
H_\delta(u)=\Delta u+\frac{n(n-2)}{4}u^{\frac{n+2}{n-2}}=0
\end{equation}
in $\mathbb{R}^n\backslash\{0\}$. It is well known that $u$ is rotationally invariant (see \cite{CGS}, Theorem 8.1), and thus the equation it satisfies may be reduced to an ordinary differential equation. 

Therefore, if we define $v(t):=e^{\frac{2-n}{2}t}u(e^{-t}\theta)=|x|^{\frac{n-2}{2}}u(x)$, where $t=-\log |x|$ and $\theta=\frac{x}{|x|}$, then we get that
\begin{equation}\label{eq68}
v''-\frac{(n-2)^2}{4}v+\frac{n(n-2)}{4}v^{\frac{n+2}{n-2}}=0.
\end{equation}

Because of their similarity with the CMC surfaces of revolution
 discovered by {\it Delaunay} a solution of this ODE is called Delaunay-type solution.

Setting $w:=v'$ this equation is transformed into a first order Hamiltonian system
$$\left\{\begin{array}{ccl}
\displaystyle v' & = & w\\
\displaystyle w' & = & \displaystyle\frac{(n-2)^2}{4}v-\displaystyle\frac{n(n-2)}{4}v^{\frac{n+2}{n-2}}
\end{array}\right.,$$
whose Hamiltonian energy, given by
\begin{equation}\label{eq94}
H(v,w)=w^2-\frac{(n-2)^2}{4}v^2+\frac{(n-2)^2}{4}v^{\frac{2n}{n-2}},
\end{equation}
is constant along solutions of (\ref{eq68}). We summarize the basic properties of this solutions in the next proposition (see Proposition 1 in \cite{MP}).
%
%%%%%%%%%%%%%%%%%%%%%%%%% Proposition %%%%%%%%%%%%%%%%%%%%%%%%%%%%%%%%%%%%%%%%%
%
\begin{proping}\label{propo06}
For any $H_0\in(-((n-2)/n)^{n/2}(n-2)/2,0 )$, there exists a unique bounded solution of (\ref{eq68}) satisfying $H(v,v')=H_0$, $v'(0)=0$ and $v''(0)>0$. This solution is periodic, and for all $t\in\mathbb{R}$ we have $v(t)\in(0,1)$. This solution can be indexed by the parameter $\varepsilon=v(0)\in(0,((n-2)/n)^{(n-2)/4})$, which is the smaller of the two values $v$ assumes when $v'(0)=0$. When $H_0=-((n-2)/n)^{n/2}(n-2)/2$, there is a unique bounded solution of (\ref{eq68}), given by
$v(t)=((n-2)/n)^{(n-2)/4}$. Finally, if $v$ is a solution with $H_0=0$ then either $v(t)=(\cosh(t-t_0))^{(2-n)/2}$ for some $t_0\in\mathbb{R}$ or $v(t)=0$.
\end{proping}
We will write the solution of (\ref{eq68}) given by Proposition \ref{propo06} as $v_\varepsilon$, where $v_{\varepsilon}(0)=\min v_{\varepsilon}=\varepsilon\in(0,((n-2)/n)^{(n-2)/4})$ and the corresponding solution of (\ref{eq01}) as $u_\varepsilon(x)=|x|^{(2-n)/2}v_\varepsilon(-\log|x|)$.

Although we do not know them explicitly, the next proposition gives sufficient information about their behavior as $\varepsilon$ tends to zero for our  purposes.
%
%%%%%%%%%%%%%%%%%%%%%%%% Proposition %%%%%%%%%%%%%%%%%%%%%%%%%%%%%%%%%%
%
\begin{proping}\label{propo05}
For any $\varepsilon\in(0,((n-2)/n)^{(n-2)/4})$ and any $x\in\mathbb{R}^n\backslash\{0\}$ with $|x|\leq 1$, the Delaunay-type solution $u_\varepsilon(x)$ satisfies the estimates
$$\left|u_\varepsilon(x)-\frac{\varepsilon}{2}(1+|x|^{2-n})\right|\leq c_n\varepsilon^{\frac{n+2}{n-2}}|x|^{-n},$$
$$\left||x|\partial_ru_\varepsilon(x)+\frac{n-2}{2}\varepsilon|x|^{2-n}\right|\leq c_n\varepsilon^{\frac{n+2}{n-2}}|x|^{-n}$$
and
$$\left||x|^2\partial_r^2u_\varepsilon(x)-\frac{(n-2)^2}{2}\varepsilon|x|^{2-n}\right|\leq c_n\varepsilon^{\frac{n+2}{n-2}}|x|^{-n},$$
for some positive constant $c_n$ that depends only on $n$.
\end{proping}
%
%%%%%%%%%%%%%%%%%%%%%%%% Proof %%%%%%%%%%%%%%%%%%%%%%%%%%%%%%%%%%
%
\begin{proof} See \cite{MP}.
\end{proof}
As indicate by Mazzeo-Pacard \cite{MP}, there are some important variations of these solutions, leading to a $(2n+2)-$dimensional family of Delaunay-type solutions. For our purpose, it is enough to consider the family of solutions where only translations along Delaunay axis and of the "point at infinity" are allowed. Therefore, we will work with the following family of solutions of (\ref{eq01})
\begin{equation}\label{eq52}
u_{\varepsilon,R,a}(x):= |x-a|x|^2|^{\frac{2-n}{2}}v_{\varepsilon}(-2\log|x|+\log|x-a|x|^2|+\log R).
\end{equation}
See \cite{MP} for details.

In Section \ref{sec02} we will find solutions to the singular Yamabe problem in the punctured ball $B_r(p)\backslash\{p\}$ only with prescribed high eigenmode Dirichlet data, so we need other parameters to control the low eigenmode. The parameters $a\in\mathbb{R}^n$ and $R\in\mathbb{R}^+$ in (\ref{eq52}) will allow us to have control over the low eigenmode. The first corollary is a direct consequence of (\ref{eq52}) and it will control the space spanned by the coordinates functions, and the second one follows from Proposition \ref{propo05} and it will control the space spanned by the constant functions in the sphere.

\noindent{\bf Notation:} We write $f=O'(Kr^k)$ to mean $f=O(Kr^{k})$ and $\nabla f=O(Kr^{k-1})$, for $K>0$ constant. $O''$ is defined similarly.
%
%%%%%%%%%%%%%%%%%%%%%%%%%% Corollary %%%%%%%%%%%%%%%%%%%%%%%%%%%%%%%%%%%%%%%
%
\begin{coring}\label{cor02}
There exists a constant $r_0\in(0,1)$, such that for any $x$ and $a$ in $\mathbb{R}^n$ with $|x|\leq 1$, $|a||x|<r_0$, $R\in\mathbb{R}^+$, and  $\varepsilon\in(0,((n-2)/n)^{(n-2)/4})$ the solution $u_{\varepsilon,R,a}$ satisfies the estimates
%
%%%%%%%%%%%%%%%%%%%%%%%%%% Equation %%%%%%%%%%%%%%%%%%%%%%%%%%%%%%%%%%%%%%%
%
\begin{equation}\label{eq37}
u_{\varepsilon,R,a}(x) = u_{\varepsilon,R}(x)+((n-2)u_{\varepsilon,R}(x)+ |x| \partial_ru_{\varepsilon,R}(x))a\cdot x+  O''(|a|^2|x|^{\frac{6-n}{2}})
\end{equation}
and
\begin{equation}\label{eq93}
u_{\varepsilon,R,a}(x) = u_{\varepsilon,R}(x)+((n-2)u_{\varepsilon,R}(x)+ |x| \partial_ru_{\varepsilon,R}(x))a\cdot x+ O''(|a|^2\varepsilon R^{\frac{2-n}{2}}|x|^2)
\end{equation}
if $R\leq |x|$.
\end{coring}
\begin{proof} Using the Taylor's expansion we obtain that
$$ v_\varepsilon\left(-\log |x|+\log\left|\displaystyle \frac{x}{|x|}-a|x|\right|+\log R\right) = v_{\varepsilon}(-\log |x|+\log R)$$
$$\begin{array}{ll}
\hspace{3cm} - & v_\varepsilon'(-\log|x|+\log R)a\cdot x + v_\varepsilon'(-\log|x|+\log R)O''(|a|^2|x|^2)\\
\\
\hspace{3cm} + & v_\varepsilon''(-\log|x|+\log R+t_{a,x})O''(|a|^2|x|^2)
\end{array}$$
for some $t_{a,x}\in\mathbb{R}$ with $0<|t_{a,x}|<\left|\log\left|\frac{x}{|x|}-a|x|\right|\right|$, since

$$\log\left|\displaystyle\frac{x}{|x|}-a|x|\right|= -a\cdot x+O''(|a|^2|x|^2).$$
for $|a||x|<r_0$ and some $r_0\in(0,1)$. Observe that $t_{a,x}\rightarrow 0$ as $|a||x|\rightarrow 0$.

Now, by the equation (\ref{eq68}) and the fact that 
$$H(v_\varepsilon,v_\varepsilon')=\frac{(n-2)^2}{4} \varepsilon^2(\varepsilon^{\frac{n+2}{n-2}}-1),$$
where $H$ is defined in (\ref{eq94}), it follows that $|v'_\varepsilon|\leq c_nv_\varepsilon$,  $|v''_\varepsilon|\leq c_nv_\varepsilon$, for some constant $c_n$ that depends only on $n$. 

Since $-\log|x|+\log R\leq 0$ if $R\leq |x|$, and $v_\varepsilon(t)\leq \varepsilon e^{\frac{n-2}{2}|t|}$, for all $t\in\mathbb{R}$ (see \cite{MP}), we obtain that
$$v_\varepsilon(-\log|x|+\log R)\leq \varepsilon R^{\frac{2-n}{2}}|x|^{\frac{n-2}{2}}$$
and
$$v_\varepsilon(-\log|x|+\log R+t_{a,x})\leq c\varepsilon R^{\frac{2-n}{2}}|x|^{\frac{n-2}{2}},$$
for some constant $c>0$ that does not depend on $x$, $\varepsilon$, $R$ and $a$.

Therefore, from (\ref{eq52}), $0<v_\varepsilon(t)\leq 1$ and
$$|x-a|x|^2|^{\frac{2-n}{2}} = \displaystyle|x|^{\frac{2-n}{2}}+\frac{n-2}{2}a\cdot x|x|^{\frac{2-n}{2}}+O''(|a|^2|x|^{\frac{6-n}{2}})$$
for $|a||x|<r_0$ and some $r_0\in(0,1)$, we deduce the result.
\end{proof}
%
%%%%%%%%%%%%%%%%%%%%%%%%%% Corollary %%%%%%%%%%%%%%%%%%%%%%%%%%%%%%%%%%%%%%%
%
\begin{coring}\label{cor04}
For any $\varepsilon\in(0,((n-2)/n)^{(n-2)/4})$ and any $x$ in $\mathbb{R}^n$ with $|x|\leq 1$, the function $u_{\varepsilon,R}:=u_{\varepsilon,R,0}$ satisfies the estimates
%
%%%%%%%%%%%%%%%%%%%%%%%%%% Equation %%%%%%%%%%%%%%%%%%%%%%%%%%%%%%%%%%%%%%%
%
$$ u_{\varepsilon,R}(x) = \displaystyle\frac{\varepsilon}{2} \left(R^{\frac{2-n}{2}}+ R^{\frac{n-2}{2}}|x|^{2-n}\right)+ O''(R^{\frac{n+2}{2}}\varepsilon^{\frac{n+2}{n-2}}|x|^{-n}),$$
$$|x|\partial_ru_{\varepsilon,R}(x)=\frac{2-n}{2}\varepsilon R^{\frac{n-2}{2}} |x|^{2-n}+ O'(R^{\frac{n+2}{2}}\varepsilon^{\frac{n+2}{n-2}}|x|^{-n})$$
and
$$|x|^2\partial_r^2u_{\varepsilon,R}(x)=\frac{(n-2)^2}{2}\varepsilon R^{\frac{n-2}{2}} |x|^{2-n}+ O(R^{\frac{n+2}{2}}\varepsilon^{\frac{n+2}{n-2}}|x|^{-n}).$$
\end{coring}
\begin{proof} Use Proposition \ref{propo05} and the fact that $u_{\varepsilon,R}(x)=R^{\frac{2-n}{2}}u_\varepsilon(R^{-1}x)$.
\end{proof}
%
%%%%%%%%%%%%%%%%%%%%%%%%%%%%%%%%%%%% Section %%%%%%%%%%%%%%%%%%%%%%%%%%%%%%%%%%%
%
\subsection{Function spaces}\label{sec09}
Now, we will define some function spaces that we will use in this work. The first one is the weighted H\"older spaces in the punctured ball. They are the most convenient spaces to define the linearized operator. The second one appears so naturally in our results that it is more helpful to put its definition here. Finally, the third one is the weighted H\"older spaces in which the exterior analysis will be carried out. These are essentially the same weighted spaces as in \cite{M}, \cite{JP} and \cite{MP}.
%
%%%%%%%%%%%%%%%%%%%%%%%%%% Definition %%%%%%%%%%%%%%%%%%%%%%%%%%%%%%%%%%%%%%%
%
\begin{definition}\label{def1}
For each $k\in\mathbb{N}$, $r>0$, $0<\alpha<1$ and $\sigma\in(0,r/2)$, let $u\in C^k(B_r(0)\backslash\{0\})$, set
$$\|u\|_{(k,\alpha),[\sigma,2\sigma]}=\sup_{|x|\in[\sigma,2\sigma]}\left(\sum_{j=0}^{k}\sigma^j |\nabla^ju(x)|\right)+\sigma^{k+\alpha}\sup_{|x|,|y|\in[\sigma,2\sigma]} \frac{|\nabla^ku(x)-\nabla^ku(y)|}{|x-y|^{\alpha}}.$$
Then, for any $\mu\in\mathbb{R}$, the space $C^{k,\alpha}_{\mu}(B_r(0)\backslash\{0\})$ is the collection of functions $u$ that are locally in $C^{k,\alpha}(B_r(0)\backslash\{0\})$ and for which the norm
$$\|u\|_{(k,\alpha),\mu,r}=\sup_{0<\sigma\leq\frac{r}{2}}\sigma^{-\mu} \|u\|_{(k,\alpha),[\sigma,2\sigma]}$$
is finite.
\end{definition}

The one result about these that we shall use frequently, and without comment, is that to check if a function $u$ is an element of some $C_{\mu}^{0,\alpha}$, say, it is sufficient to check that $|u(x)|\leq C|x|^\mu$ and $|\nabla u(x)|\leq C|x|^{\mu-1}$. In particular, the function $|x|^{\mu}$ is in $C_{\mu}^{k,\alpha}$ for any $k,$ $\alpha$, or $\mu$.

Note that $C_{\mu}^{k,\alpha}\subseteq C_{\delta}^{l,\alpha}$ if $\mu\geq\delta$ and $k\geq l$, and $\|u\|_{(l,\alpha),\delta}\leq C\|u\|_{(k,\alpha),\mu}$ for all $u\in C_{\mu}^{k,\alpha}$.
%
%%%%%%%%%%%%%%%%%%%%%%%%%% Definition %%%%%%%%%%%%%%%%%%%%%%%%%%%%%%%%%%%%%%%
%
\begin{definition}\label{def2}
For each $k\in\mathbb{N}$, $0<\alpha<1$ and $r>0$. Let $\phi\in C^k(\mathbb{S}_r^{n-1})$, set
$$\|\phi\|_{(k,\alpha),r}:=\|\phi(r\cdot)\|_{C^{k,\alpha}(\mathbb{S}^{n-1})}.$$
Then, the space $C^{k,\alpha}(\mathbb{S}^{n-1}_r)$ is the collection of functions $\phi\in C^k(\mathbb{S}^{n-1}_r)$ for which the norm $\|\phi\|_{(k,\alpha),r}$ is finite.
\end{definition}

The next lemma show a relation between the norm of Definition \ref{def1} and \ref{def2}. To prove it use the decomposition of the function spaces in the sphere.
%
%%%%%%%%%%%%%%%%%%%%%%%%%% Lemma %%%%%%%%%%%%%%%%%%%%%%%%%%%%%%%%%%%%%%%
%
\begin{lemma}\label{lem07}
Let $\alpha\in(0,1)$ and $r>0$ be constants. Then, there exists a constant $c>0$ that does not depend on $r$, such that
%
%%%%%%%%%%%%%%%%%%%%%%%%%% Equation %%%%%%%%%%%%%%%%%%%%%%%%%%%%%%%%%%%%%%%
%
\begin{equation}\label{eq43}
\|\pi_r''(u_r)\|_{(2,\alpha),r}\leq cK
\end{equation}
and
%
%%%%%%%%%%%%%%%%%%%%%%%%%% Equation %%%%%%%%%%%%%%%%%%%%%%%%%%%%%%%%%%%%%%%
%
\begin{equation}\label{eq48}
\|r\partial_r\pi_r''(u_r)\|_{(1,\alpha),r}\leq cK,
\end{equation}
for all function $u:\{x\in\mathbb{R}^n;r/2\leq|x|\leq r\}\rightarrow\mathbb{R}$ satisfying $\|u\|_{(2,\alpha),[r/2,r]}\leq K,$ for some constant $K>0$. Here, $u_r$ is the restriction of $u$ to the sphere of radius $r$, $\mathbb{S}^{n-1}_r\subset\mathbb{R}^n$.
\end{lemma}
%
%%%%%%%%%%%%%%%%%%%%%%%%%%%%%%%%% Remark %%%%%%%%%%%%%%%%%%%%%%%%%%%%%%%%%%%
%
\begin{remark}
We often will write $\pi''(C^{k,\alpha}(\mathbb{S}_r^{n-1}))$ and $\pi''(C^{k,\alpha}_\mu(B_r(0)\backslash\{0\}))$ for 
$$\{\phi\in C^{k,\alpha}(\mathbb{S}_r^{n-1});\pi''_r(\phi)=\phi\}$$
and 
$$\left\{u\in C^{k,\alpha}_\mu(B_r(0) \backslash\{0\});\pi''_s(u(s\cdot))(\theta) = u(s\theta), \forall s\in(0,r)\mbox{ and }\forall\theta\in\mathbb{S}_r^{n-1}\right\},$$
respectively.
\end{remark}
Next, consider $(M,g)$ an $n-$dimensional compact Riemannian manifold and $\Psi:B_{r_1}(0)\rightarrow M$ some coordinate system on $M$ centered at some point $p\in M$, where $B_{r_1}(0)\subset\mathbb{R}^n$ is the ball of radius $r_1>0$.

For $0<r<s\leq r_1$ define
$$M_r:=M\backslash\Psi(B_r(0))\;\;\;\mbox{ and }\;\;\;\Omega_{r,s}:=\Psi(A_{r,s}),$$
where $A_{r,s}:=\{x\in\mathbb{R}^n;r\leq|x|\leq s\}$.
%
%%%%%%%%%%%%%%%%%%%%%%%%%%%%%%%%% Definition %%%%%%%%%%%%%%%%%%%%%%%%%%%%%%%%%%%
%
\begin{definition}\label{def4}
For all $k\in\mathbb{N}$, $\alpha\in(0,1)$ and $\nu\in\mathbb{R}$, the space $C_{\nu}^{k,\alpha}(M\backslash\{p\})$ is the space of functions $v\in C_{loc}^{k,\alpha}(M\backslash\{p\})$ for which the following norm is finite
$$\|v\|_{C_\nu^{k,\alpha}(M\backslash\{p\})}:= \|v\|_{C^{k,\alpha}(M_{\frac{1}{2}r_1})} +\|v\circ\Psi\|_{(k,\alpha),\nu,r_1},$$
where the norm $\|\cdot\|_{(k,\alpha),\nu,r_1}$ is the one defined in Definition \ref{def1}.
\end{definition}

For all $0<r<s\leq r_1$, we can also define the spaces $C_\mu^{k,\alpha}(\Omega_{r,s})$ and $C_\mu^{k,\alpha}(M_r)$ to be the space of restriction of elements of $C_\mu^{k,\alpha}(M\backslash\{p\})$ to $M_r$ and $\Omega_{r,s}$, respectively. These spaces is endowed with the following norm
$$\|f\|_{C^{k,\alpha}_\mu(\Omega_{r,s})}:=\sup_{r\leq \sigma\leq\frac{s}{2}}\sigma^{-\mu}\|f\circ \Psi\|_{(k,\alpha),[\sigma,2\sigma]}$$
and
$$\|h\|_{C^{k,\alpha}_\mu(M_r)}:=\|h\|_{C^{k,\alpha}(M_{\frac{1}{2}r_1})}+ \|h\|_{C^{k,\alpha}_\mu(\Omega_{r,r_1})}.$$
%
%%%%%%%%%%%%%%%%%%%%%%%%%%%%%%%%%%%% Section %%%%%%%%%%%%%%%%%%%%%%%%%%%%%%%%%%%
%
\subsection{The linearized operator}\label{sec01}
Let us fix one of the solutions of (\ref{eq01}), $u_{\varepsilon,R,a}$ given by (\ref{eq52}). Hence, $u_{\varepsilon,R,a}$ satisfies $H_\delta(u_{\varepsilon,R,a})=0$. The linearization of $H_\delta$ at $u_{\varepsilon,R,a}$ is defined by
\begin{equation}\label{eq38}
L_{\varepsilon,R,a}(v):=L_\delta^{u_{\varepsilon,R,a}}(v)=\Delta v+\frac{n(n+2)}{4}u_{\varepsilon,R,a}^{\frac{4}{n-2}}v,
\end{equation}
where $L_\delta^{u_{\varepsilon,R,a}}$ is given by (\ref{eq49}).

In \cite{MP}, Mazzeo and Pacard studied the operator $L_{\varepsilon,R}:=L_{\varepsilon,R,0}$ defined in weighted H\"older spaces. They showed that there exists a suitable right inverse with two important features, the corresponding right inverse has norm bounded independently of $\varepsilon$ and $R$ when the weight is chosen carefully, and the weight can be improved if the right inverse is defined in the high eigenmode. These properties will be fundamental in Section \ref{sec02}. To summarize, they establish the following result.
%
%%%%%%%%%%%%%%%%%%%%%%%%%% Proposition %%%%%%%%%%%%%%%%%%%%%%%%%%%%%%%%%%%%%%%
%
\begin{proping}[Mazzeo--Pacard, \cite{MP}]\label{propo01}
Let $R\in\mathbb{R}^+,$ $\alpha\in(0,1)$ and $\mu\in\left(1,2\right)$.
Then there exists $\varepsilon_0>0$ such that, for all $\varepsilon\in(0,\varepsilon_0]$, there is an operator
$$G_{\varepsilon,R}:C_{\mu-2}^{0,\alpha}(B_1(0)\backslash\{0\})\rightarrow C_{\mu}^{2,\alpha}(B_1(0)\backslash\{0\})$$
with the norm bounded independently of $\varepsilon$ and $R$, such that for $f\in C_{\mu-2}^{0,\alpha}(B_1(0)\backslash\{0\})$, the function $w:=G_{\varepsilon,R}(f)$ solves the equation
%
%%%%%%%%%%%%%%%%%%%%%%%%%% Equation %%%%%%%%%%%%%%%%%%%%%%%%%%%%%%%%%%%%%%%
%
\begin{equation}\label{eq02}
\left\{\begin{array}{lcl}
L_{\varepsilon,R}(w)=f & \mbox{ in } & B_1(0)\backslash\{0\}\\
\pi''_1(w|_{\mathbb{S}^{n-1}})=0 & \mbox{ on } & \partial B_1(0)
\end{array}\right..
\end{equation}
Moreover, if $f\in\pi''(C^{0,\alpha}_{\mu-2}(B_1(0)\backslash\{0\}))$, then $w\in\pi''(C_\mu^{2,\alpha}(B_1(0)\backslash\{0\}))$ and we may take $\mu\in(-n,2)$.
\end{proping}
\begin{proof} The statement in \cite{MP} is that for each fixed $R$ the norm of $G_{\varepsilon,R}$ is independent of $\varepsilon$, but this bound might depend on $R$. In \cite{AB}, Byde observed that the norm of $G_{\varepsilon,R}$ also does not depend on $R$.
\end{proof}
We will work in $B_r(0)\backslash\{0\}$ with $0<r\leq 1$, then it is convenient to study the operator $L_{\varepsilon,R}$ in function spaces defined in $B_r(0)\backslash\{0\}$. 

Let $f\in C_{\mu-2}^{0,\alpha}(B_1(0)\backslash\{0\})$ and $w\in C_\mu^{2,\alpha}(B_1(0)\backslash\{0\})$ be solution of (\ref{eq02}). Considering $g(x)=r^{-2}f(r^{-1}x)$ and $w_{r}(x)=w(r^{-1}x)$ we get that (\ref{eq02}) is equivalent to
$$\left\{\begin{array}{lcl}
L_{\varepsilon,rR}(w_{r})=g & \mbox{ in } & B_r(0)\backslash\{0\}\\
\pi''_r(w_{r}|_{\mathbb{S}_r^{n-1}})=0 & \mbox{ on } &  \partial B_r(0)
\end{array}\right..$$
Furthermore, since $\nabla^j w_{r}(x)=r^{-j}\nabla^jw(r^{-1}x)$ and $\nabla^jg(x)=r^{-2-j}\nabla^jf(r^{-1}x)$, we get
$$\|w_{r}\|_{(2,\alpha),\mu,r}\leq c\|g\|_{(0,\alpha),\mu-2,r},$$
where $c>0$ is a constant that does not depend on $\varepsilon$, $r$ and $R$. Thus, we obtain the following corollary.
%
%%%%%%%%%%%%%%%%%%%%%%%%%% Corollary %%%%%%%%%%%%%%%%%%%%%%%%%%%%%%%%%%%%%%%
%
\begin{coring}\label{cor03}
Let $\mu\in\left(1,2\right)$, $\alpha\in(0,1)$, $\varepsilon_0>0$ given by Proposition \ref{propo01}. Then for all $\varepsilon\in(0,\varepsilon_0)$, $R\in\mathbb{R}^+$ and $0<r\leq 1$ there is an operator
$$G_{\varepsilon,R,r}:C_{\mu-2}^{0,\alpha}(B_r(0)\backslash\{0\})\rightarrow C_\mu^{2,\alpha}(B_r(0)\backslash\{0\})$$
with norm bounded independently of $\varepsilon,$ $R$ and $r$, such that for each $f\in C_{\mu-2}^{0,\alpha}(B_r(0)\backslash\{0\})$, the function $w:=G_{\varepsilon,R,r}(f)$ solves the equation
\begin{equation*}
\left\{\begin{array}{lcl}
L_{\varepsilon,R}(w)=f & \mbox{ in } & B_r(0)\backslash\{0\}\\
\pi''_r(w|_{\mathbb{S}^{n-1}_r})=0 & \mbox{ on } & \partial B_r(0)
\end{array}\right..
\end{equation*}
Moreover, if $f\in\pi''(C^{0,\alpha}_{\mu-2}(B_r(0)\backslash\{0\}))$, then $w\in\pi''(C^{2,\alpha}_{\mu}(B_r(0)\backslash\{0\}))$ and we may take $\mu\in(-n,2)$.
\end{coring}

In fact, we will work with the solution $u_{\varepsilon,R,a}$, and so, we need to find an inverse to $L_{\varepsilon,R,a}$ with norm bounded independently of $\varepsilon$, $R$, $a$ and $r$. But this is the content of the next corollary, whose proof is a perturbation argument.
%
%%%%%%%%%%%%%%%%%%%%%%%%%% Corollary %%%%%%%%%%%%%%%%%%%%%%%%%%%%%%%%%%%%%%%
%
\begin{coring}\label{cor01}
Let $\mu\in\left(1,2\right)$, $\alpha\in(0,1)$, $\varepsilon_0>0$ given by Proposition \ref{propo01}. Then for all $\varepsilon\in(0,\varepsilon_0)$, $R\in\mathbb{R}^+$, $a\in\mathbb{R}^n$ and $0<r\leq 1$ with $|a|r\leq r_0$ for some $r_0\in(0,1)$, there is an operator
$$G_{\varepsilon,R,r,a}:C_{\mu-2}^{0,\alpha}(B_r(0)\backslash\{0\})\rightarrow C_\mu^{2,\alpha}(B_r(0)\backslash\{0\}),$$
with norm bounded independently of $\varepsilon$, $R$, $r$ and $a$, such that for each $f\in C_{\mu-2}^{0,\alpha}(B_r(0)\backslash\{0\})$, the function $w:=G_{\varepsilon,R,r,a}(f)$ solves the equation
\begin{equation*}
\left\{\begin{array}{lcl}
L_{\varepsilon,R,a}(w)=f & \mbox{ in } & B_r(0)\backslash\{0\}\\
\pi''_r(w|_{\mathbb{S}^{n-1}_r})=0 & \mbox{ on } & \partial B_r(0)
\end{array}\right..
\end{equation*}
\end{coring}
\begin{proof} We will use a perturbation argument. Thus,
$$(L_{\varepsilon,R,a}-L_{\varepsilon,R})v= \frac{n(n+2)}{4}\left(u_{\varepsilon,R,a}^{\frac{4}{n-2}}- u_{\varepsilon,R}^{\frac{4}{n-2}}\right)v$$
implies
$$\|(L_{\varepsilon,R,a}-L_{\varepsilon,R})v\|_{(0,\alpha),[\sigma,2\sigma]}\leq c\|u_{\varepsilon,R,a}^{\frac{4}{n-2}}-u_{\varepsilon,R}^{\frac{4}{n-2}}\| _{(0,\alpha),[\sigma,2\sigma]}\|v\|_{(0,\alpha),[\sigma,2\sigma]},$$
where $c>0$ does not depend on $\varepsilon$, $R$, $a$ and $r$.

Note that
$$v_\varepsilon^{\frac{4}{n-2}}\left(-\log|x|+\log\left |\frac{x}{|x|}-a|x|\right|+\log R\right)= v_\varepsilon^{\frac{4}{n-2}}(-\log|x|+\log R)$$
$$+\frac{4}{n-2}\int^{\log\left|\frac{x}{|x|}-a|x|\right|}_0 \left(v_\varepsilon^{\frac{6-n}{n-2}}v'_\varepsilon\right) (-\log |x|+\log R+t)dt.$$

Therefore, from (\ref{eq52}) and the expansion $|x-a|x|^2|^{-2}=|x|^{-2}+O(|a||x|^{-1})$, we get
$$u_{\varepsilon,R,a}^{\frac{4}{n-2}}(x)=u_{\varepsilon,R }^{\frac{4}{n-2}}(x)+\displaystyle \frac{4|x|^{-2}}{n-2}\int_0^{\log\left|\frac{x}{|x|}-a|x|\right|} \left(v_\varepsilon^{\frac{6-n}{n-2}}v_\varepsilon'\right) (-\log|x|+\log R+t)dt$$
$$\begin{array}{l}
+O(|a||x|^{-1})v_\varepsilon^{\frac{4}{n-2}}\left(-\log|x|+\log\left| \displaystyle \frac{x}{|x|}-a|x|\right|+\log R\right).
\end{array}$$

From the proof of Corollary \ref{cor02} we know that $|v_\varepsilon'|\leq c_nv_\varepsilon$. Hence,
$$|u_{\varepsilon,R,a}^{\frac{4}{n-2}}(x)- u_{\varepsilon,R}^{\frac{4}{n-2}}(x)|\leq  c_n|x|^{-2}\displaystyle\int_0^{O(|a||x|)} v_\varepsilon^{\frac{4}{n-2}}(-\log|x|+\log R+t)dt+O(|a||x|^{-1}),$$
since $\log ||x|^{-1}-a|x||=O(a|x|)$ and $0<\varepsilon\leq v_\varepsilon\leq 1$. Thus
\begin{equation}\label{eq76}
|u_{\varepsilon,R,a}^{\frac{4}{n-2}}(x)- u_{\varepsilon,R}^{\frac{4}{n-2}}(x)|\leq c_n|a||x|^{-1},
\end{equation}
where the constant $c>0$ does not depend on $\varepsilon$, $R$ and $a$.

The estimate for the full H\"older norm is similar.

Hence
$$\|u_{\varepsilon,R,a}^{\frac{4}{n-2}}-u_{\varepsilon,R}^{\frac{4}{n-2}} \|_{(0,\alpha),[\sigma,2\sigma]}\leq c|a|\sigma^{-1}$$
and then
$$\|(L_{\varepsilon,R,a}-L_{\varepsilon,R})v\|_{(0,\alpha),\mu-2,r}\leq c|a|r \|v\|_{(2,\alpha),\mu,r},$$
where $c>0$ is a constant that does not depend on $\varepsilon$, $R$, $a$ and $r$.

Therefore, $L_{\varepsilon,R,a}$ has a bounded right inverse for small enough $|a|r$ and this inverse has norm bounded independently of $\varepsilon$, $R$, $a$ and $r$. In fact, if we choose $r_0$ so that $r_0\leq \frac{1}{2}K^{-1}$, where the constant $K>0$ satisfies $\|G_{\varepsilon,R,r}\|\leq K$ for all $\varepsilon\in(0,\varepsilon_0)$, $R\in\mathbb{R}^+$ and $r\in(0,1)$, then
$$\|L_{\varepsilon,R,a}\circ G_{\varepsilon,R,r}-I\|\leq \|L_{\varepsilon,R,a}-L_{\varepsilon,R}\|\|G_{\varepsilon,R,r}\|\leq \frac{1}{2}.$$
This implies that $L_{\varepsilon,R,a}\circ G_{\varepsilon,R,r}$ has a bounded right inverse given by 
$$(L_{\varepsilon,R,a}\circ G_{\varepsilon,R,r})^{-1}:=\sum_{i=0}^{\infty}(I-L_{\varepsilon,R,a}\circ G_{\varepsilon,R,r})^i,$$
and it has norm bounded independently of $\varepsilon$, $R$, $a$ and $r$, in fact less than 1.

Therefore we define a right inverse of $L_{\varepsilon,R,a}$ as $G_{\varepsilon,R,r,a}:=G_{\varepsilon,R,r}\circ(L_{\varepsilon,R,a}\circ G_{\varepsilon,R,r})^{-1}$.
\end{proof}

%
%%%%%%%%%%%%%%%%%%%%%%%%%%%%%%%%%%%% Section %%%%%%%%%%%%%%%%%%%%%%%%%%%%%%%%%%%
%
\subsection{Poisson operator associated to the Laplacian $\Delta$}\label{sec10}
%
%%%%%%%%%%%%%%%%%%%%%%%%%%%%%%%%%%%% subsection %%%%%%%%%%%%%%%%%%%%%%%%%%%%%%%%
%
\subsubsection{Laplacian $\Delta$ in $B_r(0)\backslash\{0\}\subset\mathbb{R}^n$}
Since $\pi_r''(G_{\varepsilon,R,r,a}(f)|_{\mathbb{S}^{n-1}_r})=0$ on $\partial B_r(0)$, we need to find some way to prescribe the high eigenmode boundary data at $\partial B_r(0)$. This is done using the Poisson operator associated to the Laplacian $\Delta$.
%
%%%%%%%%%%%%%%%%%%%%%%%%%% Proposition %%%%%%%%%%%%%%%%%%%%%%%%%%%%%%%%%%%%%%%
%
\begin{proping}\label{propo02}
Given $\alpha\in(0,1)$, there is a bounded operator
$$\mathcal{P}_1:\pi_1''(C^{2,\alpha}(\mathbb{S}^{n-1}))\longrightarrow \pi_1''(C^{2,\alpha}_2(B_1(0)\backslash\{0\})),$$
so that
$$\left\{\begin{array}{l}
\Delta(\mathcal{P}_1(\phi))=0\;\;\;\mbox{ in }\;\;\;B_1(0)\\
\pi_1''(\mathcal{P}_1(\phi)|_{\mathbb{S}^{n-1}})=\phi \;\;\;\mbox{ on }\;\;\; \partial B_1(0)
\end{array}\right..$$
\end{proping}
\begin{proof} See Proposition 2.2 in \cite{AB}, Proposition 11.25 in \cite{M} and Lemma 6.2 in \cite{PR}.
\end{proof}
For $\mu\leq 2$ and $0<r\leq 1$ we can define an analogous operator,
$$\mathcal{P}_r:\pi_r''(C^{2,\alpha}(\mathbb{S}_r^{n-1}))\longrightarrow \pi_r''(C^{2,\alpha}_\mu(B_r(0)\backslash\{0\}))$$
as 
\begin{equation}\label{eq62}
\mathcal{P}_r(\phi_r)(x)=\mathcal{P}_1(\phi)(r^{-1}x),
\end{equation}
where $\phi(\theta):=\phi_r(r\theta)$. By Proposition \ref{propo02} we deduce that
$$\left\{\begin{array}{lcl}
\Delta(\mathcal{P}_r(\phi_r))=0 & \mbox{ in } & B_r(0)\backslash\{0\}\\
\pi_r''(\mathcal{P}_r(\phi_r)|_{\mathbb{S}_r^{n-1}})=\phi_r & \mbox{ on } & \partial B_r(0)
\end{array}\right.$$
and
\begin{equation}\label{eq73}
\|\mathcal{P}_r(\phi_r)\|_{(2,\alpha),\mu,r}\leq C r^{-\mu}\|\phi_r\|_{(2,\alpha),r},
\end{equation}
where the constant $C>0$ does not depend on $r$ and the norm $\|\phi_r\|_{(2,\alpha),r}$ is defined in Definition \ref{def2}.
%
%%%%%%%%%%%%%%%%%%%%%%%%%%%%%%%%%%%% subsection %%%%%%%%%%%%%%%%%%%%%%%%%%%%%%%%
%
\subsubsection{Laplacian $\Delta$ in $\mathbb{R}^n\backslash B_r(0)$}\label{sec06}
For the same reason as before we will need a Poisson operator associated to the Laplacian $\Delta$ defined in $\mathbb{R}^n\backslash B_r(0)$.
%
%%%%%%%%%%%%%%%%%%%%%%%%%%%%%%%%% Proposition %%%%%%%%%%%%%%%%%%%%%%%%%%%%%%%%%%
%
\begin{proping}\label{lem4}
Assume that $\varphi\in C^{2,\alpha}(\mathbb{S}^{n-1})$ and let $\mathcal{Q}_1(\varphi)$ be the only solution of
$$\left\{\begin{array}{lcl}
\Delta v=0 & \mbox{ in } & \mathbb{R}^n\backslash B_1(0)\\
v=\varphi & \mbox{ on } & \partial B_1(0)
\end{array}\right.$$
which tends to 0 at $\infty$. Then
$$\|\mathcal{Q}_1(\varphi)\|_{C^{2,\alpha}_{1-n}(\mathbb{R}^n\backslash B_1(0))}\leq C\|\varphi\|_{(2,\alpha),1},$$
if $\varphi$ is $L^2-$orthogonal to the constant function.
\end{proping}
%
%%%%%%%%%%%%%%%%%%%%%%%%%%%%%%%%% Proof %%%%%%%%%%%%%%%%%%%%%%%%%%%%%%%%%%%%%
%
\begin{proof} See Lemma 13.25 in \cite{M}.
\end{proof}
Here the space $C^{k,\alpha}_\mu(\mathbb{R}^n\backslash B_r(0))$ is the collection of functions $u$ that are locally in $C^{k,\alpha}(\mathbb{R}^n\backslash B_r(0))$ and for which the norm
$$\|u\|_{C^{k,\alpha}_\mu(\mathbb{R}^n\backslash B_r(0))}:= \sup_{\sigma\geq r}\sigma^{-\mu}\|u\|_{(k,\alpha),[\sigma,2\sigma]}$$
is finite.
%
%%%%%%%%%%%%%%%%%%%%%%%%%%%%%%%%% Remark %%%%%%%%%%%%%%%%%%%%%%%%%%%%%%%%%%%%%
%
\begin{remark}\label{remark01}
In this case, it is very useful to know an explicit expression for $\mathcal{Q}_1$, since it has a component in the space spanned by the coordinate functions and this will be important to control this space in Section \ref{sec05}. Hence, if we write $\varphi=\displaystyle\sum_{i=2}^\infty \varphi_i$, with $\varphi$ belonging to the eigenspace associated to the eigenvalue $i(i+n-2)$, then
$$\mathcal{Q}_1(\varphi)(x)=\sum_{i=1}^\infty |x|^{2-n-j}\varphi_i.$$
\end{remark}
Now, define 
\begin{equation}\label{eq75}
\mathcal{Q}_r(\varphi_{r})(x):=\mathcal{Q}_1(\varphi)(r^{-1}x),
\end{equation}
where $\varphi_{r}(x):=\varphi(r^{-1}x)$. From Proposition \ref{lem4}, we deduce that
$$\left\{\begin{array}{lcl}
\Delta \mathcal{Q}_r(\varphi_{r}) =0 & \mbox{ in } & \mathbb{R}^n\backslash B_r(0)\\
\mathcal{Q}_r(\varphi_{r}) =\varphi_{r} & \mbox{ on } & \partial B_r(0)
\end{array}\right.$$
and
\begin{equation}\label{eq74}
\|\mathcal{Q}_r(\varphi_{r})\|_{C^{2,\alpha}_{1-n}(\mathbb{R}^n\backslash B_r(0))}\leq Cr^{n-1}\|\varphi_{r}\|_{(2,\alpha),r},
\end{equation}
where $C>0$ is a constant that does not depend on $r$.
%
%%%%%%%%%%%%%%%%%%%%%%%%%%%%%%%%% Section %%%%%%%%%%%%%%%%%%%%%%%%%%%%%%%%%%%%%%
%
\subsection{Conformal normal coordinates}\label{sec11}
Since our problem is conformally invariant, in Section \ref{sec02} we will work in conformal normal coordinates. In this section we introduce some notation and an asymptotic expansion for the scalar curvature in conformal normal coordinates, which will be essential in the interior analysis of Section \ref{sec02}.
%
%%%%%%%%%%%%%%%%%%%%%%%%%%%%%%%%% Theorem %%%%%%%%%%%%%%%%%%%%%%%%%%%%%%%%%%%%%%
%
\begin{theorem}[Lee--Parker, \cite{LP}]\label{teo07}
Let $M^n$ be an $n-$dimensional Riemannian manifold and $P\in M$. For each $N\geq 2$ there is a conformal metric $g$ on $M$ such that
$$\det g_{ij}=1+O(r^N),$$
where $r=|x|$ in $g-$normal coordinates at P. In these coordinates, if $N\geq 5$, the scalar curvature of $g$ satisfies $R_g=O(r^2)$.
\end{theorem}

In conformal normal coordinates it is more convenient to work with the Taylor expansion of the metric. In such coordinates, we will always write
$$g_{ij}=\exp(h_{ij}),$$
where $h_{ij}$ is a symmetric two-tensor satisfying $h_{ij}(x)=O(|x|^2)$ and tr$h_{ij}(x)=O(|x|^N)$. Here $N$ is a large number.

In what follows, we write $\partial_i\partial_jh_{ij}$ instead of $\displaystyle\sum_{i,j=1}^n \partial_i\partial_jh_{ij}$.
%
%%%%%%%%%%%%%%%%%%%%%%%%%%%%%%%%% Lemma %%%%%%%%%%%%%%%%%%%%%%%%%%%%%%%%%%%%%%
%
\begin{lemma}\label{lem11}
The functions $h_{ij}$ satisfy the following properties:
\begin{enumerate}
%\item[a)] $h_{ij}(x)x^j=0$;
\item[a)] $\displaystyle\int_{\mathbb{S}_r^{n-1}}\partial_i\partial_jh_{ij}= O(r^{N'})$;
\item[b)] $\displaystyle\int_{\mathbb{S}_r^{n-1}} x_k\partial_i\partial_jh_{ij}= O(r^{N'})$ for every $1\leq k\leq n$,
\end{enumerate}
where $N'$ is as big as we want.
\end{lemma}

This lemma plays a central role in our argument for $n\geq 8$ in Section \ref{sec02}.

Using this notation we obtain the following proposition whose proof can be found in \cite{B1} and \cite{KMS}.
%
%%%%%%%%%%%%%%%%%%%%%%%%%%%%%%%%% Proposition %%%%%%%%%%%%%%%%%%%%%%%%%%%%%%%%%%
%
\begin{proping}\label{propo08}
There exists a constant $C>0$ such that
$$|R_g-\partial_i\partial_jh_{ij}|\leq C\sum_{|\alpha|=2}^d \sum_{i,j}|h_{ij\alpha}|^2|x|^{2|\alpha|-2}+C|x|^{n-3},$$
if $|x|\leq r\leq 1$, where 
$$h_{ij}(x)=\sum_{2\leq |\alpha|\leq n-4}h_{ij\alpha}x^\alpha+O(|x|^{n-3})$$
and $C$ depends only on $n$ and $|h|_{C^N(B_r(0))}$.
\end{proping}

%
%%%%%%%%%%%%%%%%%%%%%%%%%%%%%%%%% CHAPTER %%%%%%%%%%%%%%%%%%%%%%%%%%%%%%%%%%%%%%
%

\section{Interior Analysis}\label{sec02}
%
%%%%%%%%%%%%%%%%%%%%%%%%%%%%%%%%% Section %%%%%%%%%%%%%%%%%%%%%%%%%%%%%%%%%%%%%%
%
%\subsection{Introduction}
Now that we have a right inverse for the operator $L_{\varepsilon,R,a}$ and a Poisson operator associated to the Laplacian $\Delta$, we are ready to show the existence of solutions with prescribed boundary data for the equation $H_{{g_0}}(v)=0$ in a small punctured ball $B_r(p)\backslash\{p\}\subset M$. The point $p$ is a nonremovable singularity, that is, $u$ blows-up at $p$. In fact, the hypothesis on the Weyl tensor is fundamental  for our construction if $n\geq 6$. But, if $3\leq n\leq 5$ we do not need any additional hypothesis on the point $p$. We do not know whether it is possible to show the Main Theorem assuming the Weyl tensor vanishes up to order less than $\left[\frac{n-6}{2}\right]$. This should be an interesting question.

First we will explain how to use the assumption on the Weyl tensor to reduce the problem to a problem of finding a fixed point of a map, (\ref{eq36}) and (\ref{eq59}). After that, we will show that these maps has a fixed point for suitable parameters.
%
%%%%%%%%%%%%%%%%%%%%%%%%%%%%%%%%% Section %%%%%%%%%%%%%%%%%%%%%%%%%%%%%%%%%%%%%%
%
\subsection{Analysis in $B_r(p)\backslash\{p\}\subset M$}\label{sec07}
Throughout the rest of this work $d=\left[\frac{n-2}{2}\right]$, and $g$ will be a smooth conformal metric to $g_0$ in $M$ given by Theorem \ref{teo07}, with $N$ a large number. Hence, by the proof of Theorem \ref{teo07} in \cite{LP}, we can find some smooth function $\mathcal{F}\in C^\infty(M)$ such that $g=\mathcal{F}^{\frac{4}{n-2}}g_0$ and $\mathcal{F}(x)=1+O(|x|^2)$ in $g-$normal coordinates at $p$. In this section we will work in these coordinates around $p$, in the ball $B_{r_1}(p)$ with $0<r_1\leq 1$ fixed. 

Recall that $(M,g_0)$ is an $n-$dimensional compact Riemannian manifold with $R_{g_0}=n(n-1)$, $n\geq 3$, and the Weyl tensor $W_{g_0}$ at $p$ satisfies the condition
%
%%%%%%%%%%%%%%%%%%%%%%%%%%%% Equation %%%%%%%%%%%%%%%%%%%%%%%%%%%%%%%%%%%%%%%
%
\begin{equation}\label{eq53}
\nabla^{\mathit{l}}W_{g_0}(p)=0,\;\mathit{l}=0,1,\ldots,d-2.
\end{equation}
Since the Weyl tensor is conformally invariant, it follows that $W_g$, the Weyl tensor of the metric $g$, satisfies the same condition. Note that if $3\leq n\leq 5$ then the condition on $W_g$ does not exist.

From Theorem \ref{teo07} the scalar curvature satisfies $R_g=O(|x|^2)$, but for $n\geq 8$ we can improve this decay, using the assumption of the Weyl tensor. This assumption implies $h_{ij}=O(|x|^{d+1})$ (see \cite{B}) and it follows from Proposition \ref{propo08} that 
\begin{equation}\label{eq70}
R_g=\partial_i\partial_j h_{ij}+O(|x|^{n-3}).
\end{equation}
We conclude that $R_g=O(|x|^{d-1})$. On the other hand, for $n=6$ and $7$ we have $d=2$ and in this case, we will consider $R_g=O(|x|^2)$, given directly by Theorem \ref{teo07}.

The main goal of this section is to solve the PDE
%
%%%%%%%%%%%%%%%%%%%%%%%%%%%%%%%%% Equation %%%%%%%%%%%%%%%%%%%%%%%%%%%%%%%%%%%%%
%
\begin{equation}\label{eq05}
H_g(u_{\varepsilon,R,a}+v)=0
\end{equation}
in $B_r(0)\backslash\{0\}\subset\mathbb{R}^n$ for some $0<r\leq r_1$, $\varepsilon>0$, $R>0$ and $a\in\mathbb{R}^n$, with $u_{\varepsilon,R,a}+v>0$ and prescribed
Dirichlet data, where the operator $H_g$ is defined in (\ref{eq50}) and $u_{\varepsilon,R,a}$ in (\ref{eq52}).

To solve this equation, we will use the method used by Byde and others, the fixed point method on Banach spaces. In \cite {AB}, Byde solves an equation like this assuming that $g$ is conformally flat in a neighborhood of $p$, and thus he uses directly the right inverse of $L_{\varepsilon,R}$ given by Corollary \ref{cor01}, to reduce the problem to a problem of fixed point. The main difference here is that we work with metrics not necessarily conformally flat, so we need to rearrange the terms of the equation (\ref{eq05}) in such a way that we can apply the right inverse of $L_{\varepsilon,R,a}$. 

For each $\phi\in\pi''(C^{2,\alpha}(\mathbb{S}_r^{n-1}))$ define $v_\phi:=\mathcal{P}_r(\phi)\in\pi''(C^{2,\alpha}_2(B_r(0)\backslash\{0\}))$ as in Proposition \ref{propo02}. It is easy to see that the equation (\ref{eq05}) is equivalent to
%
%%%%%%%%%%%%%%%%%%%%%%%%%%%%%%%%% Equation %%%%%%%%%%%%%%%%%%%%%%%%%%%%%%%%%%%%%
%
\begin{equation}\label{eq10}
\begin{array}{lcl}
L_{\varepsilon,R,a}(v) & = & (\Delta-\Delta_g)(u_{\varepsilon,R,a}+v_\phi+v) +\displaystyle\frac{n-2}{4(n-1)}R_g(u_{\varepsilon,R,a}+v_\phi+v)\\
& - & Q_{\varepsilon,R,a}(v_\phi+v)- \displaystyle\frac{n(n+2)}{4}u_{\varepsilon,R,a}^{\frac{4}{n-2}}v_\phi,
\end{array}
\end{equation}
since $u_{\varepsilon,R,a}$ solves the equation (\ref{eq01}). Here $L_{\varepsilon,R,a}$ is defined as in (\ref{eq38}),
%
%%%%%%%%%%%%%%%%%%%%%%%%%%%%%%%%% Equation %%%%%%%%%%%%%%%%%%%%%%%%%%%%%%%%%%%%%
%
\begin{equation}\label{eq39}
Q_{\varepsilon,R,a}(v):=Q^{u_{\varepsilon,R,a}}(v)
\end{equation}
and $Q^{u_{\varepsilon,R,a}}$ is defined in (\ref{eq44}).
\begin{remark}\label{rem01}
Throughout this work we will consider $|a|r_\varepsilon\leq 1/2$ with $r_\varepsilon=\varepsilon^s$, s restricted to $(d+1-\delta_1)^{-1}<s<4(d-2+3n/2)^{-1}$ and $\delta_1\in(0,(8n-16)^{-1})$.
\end{remark}

From this and (\ref{eq52}) it follows that there are constants $C_1>0$ and $C_2>0$ that do not depend on $\varepsilon$, $R$ and $a$, so that
\begin{equation}\label{eq105}
C_1\varepsilon|x|^{\frac{2-n}{2}}\leq u_{\varepsilon,R,a}(x)\leq C_2|x|^{\frac{2-n}{2}},
\end{equation}
for every $x$ in $B_{r_\varepsilon}(0)\backslash\{0\}$.

These restrictions are made to ensure some conditions that we need in the next lemma and in Section \ref{sec05}.
%
%%%%%%%%%%%%%%%%%%%%%%%%%%%%%%%%% Lemma %%%%%%%%%%%%%%%%%%%%%%%%%%%%%%%%%%%%%
%
\begin{lemma}\label{lem01}
Let $\mu\in (1,3/2)$. There exists $\varepsilon_0\in(0,1)$ such that for each $\varepsilon\in(0,\varepsilon_0)$, $a\in\mathbb{R}^n$ with $|a|r_\varepsilon\leq 1$, and for all $v_i\in C_\mu^{2,\alpha}(B_{r_\varepsilon}(0)\backslash\{0\})$, $i=0,1$, and $w\in C_{2+d-\frac{n}{2}}^{2,\alpha}(B_{r_\varepsilon}(0)\backslash\{0\})$ with  $\|v_i\|_{(2,\alpha),\mu,r_\varepsilon}\leq c r_\varepsilon^{2+d-\mu-\frac{n}{2}-\delta_1}$ and $\|w\|_{(2,\alpha),2+d-\frac{n}{2},r_\varepsilon}\leq c$, for some constant $c>0$ independent of $\varepsilon$, we have that $Q_{\varepsilon,R,a}$ given by (\ref{eq39}) satisfies the inequalities
%
%%%%%%%%%%%%%%%%%%%%%%%%%%%%%%%%% Equation %%%%%%%%%%%%%%%%%%%%%%%%%%%%%%%%%%%%%
%
\begin{equation*}\label{eq55}
\begin{array}{c}
\|Q_{\varepsilon,R,a}(w+v_1)- Q_{\varepsilon,R,a}(w+v_0)\|_{(0,\alpha),\mu-2,r_{\varepsilon}}\leq \\
\leq C\varepsilon^{\lambda_n}r_\varepsilon^{d+1} \|v_1-v_0\|_{(2,\alpha),\mu,r_{\varepsilon}} \left(\|w\|_{(2,\alpha),2+d-\frac{n}{2},r_{\varepsilon}}+ \|v_1\|_{(2,\alpha),\mu,r_{\varepsilon}}+ \|v_0\|_{(2,\alpha),\mu,r_{\varepsilon}}\right),
\end{array}
\end{equation*}
and
%
%%%%%%%%%%%%%%%%%%%%%%%%%%%%%%%%% Equation %%%%%%%%%%%%%%%%%%%%%%%%%%%%%%%%%%%%%
%
\begin{equation*}\label{eq56}
\|Q_{\varepsilon,R,a}(w)\|_{(0,\alpha),\mu-2,r_{\varepsilon}}\leq C\varepsilon^{\lambda_n}r_\varepsilon^{3+2d-\frac{n}{2}-\mu} \|w\|_{(2,\alpha),2+d-\frac{n}{2},r_{\varepsilon}}^2.
\end{equation*}
Here $\lambda_n=0$ for $3\leq n\leq 6$, $\lambda_n=\frac{6-n}{n-2}$ for $n\geq 7$, and the constant $C>0$ does not depend on $\varepsilon$, $R$ and $a$.
\end{lemma}
\begin{proof} By the hypothesis, we conclude that
$$|v_i(x)|\leq cr_\varepsilon^{2+d-\frac{n}{2}-\delta_1}$$
and
$$|w(x)|\leq cr_\varepsilon^{2+d-\frac{n}{2}}$$
for all $x\in B_{r_\varepsilon}(0)\backslash\{0\}$. Using (\ref{eq105}), we get
$$u_{\varepsilon,R,a}(x)+w+v_i(x) \geq \varepsilon|x|^{\frac{2-n}{2}}(C_1- c(|x|r_\varepsilon^{-1})^{\frac{n-2}{2}}\varepsilon^{s(d+1-\delta_1)-1}),$$
with $s(d+1-\delta_1)-1>0$, since $s>(d+1-\delta_1)^{-1}$. Therefore,
%
%%%%%%%%%%%%%%%%%%%%%%%%%%%%%%%%% Equation %%%%%%%%%%%%%%%%%%%%%%%%%%%%%%%%%%%%%
%
\begin{equation}\label{eq09}
0<C_3\varepsilon|x|^{\frac{2-n}{2}}\leq u_{\varepsilon,R,a}(x)+w(x)+v_i(x)\leq C_4|x|^{\frac{2-n}{2}}
\end{equation}
for small enough $\varepsilon>0$, since $|x|\leq r_\varepsilon$. Thus, by (\ref{eq44}), we can write
$$Q_{\varepsilon,R,a}(w+v_1)-Q_{\varepsilon,R,a}(w+v_0) =\displaystyle\frac{n(n+2)}{n-2}(v_1-v_0)\int_0^1\int_0^1 (u_{\varepsilon,R,a}+sz_t)^{\frac{6-n}{n-2}}z_tdtds$$
and
$$Q_{\varepsilon,R,a}(w)=\frac{n(n+2)}{n-2}w^2\int_0^1\int_0^1 (u_{\varepsilon,R,a}+stw)^{\frac{6-n}{n-2}}tdtds,$$
where $z_t=w+tv_1+(1-t)v_0$. From this we obtain
$$\|Q_{\varepsilon,R,a}(w+v_1)-Q_{\varepsilon,R,a}(w+v_0) \|_{(0,\alpha),[\sigma,2\sigma]}\leq C\|v_1-v_0\|_{(0,\alpha),[\sigma,2\sigma]} \left(\|w\|_{(0,\alpha),[\sigma,2\sigma]}+\right.$$
$$\left.+\|v_1\|_{(0,\alpha),[\sigma,2\sigma]}+ \|v_0\|_{(0,\alpha),[\sigma,2\sigma]}\right)\displaystyle \max_{0\leq s,t\leq 1}\|(u_{\varepsilon,R,a}+sz_t)^{\frac{6-n}{n-2}} \|_{(0,\alpha),[\sigma,2\sigma]}$$
and
$$\|Q_{\varepsilon,R,a}(w)\|_{(0,\alpha),[\sigma,2\sigma]}\leq C\|w\|_{(0,\alpha),[\sigma,2\sigma]}^2\max_{0\leq s,t\leq 1}\|(u_{\varepsilon,R,a}+stw)^{\frac{6-n}{n-2}}\|_{(0,\alpha),[\sigma,2\sigma]}.$$

From (\ref{eq09}) we deduce that
$$|(u_{\varepsilon,R,a}+sz_t)^{\frac{6-n}{n-2}}(x)|\leq C\varepsilon^{\lambda_n}|x|^{\frac{n-6}{2}}$$
and
$$|(u_{\varepsilon,R,a}+stw)^{\frac{6-n}{n-2}}(x)|\leq C\varepsilon^{\lambda_n}|x|^{\frac{n-6}{2}},$$
for some constant $C>0$ independent of $\varepsilon$, $a$ and $R$.

The estimate for the full H\"older norm is similar. Hence, we conclude that
$$\max_{0\leq s,t\leq 1}\|(u_{\varepsilon,R,a}+ sz_t)^{\frac{6-n}{n-2}}\| _{(0,\alpha),[\sigma,2\sigma]}\leq C\varepsilon^{\lambda_n}\sigma^{\frac{n-6}{2}}$$
and
$$\max_{0\leq s,t\leq 1}\|(u_{\varepsilon,R,a}+ stw)^{\frac{6-n}{n-2}}\| _{(0,\alpha),[\sigma,2\sigma]}\leq C\varepsilon^{\lambda_n}\sigma^{\frac{n-6}{2}}.$$

Therefore,
$$\sigma^{2-\mu}\|Q_{\varepsilon,R,a}(w+v_1)-Q_{\varepsilon,R,a}(w+v_0) \|_{(0,\alpha),[\sigma,2\sigma]}\leq$$
$$\leq C\varepsilon^{\lambda_n}r_\varepsilon^{d+1} \|v_1-v_0\|_{(2,\alpha),\mu,r_\varepsilon} (\|w\|_{(2,\alpha),2+d-\frac{n}{2},r_\varepsilon} +\|v_1\|_{(2,\alpha),\mu,r_\varepsilon}+ \|v_0\|_{(2,\alpha),\mu,r_\varepsilon})$$
and
$$ \sigma^{2-\mu}\|Q_{\varepsilon,R,a}(w)\|_{(0,\alpha),[\sigma,2\sigma]} \leq C\varepsilon^{\lambda_n} r_\varepsilon^{3+2d-\frac{n}{2}-\mu} \|w\|^2_{(2,\alpha),2+d-\frac{n}{2},r_\varepsilon},$$
since $1<\mu<3/2$ implies $2+d-n/2<\mu$ and $3+2d-n/2-\mu>0$.
\end{proof}

Now to use the right inverse of $L_{\varepsilon,R,a}$, given by $G_{\varepsilon,R,r_{\varepsilon},a}$, all terms of the right hand side of the equation (\ref{eq10}) have to belong to the domain of $G_{\varepsilon,R,r_{\varepsilon},a}$. But this does not happen with the term $R_gu_{\varepsilon,R,a}$ if $n\geq 8$, since $R_g=O(|x|^{d-1})$ implies $R_gu_{\varepsilon,R,a}=O(|x|^{d-\frac{n}{2}})$ and so $R_gu_{\varepsilon,R,a}\not\in C_{\mu-2}^{0,\alpha}(B_{r_\varepsilon} (0)\backslash\{0\})$ for every $\mu>1$. However, when $3\leq n\leq 7$ we get the following lemma:
%
%%%%%%%%%%%%%%%%%%%%%%%%%%%%%%%%% Lemma %%%%%%%%%%%%%%%%%%%%%%%%%%%%%%%%%%%%%
%
\begin{lemma}\label{lem05}
Let $3\leq n\leq 7$, $\mu\in(1,3/2)$, $\kappa>0$ and $c>0$ be fixed constants. There exists $\varepsilon_0\in(0,1)$ such that for each $\varepsilon\in(0,\varepsilon_0)$, for all $v\in C_\mu^{2,\alpha}(B_{r_\varepsilon}(0)\backslash\{0\})$ and $\phi\in\pi''(C^{2,\alpha}(\mathbb{S}_{r_\varepsilon}^{n-1}))$ with $\|v\|_{(2,\alpha),\mu,r_\varepsilon}\leq cr_\varepsilon^{2+d-\mu-\frac{n}{2}-\delta_1}$ and $\|\phi\|_{(2,\alpha),r_\varepsilon}\leq \kappa r_\varepsilon^{2+d-\frac{n}{2}-\delta_1}$, we have that the right hand side of (\ref{eq10}) belongs to $C_{\mu-2}^{0,\alpha}(B_{r_\varepsilon}(0)\backslash\{0\})$.
\end{lemma}
%
%%%%%%%%%%%%%%%%%%%%%%%%%%%%%%%%% Proof %%%%%%%%%%%%%%%%%%%%%%%%%%%%%%%%%%%%%
%
\begin{proof} Initially, note that by (\ref{eq73}) we obtain $$\|v_\phi+v\|_{(2,\alpha),\mu,r_\varepsilon}\leq (c+\kappa) r_\varepsilon^{2+d-\mu-\frac{n}{2}-\delta_1},$$
by Lemma \ref{lem01} we get that $Q_{\varepsilon,R,a}(v_\phi+v)\in C_{\mu-2}^{0,\alpha}(B_{r_\varepsilon}(0) \backslash\{0\})$.

Now it is enough to show that the other terms have the decay $O(|x|^{\mu-2})$. 

Using the expansion (\ref{eq37}), it follows that
$$(\Delta-\Delta_g)u_{\varepsilon,R,a}= (\Delta-\Delta_g)u_{\varepsilon,R}+(\Delta-\Delta_g)(u_{\varepsilon,R,a}- u_{\varepsilon,R}),$$
with $u_{\varepsilon,R,a}- u_{\varepsilon,R}=O''(|a||x|^{\frac{4-n}{2}})$. Moreover, since in conformal normal  coordinates $\Delta_g=\Delta+O(|x|^N)$ when applied to functions that depend only on $|x|$, where $N$ can be any big number (see proof of Theorem 3.5 in \cite{SY}, for example), we get
$$(\Delta-\Delta_g)u_{\varepsilon,R}=O(|x|^{N'}),$$
where $N'$ is big for $N$ big.

Since $g_{ij}=\delta_{ij}+O(|x|^{d+1})$, we get
$$(\Delta-\Delta_g)(u_{\varepsilon,R,a}- u_{\varepsilon,R})=O(|x|^{d+\frac{2-n}{2}})= O(|x|^{\mu-2})$$
when $\mu\leq 3+d-3/2$.

Since $v_\phi=O(|x|^2)$, $g_{ij}=\delta_{ij}+O(|x|^{d+1})$, $R_g=O(|x|^2)$, using (\ref{eq105}) we get the same decay for the remaining terms. Hence the assertion follows.
\end{proof}
Now this lemma allows us to use the map $G_{\varepsilon,R,r_\varepsilon,a}$. Let $\mu\in(1,3/2)$ and $c>0$ be fixed constants. To solve the equation (\ref{eq05}) we need to show that the map $\mathcal{N}_\varepsilon(R,a,\phi,\cdot): \mathcal{B}_{\varepsilon,c,\delta_1}\rightarrow C_{\mu}^{2,\alpha}(B_{r_\varepsilon}(0)\backslash\{0\})$ has a fixed point for suitable parameters $\varepsilon$, $R$, $a$ and $\phi$, where $\mathcal{B}_{\varepsilon,c,\delta_1}$ is the ball in $C_{\mu}^{2,\alpha}(B_{r_\varepsilon}(0)\backslash\{0\})$ of radius $cr_\varepsilon^{2+d-\mu-\frac{n}{2}-\delta_1}$ and $\mathcal{N}_\varepsilon(R,a,\phi,\cdot)$ is defined by
%
%%%%%%%%%%%%%%%%%%%%%%%%%%%%%%%%% Equation %%%%%%%%%%%%%%%%%%%%%%%%%%%%%%%%%%%%%
%
\begin{equation}\label{eq36}
\begin{array}{c}
\mathcal{N}_\varepsilon(R,a,\phi,v) = G_{\varepsilon,R,r,a}\left( (\Delta-\Delta_g) v +\displaystyle\frac{n-2}{4(n-1)}R_gv- Q_{\varepsilon,R,a}(v_{\phi}+v)\right.\\
+ (\Delta-\Delta_g) (u_{\varepsilon,R,a}+v_\phi)+ \displaystyle\frac{n-2}{4(n-1)}R_g(u_{\varepsilon,R,a}+v_\phi)- \left.\displaystyle\frac{n(n+2)}{4}u_{\varepsilon,R,a}^{\frac{4}{n-2}} v_\phi\right).
\end{array}
\end{equation}

Let us now consider $n\geq 8$. Since $R_g=O(|x|^{d-1})$, we have $R_gu_{\varepsilon,R,a}=O(|x|^{d-\frac{n}{2}})$, and this implies that $R_gu_{\varepsilon,R,a}\not\in C_{\mu-2}^{0,\alpha}(B_{r_\varepsilon} (0)\backslash\{0\})$ for $\mu>1$. Hence we cannot use $G_{\varepsilon,R,{r_\varepsilon},a}$ directly. To overcome this difficulty we will consider the expansion (\ref{eq37}), the expansion (\ref{eq70}) and use the fact that $\partial_i\partial_jh_{ij}$ is orthogonal to $\{1,x_1,\ldots,x_n\}$ modulo a term of order $O(|x|^{N''})$ with $N''$ as big as we want (see Lemma \ref{lem11}.) 

It follows from this fact and Corollary \ref{cor03}, that there exists $w_{\varepsilon,R}\in C_{2+d-\frac{n}{2}}^{2,\alpha}(B_{r_\varepsilon} (0)\backslash\{0\})$ such that
%
%%%%%%%%%%%%%%%%%%%%%%%%%%%%%%%%% Equation %%%%%%%%%%%%%%%%%%%%%%%%%%%%%%%%%%%%%
%
\begin{equation}\label{eq35}
L_{\varepsilon,R}(w_{\varepsilon,R})=\frac{n-2}{4(n-1)} \pi''(\partial_i\partial_j h_{ij})u_{\varepsilon,R}.
\end{equation}
This is because $u_{\varepsilon,R}$ depends only on $|x|$.

Again by Corollary \ref{cor03}
%
%%%%%%%%%%%%%%%%%%%%%%%%%%%%%%%%% Equation %%%%%%%%%%%%%%%%%%%%%%%%%%%%%%%%%%%%%
%
\begin{equation}\label{eq57}
\|w_{\varepsilon,R}\|_{(2,\alpha),2+d-\frac{n}{2},r_\varepsilon} \leq c \|\pi''(\partial_i\partial_j h_{ij}) u_{\varepsilon,R} \|_{(0,\alpha),d-\frac{n}{2},r_\varepsilon}\leq c,
\end{equation}
for some constant $c>0$ that does not depend on $\varepsilon$ and $R$, since $\partial_i\partial_jh_{ij}u_{\varepsilon,R}=O(|x|^{d-\frac{n}{2}})$.

Considering the expansion (\ref{eq37}) and substituting $v$ for $w_{\varepsilon,R}+v$ in the equation (\ref{eq10}), we obtain
%
%%%%%%%%%%%%%%%%%%%%%%%%%%%%%%%%% Equation %%%%%%%%%%%%%%%%%%%%%%%%%%%%%%%%%%%%%
%
\begin{equation}\label{eq07}
\begin{array}{c}
\hspace{-1cm} L_{\varepsilon,R,a}(v) =  (\Delta-\Delta_g)(u_{\varepsilon,R,a}+w_{\varepsilon,R}+v_\phi+v)+ \displaystyle\frac{n-2}{4(n-1)}R_g(w_{\varepsilon,R}+v_\phi+v)\\
\vspace{-0,2cm}\\
\hspace{-1cm}- Q_{\varepsilon,R,a}(w_{\varepsilon,R}+v_\phi+v) + \displaystyle\frac{n-2}{4(n-1)}\partial_i\partial_jh_{ij}(u_{\varepsilon,R,a}- u_{\varepsilon,R})\\
\vspace{-0,2cm}\\
\hspace{1cm}+ \displaystyle \frac{n-2}{4(n-1)} (R_g-\partial_i\partial_jh_{ij})u_{\varepsilon,R,a} + \displaystyle\frac{n(n+2)}{2}(u_{\varepsilon,R}^{\frac{4}{n-2}}- u_{\varepsilon,R,a}^{\frac{4}{n-2}})w_{\varepsilon,R}\\
\vspace{-0,2cm}
 - \displaystyle\frac{n(n+2)}{4}u_{\varepsilon,R,a}^{\frac{4}{n-2}}v_\phi + \displaystyle \frac{n-2}{4(n-1)}\overline{h}u_{\varepsilon,R}
\end{array}
\end{equation}
where $R_g-\partial_i\partial_jh_{ij}=O(|x|^{n-3})$, $u_{\varepsilon,R,a}-u_{\varepsilon,R}=O(|a||x|^{\frac{4-n}{2}})$, $u_{\varepsilon,R,a}^{\frac{4}{n-2}}- u_{\varepsilon,R}^{\frac{4}{n-2}}=O(|a||x|^{-1})$ by the proof of Corollary \ref{cor01}, and $\overline{h}=\partial_i\partial_j h_{ij}- \pi''(\partial_i\partial_j h_{ij})=O(|x|^{N''})$
with $N''$ large. Hence we obtain the following lemma
%
%%%%%%%%%%%%%%%%%%%%%%%%%%%%%%%%% Lemma %%%%%%%%%%%%%%%%%%%%%%%%%%%%%%%%%%%%%
%
\begin{lemma}\label{lem09}
Let $n\geq 8$, $\mu\in(1,3/2)$, $\kappa>0$ and $c>0$ be fixed constants. There exists $\varepsilon_0\in(0,1)$ such that for each $\varepsilon\in(0,1)$, for all $v\in C_\mu^{2,\alpha}(B_{r_\varepsilon}(0)\backslash\{0\})$ and $\phi\in \pi''(C^{2,\alpha}(\mathbb{S}_{r_\varepsilon}^{n-1}))$ with $\|v\|_{(2,\alpha),\mu,r_\varepsilon}\leq c{r_\varepsilon}^{2+d-\mu-\frac{n}{2}-\delta_1}$ and $\|\phi\|_{(2,\alpha),r_\varepsilon}\leq \kappa r_\varepsilon^{2+d-\frac{n}{2}-\delta_1}$, we have that the right hand side of (\ref{eq07}) belongs to $C_{\mu-2}^{0,\alpha}(B_{r_\varepsilon} (0)\backslash\{0\})$.
\end{lemma}
\begin{proof}

As before in Lemma \ref{lem05}, we obtain 
$$Q_{\varepsilon,R,a} (w_{\varepsilon,R}+v_\phi+v)\in C_{\mu-2}^{0,\alpha}(B_{r_\varepsilon}(0)\backslash\{0\})$$
and
$$(\Delta-\Delta_g)u_{\varepsilon,R,a}=O(|x|^{1+d-\frac{n}{2}})= O(|x|^{\mu-2}).$$

Therefore, the assertion follows, since for the remaining terms we obtain the same estimate.
\end{proof}
%
%%%%%%%%%%%%%%%%%%%%%%%%%%%%%%%%% Remark %%%%%%%%%%%%%%%%%%%%%%%%%%%%%%%%%%%%%
%

Let $\mu\in(1,3/2)$ and $c>0$ be fixed constants. It is enough to show that the map $\mathcal{N}_\varepsilon(R,a,\phi,\cdot): \mathcal{B}_{\varepsilon,c,\delta_1}\rightarrow C_{\mu}^{2,\alpha}(B_r(0)\backslash\{0\})$ has a fixed point for suitable parameters $\varepsilon$, $R$, $a$ and $\phi$, where $\mathcal{B}_{\varepsilon,c,\delta_1}$ is the ball in $C_{\mu}^{2,\alpha}(B_r(0)\backslash\{0\})$ of radius $cr_\varepsilon^{2+d-\mu-\frac{n}{2}-\delta_1}$ and $\mathcal{N}_\varepsilon(R,a,\phi,\cdot)$ is defined by
%
%%%%%%%%%%%%%%%%%%%%%%%%%%%%%%%%% Equation %%%%%%%%%%%%%%%%%%%%%%%%%%%%%%%%%%%%%
%

\begin{equation}\label{eq59}
\begin{array}{c}
\mathcal{N}_\varepsilon(R,a,\phi,v) = \displaystyle G_{\varepsilon,R,r,a}\left( (\Delta-\Delta_g) v + \displaystyle\frac{n-2}{4(n-1)}R_gv - Q_{\varepsilon,R,a}(v_{\phi}+w_{\varepsilon,R}+v)\right.\\
\vspace{-0,2cm}\\
+ (\Delta-\Delta_g) (u_{\varepsilon,R,a}+v_{\phi}+w_{\varepsilon,R})+ \displaystyle\frac{n-2}{4(n-1)}R_g(v_{\phi}+w_{\varepsilon,R})\\
\vspace{-0,2cm}\\
+ \displaystyle \frac{n-2}{4(n-1)}(R_g-\partial_i\partial_jh_{ij})u_{\varepsilon,R,a} - \displaystyle\frac{n(n+2)}{4} u_{\varepsilon,R,a}^{\frac{4}{n-2}}v_\phi\\
\vspace{-0,2cm}\\
\hspace{1cm}+\displaystyle\frac{n(n+2)}{2}(u_{\varepsilon,R}^{\frac{4}{n-2}}- u_{\varepsilon,R,a}^{\frac{4}{n-2}})w_{\varepsilon,R} + \displaystyle\frac{n-2}{4(n-1)}\overline{h}u_{\varepsilon,R}\\
\\
\hspace{1cm}+ \left.\displaystyle\frac{n-2}{4(n-1)}\partial_i\partial_jh_{ij}(u_{\varepsilon,R,a}- u_{\varepsilon,R})\right).
\end{array}
\end{equation}

In fact, we will show that the map $\mathcal{N}_\varepsilon(R,a,\phi,\cdot)$ is a contraction for small enough $\varepsilon>0$, and as a consequence of this we will get that the fixed point is continuous with respect to the parameters $\varepsilon$, $R$, $a$ and $\phi$. 
\begin{remark}\label{remark02}
The vanishing of the Weyl tensor up to the order $d-2$ is sharp, in the following sense: if $\nabla^lW_g(0)=0,$ $l=0,1,\ldots, d-3$, then for $n\geq 6$, $g_{ij}=\delta_{ij}+O(|x|^d)$ and
$$(\Delta-\Delta_g)u_{\varepsilon,R,a}=O(|x|^{d-\frac{n}{2}}).$$
This implies $(\Delta-\Delta_g)u_{\varepsilon,R,a}\not\in C^{0,\alpha}_{\mu-2}(B_{r_\varepsilon}(0)\backslash\{0\})$, with $\mu>1$.
\end{remark}
The next lemma will be very useful to show Proposition \ref{propo03}. To prove it use the Laplacian in local coordinates.
%
%%%%%%%%%%%%%%%%%%%%%%%%%%%%%%%%% Lemma %%%%%%%%%%%%%%%%%%%%%%%%%%%%%%%%%%%%%
%
\begin{lemma}\label{lem02}
Let $g$ be a metric in $B_r(0)\subset\mathbb{R}^n$ in conformal normal coordinates with the Weyl tensor satisfying the assumption (\ref{eq53}). Then, for all $\mu\in\mathbb{R}$ and $v\in C_{\mu}^{2,\alpha}(B_r(0)\backslash\{0\})$ there is a constant $c>0$ that does not depend on $r$ and $\mu$ such that
$$\|(\Delta-\Delta_g)(v)\|_{(0,\alpha),\mu-2,r}\leq cr^{d+1} \|v\|_{(2,\alpha),\mu,r}.$$
\end{lemma}
%
%%%%%%%%%%%%%%%%%%%%%%%%%%%%%%%%% Section %%%%%%%%%%%%%%%%%%%%%%%%%%%%%%%%%%%%%
%
\subsection{Complete Delaunay-type ends}\label{sec12}
The previous discussion tells us that to solve the equation (\ref{eq05}) with  prescribed boundary data on a small sphere centered at $0$, we have to show that the map $\mathcal{N}_\varepsilon(R,a,\phi,\cdot)$, defined in (\ref{eq36}) for $3\leq n\leq 7$ and in (\ref{eq59}) for $n\geq 8$, has a fixed point. To do this, we will show that this map is a contraction using the fact that the right inverse $G_{\varepsilon,R,r_\varepsilon,a}$ of $L_{\varepsilon,R,a}$ in the punctured ball $B_{r_\varepsilon}(0)\backslash\{0\}$, given by Corollary \ref{cor01}, has norm bounded independently of $\varepsilon$, $R$, $a$ and $r_\varepsilon$.

Next  we will prove the main result of this section. This will solve the singular Yamabe problem locally.
\begin{remark}\label{remark03}
To ensure some estimates that we will need, from now on, we will consider $R^{\frac{2-n}{2}}=2(1+b)\varepsilon^{-1}$, with $|b|\leq 1/2$.
\end{remark}
%
%%%%%%%%%%%%%%%%%%%%%%%%%%%%%%%%% Proposition %%%%%%%%%%%%%%%%%%%%%%%%%%%%%%%%%%
%
\begin{proping}\label{propo03}
Let $\mu \in (1,5/4)$, $\tau>0$, $\kappa>0$ and $\delta_2>\delta_1$ be fixed constants. There exists a constant $\varepsilon_0\in(0,1)$ such that for each $\varepsilon\in(0,\varepsilon_0]$, $|b|\leq 1/2$, $a\in\mathbb{R}^n$ with $|a|r_\varepsilon^{1-\delta_2}\leq 1$, and $\phi\in\pi''(C^{2,\alpha}(\mathbb{S}_{r_{\varepsilon}}^{n-1}))$ with $\|\phi\|_{(2,\alpha),r_\varepsilon}\leq \kappa r_\varepsilon^{2+d-\frac{n}{2}-\delta_1}$, there exists a fixed point of the map $\mathcal{N}_\varepsilon(R,a,\phi,\cdot)$ in the ball of radius $\tau r_\varepsilon^{2+d-\mu-\frac{n}{2}}$ in $C_{\mu}^{2,\alpha}(B_{r_\varepsilon}(0)\backslash\{0\})$.
\end{proping}
%
%%%%%%%%%%%%%%%%%%%%%%%%%%%%%%%%% Proof %%%%%%%%%%%%%%%%%%%%%%%%%%%%%%%%%%%%%
%
\begin{proof} First note that $|a|r_\varepsilon\leq r_\varepsilon^{\delta_2}\rightarrow 0$ when $\varepsilon$ tends to zero. It follows from Corollary \ref{cor01}, Lemma \ref{lem05} and \ref{lem09} that the map $\mathcal{N}_\varepsilon(R,a,\phi,\cdot)$ is well defined in the ball of radius $\tau r_\varepsilon^{2+d-\mu-\frac{n}{2}}$ in $C_\mu^{2,\alpha}(B_{r_\varepsilon}(0)\backslash\{0\})$ for small $\varepsilon>0$.

Following \cite{AB} we will show that
\begin{equation}\label{eq112}
\|\mathcal{N}_\varepsilon(R,a,\phi,0)\|_{(2,\alpha),\mu,r_\varepsilon}< \frac{1}{2}\tau r_\varepsilon^{2+d-\mu-\frac{n}{2}},
\end{equation}
and for all $v_i\in C_{\mu}^{2,\alpha}(B_{r_\varepsilon}(0)\backslash\{0\})$ with $\|v_i\|_{(2,\alpha),\mu,r_\varepsilon} \leq \tau r_\varepsilon^{2+d-\mu-\frac{n}{2}}$, $i=1,2$, we will have
\begin{equation}\label{eq113}
\|\mathcal{N}_\varepsilon(R,a,\phi,v_1)- \mathcal{N}_\varepsilon(R,a,\phi,v_2)\|_{(2,\alpha),\mu,r_\varepsilon} <\frac{1}{2} \|v_1-v_2\|_{(2,\alpha),\mu,r_\varepsilon}.
\end{equation}
It will follow from this that for all $v\in C_{\mu}^{2,\alpha}(B_{r_\varepsilon}(0)\backslash\{0\})$ in the ball of radius $\tau r_\varepsilon^{2+d-\mu-\frac{n}{2}}$ we will get
$$\|\mathcal{N}_\varepsilon(R,a,\phi,v)\|_{(2,\alpha),\mu,r_\varepsilon} \leq \|\mathcal{N}_\varepsilon(R,a,\phi,v)- \mathcal{N}_\varepsilon(R,a,\phi,0)\|_{(2,\alpha),\mu,r_\varepsilon}+ \|\mathcal{N}_\varepsilon(R,a,\phi,0)\|_{(2,\alpha),\mu,r_\varepsilon}.$$
Hence we conclude that the map $\mathcal{N}_\varepsilon(R,a,\phi,\cdot)$ will have a fixed point  belonging to the ball of radius $\tau r_\varepsilon^{2+d-\mu-\frac{n}{2}}$ in $C_\mu^{2,\alpha}(B_{r_\varepsilon}(0)\backslash\{0\})$.

Consider $3\leq n\leq 7$. 

Since $G_{\varepsilon,R,r_{\varepsilon},a}$ is bounded independently of $\varepsilon$, $R$ and $a$, it follows that
$$ \|\mathcal{N}_\varepsilon(R,a,\phi,0)\|_{(2,\alpha),\mu,r_{\varepsilon}} \leq c\left(\|(\Delta-\Delta_g)(u_{\varepsilon,R,a}+v_{\phi}) \|_{(0,\alpha),\mu-2,r_{\varepsilon}}\right.$$
$$\left.+ \|R_g(u_{\varepsilon,R,a}+v_{\phi})\|_{(0,\alpha),\mu-2,r_{\varepsilon}}+ \|Q_{\varepsilon,R,a}(v_{\phi})\|_{(0,\alpha),\mu-2,r_{\varepsilon}}+ \|u_{\varepsilon,R,a}^{\frac{4}{n-2}} v_{\phi}\|_{(0,\alpha),\mu-2,r_{\varepsilon}}\right),$$
where $c>0$ is a constant that does not depend on $\varepsilon$, $R$ and $a$.

Using local coordinates we obtain that
$$ \sigma^{2-\mu}\|(\Delta-\Delta_g)(u_{\varepsilon,R,a}- u_{\varepsilon,R})\|_{(0,\alpha),[\sigma,2\sigma]} \leq c\sigma^{1+d-\mu}\|u_{\varepsilon,R,a}- u_{\varepsilon,R}\|_{(2,\alpha),[\sigma,2\sigma]}\leq c|a|\sigma^{3+d-\mu-\frac{n}{2}},$$
since $u_{\varepsilon,R,a}=u_{\varepsilon,R}+O''(|a||x|^{\frac{4-n}{2}})$, by (\ref{eq37}). The condition $\mu<3/2$ implies
\begin{equation}\label{eq106}
\|(\Delta-\Delta_g)( u_{\varepsilon,R,a}-u_{\varepsilon,R})\|_{(0,\alpha),\mu-2,r_\varepsilon} \leq c|a|r_\varepsilon^{3+d-\mu-\frac{n}{2}}.
\end{equation}

As in the proof of Lemma \ref{lem05} we have that $(\Delta-\Delta_g)u_{\varepsilon,R}=O(|x|^N)$, and from this we obtain
\begin{equation}\label{eq107}
\|(\Delta-\Delta_g)u_{\varepsilon,R}\|_{(0,\alpha),\mu-2,r_\varepsilon}\leq cr_\varepsilon^{N'},
\end{equation}
where $N'$ is as big as we want. Hence, from (\ref{eq106}) and (\ref{eq107}), we get
%
%%%%%%%%%%%%%%%%%%%%%%%%%%%%%%%%% Equation %%%%%%%%%%%%%%%%%%%%%%%%%%%%%%%%%%%%%
%
\begin{equation}\label{eq04}
\|(\Delta-\Delta_g)u_{\varepsilon,R,a}\|_{(0,\alpha),\mu-2,r_\varepsilon} \leq cr_\varepsilon^{\delta_2}r_\varepsilon^{2+d-\mu-\frac{n}{2}},
\end{equation}
since $|a|r_\varepsilon^{1-\delta_2}\leq 1$, with $\delta_2>0$.

From Lemma \ref{lem02} and (\ref{eq73}), we conclude that
\begin{equation*}\label{eq79}
\|(\Delta-\Delta_g)v_{\phi}\|_{(0,\alpha)\mu-2,r_\varepsilon} \leq cr_\varepsilon^{1+d-\mu} \|\phi\|_{(2,\alpha),r_\varepsilon} \leq c\kappa r_\varepsilon^{3+2d-\mu-\frac{n}{2}-\delta_1}
\end{equation*}
and then
%
%%%%%%%%%%%%%%%%%%%%%%%%%%%%%%%%% Equation %%%%%%%%%%%%%%%%%%%%%%%%%%%%%%%%%%%%%
%
\begin{equation}\label{eq06}
\|(\Delta-\Delta_g)v_{\phi}\|_{(0,\alpha)\mu-2,r_\varepsilon}\leq c\kappa r_\varepsilon^{1+d-\delta_1} r_\varepsilon^{2+d-\mu-\frac{n}{2}}.
\end{equation}

Furthermore, since $5-\mu-n/2\geq 3+d-\mu-n/2$, $R_g=O(|x|^{2})$ and we have (\ref{eq105}), we get that
%
%%%%%%%%%%%%%%%%%%%%%%%%%%%%%%%%% Equation %%%%%%%%%%%%%%%%%%%%%%%%%%%%%%%%%%%%%
%
\begin{equation}\label{eq11}
\|R_gu_{\varepsilon,R,a}\|_{(0,\alpha),\mu-2,r_\varepsilon}\leq cr_\varepsilon^{5-\mu-\frac{n}{2}}\leq cr_\varepsilon r_\varepsilon^{2+d-\mu-\frac{n}{2}}.
\end{equation}
Using (\ref{eq73}), we also get
%
%%%%%%%%%%%%%%%%%%%%%%%%%%%%%%%%% Equation %%%%%%%%%%%%%%%%%%%%%%%%%%%%%%%%%%%%%
%
\begin{equation}\label{eq12}
\|R_gv_{\phi}\|_{(0,\alpha),\mu-2,r_\varepsilon}\leq cr_\varepsilon^{4-\mu}\|\phi\|_{(2,\alpha),r_\varepsilon} \leq c\kappa r_\varepsilon^{4-\delta_1} r_\varepsilon^{2+d-\mu-\frac{n}{2}},
\end{equation}
with $4-\delta_1>0$.

By Lemma \ref{lem01} and (\ref{eq73}), we obtain
%
%%%%%%%%%%%%%%%%%%%%%%%%%%%%%%%%% Equation %%%%%%%%%%%%%%%%%%%%%%%%%%%%%%%%%%%%%
%
\begin{equation}\label{eq14}
\|Q_{\varepsilon,R,a}(v_{\phi})\|_{(0,\alpha),\mu-2,r_\varepsilon}\leq c\varepsilon^{\lambda_n}r_\varepsilon^{1+d-2\mu}\|\phi\| ^2_{(2,\alpha),r_\varepsilon}\leq c\kappa^2\varepsilon^{\delta'}
r_\varepsilon^{2+d-\mu-\frac{n}{2}},
\end{equation}
with $\delta'=\lambda_n+s(3+2d-\mu-n/2-2\delta_1)> 0$, since  $\mu<5/4$, $s>(d+1-\delta_1)^{-1}$ and $0<\delta_1<(8n-16)^{-1}$.

Let us estimate the norm $\|u_{\varepsilon,R,a}^{\frac{4}{n-2}}v_\phi \|_{(0,\alpha),\mu-2,r_\varepsilon}$.

First, (\ref{eq76}) implies $u_{\varepsilon,R,a}^{\frac{4}{n-2}}= u_{\varepsilon,R}^{\frac{4}{n-2}}+O(|a||x|^{-1})$. Hence, using (\ref{eq73}), we deduce that
\begin{equation}\label{eq77}
\sigma^{2-\mu}\|(u_{\varepsilon,R,a}^{\frac{4}{n-2}}- u_{\varepsilon,R}^{\frac{4}{n-2}} )v_\phi \|_{(0,\alpha),[\sigma,2\sigma]} \leq C|a|r_\varepsilon^{1-\mu}\|\phi\|_{(2,\alpha),r_\varepsilon}\leq C\kappa r_\varepsilon^{\delta_2-\delta_1} r_\varepsilon^{2+d-\mu-\frac{n}{2}},
\end{equation}
since $|a|r_\varepsilon^{1-\delta_2}\leq 1$, with $\delta_2-\delta_1>0$.

If 
$r_\varepsilon^{1+\lambda}\leq |x|\leq r_\varepsilon$ with $\lambda>0$, then $$-s\log\varepsilon\leq -\log|x|\leq -s(1+\lambda)\log\varepsilon,$$
and by the choice of $R$, $R^{\frac{2-n}{2}}=2(1+b)\varepsilon^{-1}$ with $|b|<1/2$, see Remark \ref{remark03},we  obtain
$$\hspace{-3,5cm}\left(\frac{2}{n-2}-s\right)\log\varepsilon +\log(2+2b)^{\frac{2}{2-n}} \leq -\log|x|+\log R\leq$$
$$\hspace{4,6cm}\leq \left(\frac{2}{n-2}-s(1+\lambda)\right)\log\varepsilon +\log (2+2b)^{\frac{2}{2-n}},$$
with $\frac{2}{n-2}-s>0$, since $s<4(d-2+3n/2)^{-1}< 2(n-2)^{-1}$. We also have
$$v_{\varepsilon}(-\log|x|+\log R)\leq \varepsilon e^{\left(\frac{n-2}{2} s-1 \right)\log \varepsilon+\log(2+2b)}= (2+2b)\varepsilon^{\frac{n-2}{2} s}$$
for small enough $\lambda>0$. This follows from the estimate $v_\varepsilon(t)\leq \varepsilon e^{\frac{n-2}{2}|t|}$, $\forall t\in \mathbb{R}$. Hence
\begin{equation}\label{eq32}
u_{\varepsilon,R}^{\frac{4}{n-2}}(x)= |x|^{-2}v_{\varepsilon}^{\frac{4}{n-2}}(-\log|x|+\log R)\leq C_n|x|^{-2}r_\varepsilon^2.
\end{equation}

If we take $0<\lambda<\frac{2}{s(n-2)}-1$ fixed, then $\frac{2}{n-2}-s(1+\lambda)>0$ and from (\ref{eq32}) we get
$$\|u_{\varepsilon,R}^{\frac{4}{n-2}} \|_{(0,\alpha),[\sigma,2\sigma]}\leq C\sigma^{-2}r_\varepsilon^2,$$
for $r_\varepsilon^{1+\lambda}\leq \sigma\leq 2^{-1}r_\varepsilon$,
and then
\begin{equation}\label{eq95}
\sigma^{2-\mu}\|u_{\varepsilon,R}^{\frac{4}{n-2}} v_\phi\|_{(0,\alpha),[\sigma,2\sigma]}\leq C\kappa r_\varepsilon^{2-\delta_1} r_\varepsilon^{2+d-\mu-\frac{n}{2}},
\end{equation}
with $2-\delta_1>0$.

For $0\leq \sigma\leq r_\varepsilon^{1+\lambda}$, we have
\begin{equation}\label{eq96}
\sigma^{2-\mu}\|u_{\varepsilon,R}^{\frac{4}{n-2}} v_\phi\|_{(0,\alpha),[\sigma,2\sigma]} \leq Cr_\varepsilon^{(2-\mu)(1+\lambda)-2} \|\phi\|_{(2,\alpha),r_\varepsilon}\leq C\kappa r_\varepsilon^{(2-\mu)\lambda-\delta_1} r_\varepsilon^{2+d-\mu-\frac{n}{2}},
\end{equation}
Since $s<4(d-2+3n/2)^{-1}$, we can take $\lambda$ such that $\frac{1}{4n-8}<\lambda<\frac{2}{s(n-2)}-1$. This together with $\mu<5/4$ and $0<\delta_1<(8n-16)^{-1}$ implies $(2-\mu)\lambda-\delta_1>0$.

Therefore, by (\ref{eq77}), (\ref{eq95}) and (\ref{eq96}) we obtain
%
%%%%%%%%%%%%%%%%%%%%%%%%%%%%%%%%% Equation %%%%%%%%%%%%%%%%%%%%%%%%%%%%%%%%%%%%%
%
\begin{equation}\label{eq54}
\|u_{\varepsilon,R,a}^{\frac{4}{n-2}}v_\phi\|_{(0,\alpha),\mu-2,r_\varepsilon} \leq cr_\varepsilon^{\delta''-\mu}\|\phi\|\leq c\kappa r_\varepsilon^{\delta''-\delta_1}
r_{\varepsilon}^{2+d-\mu-\frac{n}{2}},
\end{equation}
for some $\delta''>\delta_1$ fixed independent of $\varepsilon$.

Therefore, from (\ref{eq04}), (\ref{eq06}), (\ref{eq11}), (\ref{eq12}), (\ref{eq14}) and (\ref{eq54}) it follows (\ref{eq112}) for small enough $\varepsilon>0$.

For the same reason as before,
$$ \|\mathcal{N}_\varepsilon(R,a,\phi,v_1)- \mathcal{N}_\varepsilon(R,a,\phi,v_2) \|_{(2,\alpha),\mu,r_\varepsilon}\leq c\left(\|(\Delta_g-\Delta)(v_1-v_2) \|_{(0,\alpha),\mu-2,r_\varepsilon}\right.$$
$$\left.+\|R_g(v_1-v_2)\|_{(0,\alpha),\mu-2,r_\varepsilon}
+ \|Q_{\varepsilon,R,a}( v_{\phi}+v_1)-Q_{\varepsilon,R,a}(v_{\phi}+v_2) \|_{(0,\alpha),\mu-2,r_\varepsilon}\right),$$
where $c>0$ is a constant independent of $\varepsilon$, $R$ and $a$.

From Lemma \ref{lem02} and $R_g=O(|x|^2)$ we obtain
%
%%%%%%%%%%%%%%%%%%%%%%%%%%%%%%%%% Equation %%%%%%%%%%%%%%%%%%%%%%%%%%%%%%%%%%%%%
%
\begin{equation}\label{eq15}
\|(\Delta-\Delta_g)(v_1-v_2)\|_{(0,\alpha),\mu-2,r_\varepsilon}\leq cr_\varepsilon^{d+1}\|v_1-v_2\|_{(2,\alpha),\mu,r_\varepsilon}
\end{equation}
and
%
%%%%%%%%%%%%%%%%%%%%%%%%%%%%%%%%% Equation %%%%%%%%%%%%%%%%%%%%%%%%%%%%%%%%%%%%%
%
\begin{equation}\label{eq16}
\|R_g(v_1-v_2)\|_{(0,\alpha),\mu-2,r_\varepsilon}\leq cr_\varepsilon^4\|v_1-v_2\|_{(2,\alpha),\mu,r_\varepsilon}.
\end{equation}

As before, Lemma \ref{lem01} and (\ref{eq73}) imply
%
%%%%%%%%%%%%%%%%%%%%%%%%%%%%%%%%% Equation %%%%%%%%%%%%%%%%%%%%%%%%%%%%%%%%%%%%%
%
\begin{equation}\label{eq17}
\|Q_{\varepsilon,R,a}(v_{\phi}+v_1) -Q_{\varepsilon,R,a}(v_{\phi} +v_2)\|_{(0,\alpha),\mu-2,r_\varepsilon}\leq
c_\kappa\varepsilon^{\lambda_n+s(3+2d-\mu-\frac{n}{2}-\delta_1)} \|v_1-v_2\|_{(2,\alpha),\mu,r_\varepsilon}
\end{equation}
with $\lambda_n+s(3+2d-\mu-n/2-\delta_1)>0$ as in (\ref{eq14}).

Therefore, from (\ref{eq15}), (\ref{eq16}) and (\ref{eq17}), we deduce (\ref{eq113}) provided $v_1$, $v_2$ belong to the ball of radius $\tau r_\varepsilon^{2+d-\mu-\frac{n}{2}}$ in $C_\mu^{2,\alpha}(B_{r_\varepsilon}(0)\backslash\{0\}))$ for $\varepsilon>0$ chosen small enough.

Consider $n\geq 8$. 

Similarly
$$ \hspace{-2cm}\|\mathcal{N}_\varepsilon(R,a,\phi,0)\|_{(2,\alpha),\mu,r_\varepsilon} \leq c\left(\|(\Delta-\Delta_g)(u_{\varepsilon,R,a}+v_\phi+w_{\varepsilon,R}) \|_{(0,\alpha),\mu-2,r_\varepsilon}\right.$$
$$\hspace{-2cm}+ \|R_g(v_\phi+w_{\varepsilon,R})\|_{(0,\alpha),\mu-2,r_\varepsilon} + \|Q_{\varepsilon,R,a}(v_\phi+w_{\varepsilon,R}) \|_{(0,\alpha),\mu-2,r_\varepsilon}$$
$$+ \|(R_g-\partial_i\partial_jh_{ij})u_{\varepsilon,R,a} \|_{(0,\alpha),\mu-2,r_\varepsilon}+ \|\partial_i\partial_jh_{ij} (u_{\varepsilon,R,a}-u_{\varepsilon,R})\|_{(0,\alpha),\mu-2,r_\varepsilon}$$
$$\hspace{2cm}+ \|u_{\varepsilon,R,a}^{\frac{4}{n-2}} v_\phi\|_{(0,\alpha),\mu-2,r_\varepsilon}+ \|(u_{\varepsilon,R}^{\frac{4}{n-2}}- u_{\varepsilon,R,a}^{\frac{4}{n-2}})w_{\varepsilon,R} \|_{(0,\alpha),\mu-2,r_\varepsilon}+ \left.\|\displaystyle \overline{h}u_{\varepsilon,R} \|_{(0,\alpha),\mu-2,r_\varepsilon}\right),$$
where $c>0$ does not depend on $\varepsilon$, $R$ and $a$.

From (\ref{eq73}), (\ref{eq57}), Lemma \ref{lem01} and the fact that $R_g=O(|x|^{d-1})$, we get
%
%%%%%%%%%%%%%%%%%%%%%%%%%%%%%%%%% Equation %%%%%%%%%%%%%%%%%%%%%%%%%%%%%%%%%%%%%
%
\begin{equation}\label{eq21}
\|(\Delta-\Delta_g)w_{\varepsilon,R}\|_{(0,\alpha),\mu-2,r_\varepsilon} \leq cr_\varepsilon^{d+1} r_\varepsilon^{2+d-\mu-\frac{n}{2}},
\end{equation}
%
%%%%%%%%%%%%%%%%%%%%%%%%%%%%%%%%% Equation %%%%%%%%%%%%%%%%%%%%%%%%%%%%%%%%%%%%%
%
\begin{equation}\label{eq22}
\|R_g(v_\phi+w_{\varepsilon,R})\|_{(0,\alpha),\mu-2,r_\varepsilon}\leq c\kappa r_\varepsilon^{1+d-\delta_1} r_\varepsilon^{2+d-\mu-\frac{n}{2}},
\end{equation}
and
%
%%%%%%%%%%%%%%%%%%%%%%%%%%%%%%%%% Equation %%%%%%%%%%%%%%%%%%%%%%%%%%%%%%%%%%%%%
%
\begin{equation}\label{eq13}
\|Q_{\varepsilon,R,a}(v_\phi+w_{\varepsilon,R}) \|_{(0,\alpha),\mu-2,r_\varepsilon}\leq cr_\varepsilon^{\delta'} r_\varepsilon^{2+d-\mu-\frac{n}{2}},
\end{equation}
for some $\delta'>0$.

Note that 
$$(R_g-\partial_i\partial_jh_{ij}) u_{\varepsilon,R,a}=O(|x|^{\frac{n}{2}-2})$$ 
and 
$$\partial_i\partial_jh_{ij}(u_{\varepsilon,R,a}-u_{\varepsilon,R}) = O(|a||x|^{1+d-\frac{n}{2}}),$$
by Corollary \ref{cor02}. This implies
%
%%%%%%%%%%%%%%%%%%%%%%%%%%%%%%%%% Equation %%%%%%%%%%%%%%%%%%%%%%%%%%%%%%%%%%%%%
%
\begin{equation}\label{eq24}
\|(R_g-\partial_i\partial_jh_{ij}) u_{\varepsilon,R,a}\|_{(0,\alpha),\mu-2,r_\varepsilon}\leq cr_\varepsilon r_\varepsilon^{2+d-\mu-\frac{n}{2}}
\end{equation}
and
%
%%%%%%%%%%%%%%%%%%%%%%%%%%%%%%%%% Equation %%%%%%%%%%%%%%%%%%%%%%%%%%%%%%%%%%%%%
%
\begin{equation}\label{eq26}
\|\partial_i\partial_jh_{ij}(u_{\varepsilon,R,a}- u_{\varepsilon,R})\|_{(0,\alpha),\mu-2,r_\varepsilon}\leq c|a|r_\varepsilon r_\varepsilon^{2+d-\mu-\frac{n}{2}}.
\end{equation}

Finally, by the proof of Corollary \ref{cor01} we have  $u_{\varepsilon,R,a}^{\frac{4}{n-2}}- u_{\varepsilon,R}^{\frac{4}{n-2}}=O(|a||x|^{-1}).$
Hence,
%
%%%%%%%%%%%%%%%%%%%%%%%%%%%%%%%%% Equation %%%%%%%%%%%%%%%%%%%%%%%%%%%%%%%%%%%%%
%
\begin{equation}\label{eq58}
\|(u_{\varepsilon,R}^{\frac{4}{n-2}}- u_{\varepsilon,R,a}^{\frac{4}{n-2}})w_{\varepsilon,R} \|_{(0,\alpha),\mu-2,r_\varepsilon}\leq c|a|r_\varepsilon r_\varepsilon^{2+d-\mu-\frac{n}{2}}.
\end{equation}

Since $\overline{h}=O(|x|^{N'})$, where $N'$ is as big as we want, by (\ref{eq04}), (\ref{eq06}), (\ref{eq54}), (\ref{eq21}), (\ref{eq22}), (\ref{eq13}), (\ref{eq24}), (\ref{eq26}) and (\ref{eq58}), we deduce (\ref{eq113}) for $\varepsilon>0$ small enough.

Now, we have
$$\|\mathcal{N}_\varepsilon(R,a,\phi,v_1)- \mathcal{N}_\varepsilon(R,a,\phi,v_2) \|_{(2,\alpha),\mu,r_\varepsilon}\leq c\left(\|(\Delta_g-\Delta)(v_1-v_2) \|_{(0,\alpha),\mu-2,r_\varepsilon}\right.$$
$$+ \|Q_{\varepsilon,R,a}( v_\phi+w_{\varepsilon,R,a}+v_1)- Q_{\varepsilon,R,a}(v_\phi+w_{\varepsilon,R,a}+v_2) \|_{(0,\alpha),\mu-2},r_\varepsilon$$
$$\left.+\|R_g(v_1-v_2)\|_{(0,\alpha),\mu-2,r_\varepsilon}\right).$$

As before we obtain
%
%%%%%%%%%%%%%%%%%%%%%%%%%%%%%%%%% Equation %%%%%%%%%%%%%%%%%%%%%%%%%%%%%%%%%%%%%
%
\begin{equation}\label{eq33}
\|(\Delta-\Delta_g)(v_1-v_2)\|_{(0,\alpha),\mu-2,r_\varepsilon}\leq cr_\varepsilon^{d+1}\|v_1-v_2\|_{(2,\alpha),\mu,r_\varepsilon}
\end{equation}
and
%
%%%%%%%%%%%%%%%%%%%%%%%%%%%%%%%%% Equation %%%%%%%%%%%%%%%%%%%%%%%%%%%%%%%%%%%%%
%
\begin{equation}\label{eq34}
\|R_g(v_1-v_2)\|_{(0,\alpha),\mu-2,r_\varepsilon}\leq cr_\varepsilon^{d+1}\|v_1-v_2\|_{(0,\alpha),\mu,r_\varepsilon}.
\end{equation}

By Lemma \ref{lem01} and (\ref{eq73}), we obtain
\begin{equation}\label{eq81}
\begin{array}{c}
\|Q_{\varepsilon,R,a}( v_\phi+w_{\varepsilon,R,a}+v_1)- Q_{\varepsilon,R,a}(v_\phi+w_{\varepsilon,R,a}+v_2) \|_{(0,\alpha),\mu-2,r_\varepsilon}\leq\\
\leq c_\kappa \varepsilon^{\lambda_n+s(3+2d-\mu-\frac{n}{2}-\delta_1)} \|v_1-v_2\|_{(2,\alpha),\mu,r_\varepsilon},
\end{array}
\end{equation}
with $\lambda_n+s(3+2d-\mu-n/2-\delta_1)>0$.

Therefore, from (\ref{eq33}), (\ref{eq34}) and (\ref{eq81}), we deduce (\ref{eq113}) provided $v_1$, $v_2$ belong to the ball of radius $\tau r_\varepsilon^{2+d-\mu-\frac{n}{2}}$ in $C_\mu^{2,\alpha}(B_{r_\varepsilon}(0)\backslash\{0\}))$ for $\varepsilon>0$ chosen small enough.
\end{proof}

We summarize the main result of this section in the next theorem.
%
%%%%%%%%%%%%%%%%%%%%%%%%%%%%%%%%% Theorem %%%%%%%%%%%%%%%%%%%%%%%%%%%%%%%%%%%%%
%
\begin{theorem}\label{teo01}
Let $\mu\in (1,5/4)$, $\tau>0$, $\kappa>0$ and $\delta_2>\delta_1$ be fixed constants. There exists a constant $\varepsilon_0\in(0,1)$ such that for each $\varepsilon\in(0,\varepsilon_0]$, $|b|\leq 1/2$, $a\in\mathbb{R}^n$ with $|a|r_\varepsilon^{1-\delta_2}\leq 1$ and $\phi\in\pi''(C^{2,\alpha}(\mathbb{S}_{r_\varepsilon}^{n-1}))$ with $\|\phi\|_{(2,\alpha),r_\varepsilon}\leq \kappa r_\varepsilon^{2+d-\frac{n}{2}-\delta_1}$, there exists a solution $U_{\varepsilon,R,a,\phi}\in C^{2,\alpha}_{\mu}(B_{r_\varepsilon}(0)\backslash\{0\})$ for the equation
$$\left\{\begin{array}{lcl}
H_g(u_{\varepsilon,R,a}+w_{\varepsilon,R}+v_{\phi}+ U_{\varepsilon,R,a,\phi})=0 & \mbox{ in } & B_{r_\varepsilon}(0)\backslash\{0\}\\
\pi''_{r_\varepsilon}((v_{\phi}+U_{\varepsilon,R,a,\phi})|_{\partial B_{r_\varepsilon}(0)})=\phi  & \mbox{ on } & \partial B_{r_\varepsilon}(0)
\end{array}\right.$$
where $w_{\varepsilon,R}\equiv 0$ for $3\leq n\leq 7$ and $w_{\varepsilon,R}\in \pi''(C_{2+d-\frac{n}{2}}^{2,\alpha}(B_{r_\varepsilon}(0)\backslash\{0\}))$ is solution of the equation (\ref{eq35}) for $n\geq 8$.

Moreover,
\begin{equation}\label{eq71}
\|U_{\varepsilon,R,a,\phi}\|_{(2,\alpha),\mu,r_\varepsilon}\leq \tau r_\varepsilon^{2+d-\mu-\frac{n}{2}}
\end{equation}
and
\begin{equation}\label{eq82}
\|U_{\varepsilon,R,a,\phi_1}- U_{\varepsilon,R,a,\phi_2}\|_{(2,\alpha),\mu,r_\varepsilon}\leq Cr_\varepsilon^{\delta_3-\mu}\|\phi_1-\phi_2\|_{(2,\alpha),r_\varepsilon},
\end{equation}
for some constants $\delta_3>0$  that does not depend on $\varepsilon$, $R$, $a$ and $\phi_i$, $i=1,2$.
\end{theorem}
%
%%%%%%%%%%%%%%%%%%%%%%%%%%%%%%%%% Proof %%%%%%%%%%%%%%%%%%%%%%%%%%%%%%%%%%%%%
%
\begin{proof}
The solution $U_{\varepsilon,R,a,\phi}$ is the fixed point of the map $\mathcal{N}_\varepsilon(R,a,\phi,\cdot)$ given by Proposition \ref{propo03} with the estimate (\ref{eq71}).

Use the fact that $U_{\varepsilon,R,a,\phi}$ is a fixed point of the map $\mathcal{N}_\varepsilon(R,a,\phi,\cdot)$ to show that
$$\|U_{\varepsilon,R,a,\phi_1}- U_{\varepsilon,R,a,\phi_2}\|_{(2,\alpha),\mu,r_\varepsilon}\leq 2\|\mathcal{N}_\varepsilon(R,a,\phi_1, U_{\varepsilon,R,a,\phi_2})- \mathcal{N}_\varepsilon(R,a,\phi_2, U_{\varepsilon,R,a,\phi_2})\|_{(2,\alpha),\mu,r_\varepsilon}.$$

From this we obtain the inequality (\ref{eq82}).
\end{proof}
We will write the full conformal factor of the resulting constant scalar curvature metric with respect to the metric $g$ as
$$\mathcal{A}_{\varepsilon}(R,a,\phi):=u_{\varepsilon,R,a}+w_{\varepsilon,R}+ v_\phi+U_{\varepsilon,R,a,\phi},$$
in conformal normal coordinates. More precisely, the previous analysis says that the metric $\hat{g}=\mathcal{A}_{\varepsilon}(R,a,\phi)^{\frac{4}{n-2}}g$ is defined in $\overline{B_{r_\varepsilon}(p)}\backslash\{p\}\subset M$, it is complete and has constant scalar curvature $R_{\hat{g}}=n(n-1)$. The completeness follows from the estimate
$$\mathcal{A}_{\varepsilon}(R,a,\phi)\geq c|x|^{\frac{2-n}{2}},$$
for some constant $c>0$.

%
%%%%%%%%%%%%%%%%%%%%%%%%%%%%%%%%% CHAPTER %%%%%%%%%%%%%%%%%%%%%%%%%%%%%%%%%%%%%%
%

\section{Exterior Analysis}\label{sec04}
%
%%%%%%%%%%%%%%%%%%%%%%%%%%%%%%%%% Section %%%%%%%%%%%%%%%%%%%%%%%%%%%%%%%%%%%%%%
%
%\section{Introduction}
In Section \ref{sec02} we have found a family of constant scalar curvature metrics on $\overline{B_{r_\varepsilon}(p)}\backslash\{p\}\subset M$, conformal to $g_0$ and with prescribed high eigenmode data. Now we will use the same method of the previous section to perturb the metric $g_0$ and build a family of constant scalar curvature metrics on the complement of some suitable ball centered at $p$ in $M$.

First, using the non-degeneracy we find a right inverse for the operator $L_{g_0}^1$ (see (\ref{eq29})), in the complement of the ball $B_r(p)\subset M$ for small enough $r$, with bounded norm independently of $r$.

In contrast with the previous section, in which we worked with  conformal normal coordinates, in this section it is better to work with the constant scalar curvature metric, since in this case the constant function $1$ satisfies $H_{g_0}(1)=0$. Hence, in this section, $(M^n,{g_0})$ is an $n-$dimensional nondegenerate compact Riemannian manifold of constant scalar curvature $R_{{g_0}}=n(n-1)$. 

%
%%%%%%%%%%%%%%%%%%%%%%%%%%%%%%%%% Section %%%%%%%%%%%%%%%%%%%%%%%%%%%%%%%%%%%%%%
%
\subsection{Analysis in $M\backslash B_r(p)$}
Let $r_1\in(0,1)$ and $\Psi:B_{r_1}(0)\rightarrow M$ be a normal coordinate system with respect to $g=\mathcal{F}^{\frac{4}{n-2}}g_0$ on $M$ centered at $p$, where $\mathcal{F}$ is defined in Section \ref{sec02}. We denote by $G_p(x)$ the Green's function for $L_{g_0}^1=\Delta_{g_0}+n$, the linearization of $H_{g_0}$ about the constant function 1, with pole at $p$ (the origin in our coordinate system). We assume that $G_p(x)$ is normalized such that in the coordinates $\Psi$ we have $\displaystyle\lim_{x\rightarrow 0}|x|^{n-2}G_p(x)=1$. This implies that $|G_p\circ\Psi(x)|\leq C|x|^{2-n}$, for all $x\in B_{r_1}(0)$. In these coordinates we have that $(g_0)_{ij}=\delta_{ij}+O(|x|^2)$, since $g_{ij}=\delta_{ij}+O(|x|^2)$ and $\mathcal{F}=1+O(|x|^2)$.

Our goal in this section is to solve the equation
%
%%%%%%%%%%%%%%%%%%%%%%%%%%%%%%%%% Equation %%%%%%%%%%%%%%%%%%%%%%%%%%%%%%%%%%%%%
%
\begin{equation}\label{eq25}
H_{{g_0}}(1+\lambda G_{p}+u)=0\;\;\;\mbox{ on }\;\;\;M\backslash B_r(p)
\end{equation}
with $\lambda\in\mathbb{R}$, $r\in (0,r_1)$ and prescribed boundary data on $\partial B_r(p)$. In fact, we will get a solution with prescribed boundary data, except in the space spanned by the constant functions.

To solve this equation we will use basically the same techniques that were used in Proposition \ref{propo03}. We linearize $H_{{g_0}}$ about 1 to get
$$H_{{g_0}}(1+\lambda G_{p}+u)=L_{{g_0}}^1(u)+Q^1(\lambda G_{p}+u),$$
since $H_{{g_0}}(1)=0$ and $L_{{g_0}}^1(G_{p})=0$, where $Q^1$ is given by (\ref{eq44}). Next, we will find a right inverse for $L_{g_0}^1$ in a suitable space and so we will reduce the equation (\ref{eq25}) to the problem of fixed point as in the previous section.
%
%%%%%%%%%%%%%%%%%%%%%%%%%%%%%%%%%Section%%%%%%%%%%%%%%%%%%%%%%%%%%%%%%%%%%%%%%
%
\subsection{Inverse for $L_{{g_0}}^1$ in $M\backslash\Psi(B_r(0))$}\label{sec15}
To find a right inverse for $L_{{g_0}}^1$, we will follow the method of Jleli in \cite{M}. This problem is approached by decomposing $f$ as the sum of two functions, one of them with support contained in an annulus inside $\Psi(B_{r_1}(0))$. Inside the annulus we transfer the problem to normal coordinates and solve. For the remainder term we use the right invertibility of $L_{{g_0}}^1$ on $M$ which is a consequence of the non-degeneracy.

The next two lemmas allow us to use a perturbation argument in the annulus contained in $\Psi(B_{r_1}(0))$.
%
%%%%%%%%%%%%%%%%%%%%%%%%%%%%%%%%% Lemma %%%%%%%%%%%%%%%%%%%%%%%%%%%%%%%%%%%%%
%
\begin{lemma}\label{lem06}
Fix any $\nu\in\mathbb{R}$. There exists $C>0$ independent of $r$ and $s$  such that if $0<2r<s\leq r_1$, then
$$\|(L_{{g_0}}^1-\Delta) (v)\|_{C^{0,\alpha}_{\nu-2}(\Omega_{r,s})}\leq Cs^2\|v\|_{C^{2,\alpha}_{\nu}(\Omega_{r,s})},$$
for all $v\in C_\nu^{2,\alpha}(\Omega_{r,s})$.
\end{lemma}
%
%%%%%%%%%%%%%%%%%%%%%%%%%%%%%%%%% Proof %%%%%%%%%%%%%%%%%%%%%%%%%%%%%%%%%%%%%
%
\begin{proof} Use the Laplacian in local coordinates.
\end{proof}
%
%%%%%%%%%%%%%%%%%%%%%%%%%%%%%%%%% Lemma %%%%%%%%%%%%%%%%%%%%%%%%%%%%%%%%%%%%%
%
\begin{lemma}\label{lem3}
Assume that $\nu\in\left(1-n,2-n\right)$ is fixed and that $0<2r<s\leq r_1$. Then there exists an operator
$$\tilde{G}_{r,s}:C^{0,\alpha}_{\nu-2}(\Omega_{r,s})\rightarrow C^{2,\alpha}_{\nu}(\Omega_{r,s})$$
such that, for all $f\in C^{0,\alpha}_{\nu}(\Omega_{r,s})$, the function $w=\tilde{G}_{r,s}(f)$ is a solution of
%
%%%%%%%%%%%%%%%%%%%%%%%%%%%%%%%%% Equation %%%%%%%%%%%%%%%%%%%%%%%%%%%%%%%%%%%%%
%
\begin{eqnarray*}%\label{eq29}
\left\{\begin{array}{rclccl}
\Delta w & = & f & \mbox{ in } & B_s(0)\backslash B_r(0)\\
w & = & 0 & \mbox{ on } & \partial B_s(0)\\
w & \in & \mathbb{R} & \mbox{ on } & \partial B_r(0)
\end{array}\right..
\end{eqnarray*}
In addition,
$$\|\tilde{G}_{r,s}(f)\|_{C_\nu^{2,\alpha}(\Omega_{r,s})}\leq C\|f\|_{C_{\nu-2}^{0,\alpha}(\Omega_{r,s})},$$
for some constant $C>0$ that does not depend on $s$ and $r$.
\end{lemma}
\begin{proof} See lemma 13.23 in \cite{M} and \cite{JP}.
\end{proof}
%
%%%%%%%%%%%%%%%%%%%%%%%%%%%%%%%%% Propostion %%%%%%%%%%%%%%%%%%%%%%%%%%%%%%%%%%
%
\begin{proping}\label{propo04}
Fix $\nu\in(1-n,2-n)$. There exists $r_2< \frac{1}{4}r_1$ such that, for all $r\in(0,r_2)$ we can define an operator
$$G_{r,{{g_0}}}:C^{0,\alpha}_{\nu-2}(M_r)\rightarrow C^{2,\alpha}_{\nu}(M_r),$$
with the property that, for all $f\in C^{0,\alpha}_{\nu-2}(M_r)$ the function $w=G_{r,{{g_0}}}(f)$ solves
$$L_{g_0}^1(w)  =  f,$$
in $M_r$ with $w\in\mathbb{R}$ constant on $\partial B_r(p)$. In addition
$$\|G_{r,{{g_0}}}(f)\|_{C^{2,\alpha}_\nu(M_r)}\leq C\|f\|_{C^{0,\alpha}_{\nu-2}(M_r)},$$
where $C>0$ does not depend on $r.$
\end{proping}
%
%%%%%%%%%%%%%%%%%%%%%%%%%%%%%%%%% Proof %%%%%%%%%%%%%%%%%%%%%%%%%%%%%%%%%%%%%
%
\begin{proof} From Lemma \ref{lem06} and a perturbation argument follows that the result of Lemma \ref{lem3} holds for $s=r_1$ small enough when $\Delta$ is replaced by $L_{{g_0}}^1$. We denote by $G_{r,r_1}$ the corresponding operator.

Let $f\in C^{0,\alpha}_{\nu-2}(M_r)$ and define a function $w_0\in C^{2,\alpha}_\nu(M_r)$ by
$$w_0:=\eta G_{r,r_1}(f|_{\Omega_{r,r_1}})$$%=\left\{\begin{array}{cl}
%G_{r,r_1}(f|_{\Omega_{r,r_1}}) & \mbox{on } \Omega_{r,\frac{r_1}{2}}\\
%0 & \mbox{on } M_{r_1}\\
%\end{array}\right..$$
where $\eta$ is a smooth, radial function equal to 1 in $B_{\frac{1}{2}r_1}(p)$, vanishing in $M_{r_1}$ and satisfying $|\partial_r\eta(x)|\leq c|x|^{-1}$ and $|\partial_r^2\eta(x)|\leq c|x|^{-2}$ for all $x\in B_{r_1}(0)$. From this it follows that $\|\eta\|_{(2,\alpha),[\sigma,2\sigma]}$ is uniformly bounded in $\sigma$, for every $r\leq \sigma\leq \frac{1}{2}r_1$. Thus,
\begin{equation}\label{eq08}
\|w_0\|_{C^{2,\alpha}_\nu(M_r)}\leq C\| f\|_{{C^{0,\alpha}_{\nu-2}(M_r)}},
\end{equation}
where the constant $C>0$ is independent of $r$ and $r_1$.

Since $w_0=G_{r,r_1}(f|_{\Omega_{r,r_1}})$ in $\Omega_{r,\frac{1}{2}r_1}$, the function
$$h:=f-L_{{g_0}}^1(w_0)$$
is supported in $M_{\frac{1}{2}r_1}$. We can consider that $h$ is defined on the whole $M$ with $h\equiv 0$ in $B_{\frac{1}{2}r_1}(p)$. By (\ref{eq08}) we get
\begin{equation}\label{eq23}
\|h\|_{C^{0,\alpha}(M)}\leq C_{r_1}\|f\|_{C^{0,\alpha}_{\nu-2}(M_r)},
\end{equation}
with the constant $C_{r_1}>0$ independent of $r$.

Since $L_{{g_0}}^1:C^{2,\alpha}(M)\rightarrow C^{0,\alpha}(M)$ has a bounded inverse, we can define the function
$$w_1:=\chi (L_{{g_0}}^1)^{-1}(h),$$
where $\chi$ is a smooth, radial function equal to 1 in $M_{2r_2}$, vanishing in $B_{r_2}(p)$ and satisfying $|\partial_r\chi(x)|\leq c|x|^{-1}$ and $|\partial_r^2\chi(x)|\leq c|x|^{-2}$ for all $x\in B_{2r_2}(0)$ and some $r_2\in(r,\frac{1}{4}r_1)$ to be chosen later. This implies that $\|\chi\|_{(2,\alpha),[\sigma,2\sigma]}$ is uniformly bounded for $r\leq \sigma\leq \frac{1}{2}r_1$.

Hence, from (\ref{eq23})
\begin{equation}\label{eq19}
\|w_1\|_{C^{2,\alpha}_\nu(M_r)}\leq C_{r_1}\|(L_{{g_0}}^1)^{-1}(h)\|_{C^{2,\alpha}(M)}\leq C_{r_1}\|h\|_{C^{0,\alpha}(M)}\leq C_{r_1}\|f\|_{C^{0,\alpha}_{\nu-2}(M_r)},
\end{equation}
since $\nu<0$, where the constant $C_{r_1}>0$ is independent of $r$ and $r_2$.

Define an application $F_{r,{{g_0}}}:C^{0,\alpha}_{\nu-2}(M_r)\rightarrow C^{2,\alpha}_{\nu}(M_r)$ as
$$F_{r,{{g_0}}}(f)=w_0+w_1.$$

From (\ref{eq08}), (\ref{eq23}) and (\ref{eq19}) we obtain
\begin{equation}\label{eq31}
\|F_{r,g_0}(f)\|_{C^{2,\alpha}_{\nu}(M_r)}\leq C\|f\|_{C^{0,\alpha}_{\nu-2}(M_r)},
\end{equation}
and
\begin{equation}\label{eq30}
\|L_{g_0}^1(F_{r,g_0}(f))-f\|_{C^{0,\alpha}_{\nu-2}(M_r)}\leq Cr_2^{-1-\nu}\|f\|_{C^{0,\alpha}_{\nu-2}(M_r)}
\end{equation}
since $1-n<\nu<2-n$ implies that $2-\nu>0$ and $-1-\nu>0$, for some constant $C>0$ independent of $r$ and $r_2$. The assertion follows from a perturbation argument by (\ref{eq31}) and (\ref{eq30}), as in the proof of Corollary \ref{cor01}.
\end{proof}
%
%%%%%%%%%%%%%%%%%%%%%%%%%%%%%%%%%Section%%%%%%%%%%%%%%%%%%%%%%%%%%%%%%%%%%%%%%
%
\subsection{Constant scalar curvature metrics on $M\backslash B_r(p)$}\label{sec13}
In this section we will solve the equation (\ref{eq25}) using the method employed in the interior analysis, the fixed point method. In fact we will find a family of metrics with parameters $\lambda\in\mathbb{R}$, $0<r<r_1$ and some boundary data.

For each $\varphi \in C^{2,\alpha}(\mathbb{S}_r^{n-1})$ $L^2-$orthogonal to the constant functions, let $u_{\varphi}\in C^{2,\alpha}_\nu(M_r)$ be such that $u_{\varphi}\equiv 0$ in $M_{r_1}$ and $u_{\varphi}\circ\Psi=  \eta\mathcal{Q}_r(\varphi)$, where $\mathcal{Q}_r$ is defined in Section \ref{sec06}, $\eta$ is a smooth, radial function equal to 1 in $B_{\frac{1}{2}r_1}(0)$, vanishing in $\mathbb{R}^n\backslash B_{r_1}(0)$, and satisfying $|\partial_r\eta(x)|\leq c|x|^{-1}$ and $|\partial_r^2\eta(x)|\leq c|x|^{-2}$ for all $x\in B_{r_1}(0)$. As before, we have $\|\eta\|_{(2,\alpha),[\sigma,2\sigma]}\leq c$, for every $r\leq\sigma\leq\frac{1}{2}r_1$. Hence, using (\ref{eq74}) we conclude that
\begin{equation}\label{eq111}
\|u_{\varphi}\|_{C^{2,\alpha}_\nu(M_r)}\leq cr^{-\nu}\|\varphi\|_{(2,\alpha),r},
\end{equation}
for all $\nu\geq 1-n$.

Finally, substituting $u:=u_{\varphi}+v$ in equation (\ref{eq25}), we have that to show the existence of a solution of the equation (\ref{eq25}) it is enough to show that for suitable $\lambda\in\mathbb{R}$, and $\varphi\in C^{2,\alpha}(\mathbb{S}^{n-1}_r)$ the map $\mathcal{M}_r(\lambda,\varphi,\cdot):C^{2,\alpha}_\nu(M_r)\rightarrow C^{2,\alpha}_\nu(M_r)$, given by
%
%%%%%%%%%%%%%%%%%%%%%%%%%%%%%%%%% Equation %%%%%%%%%%%%%%%%%%%%%%%%%%%%%%%%%%%%%
%
\begin{equation}\label{eq60}
\mathcal{M}_r(\lambda,\varphi,v) = -G_{r,{{g_0}}}(Q^1(\lambda G_{p} +u_{\varphi}+v)+L_{{g_0}}^1(u_{\varphi})),
\end{equation}
has a fixed point for small enough $r>0$. We will show that $\mathcal{M}_r(\lambda,\varphi,\cdot)$ is a contraction, and as a consequence the fixed point will depend continuously on the parameters $r$, $\lambda$ and $\varphi$.
%
%%%%%%%%%%%%%%%%%%%%%%%%%%%%%%%%% Proposition %%%%%%%%%%%%%%%%%%%%%%%%%%%%%%%%%%
%
\begin{proping}\label{propo07}
Let $\nu\in(3/2-n,2-n)$, $\delta_4\in(0,1/2)$, $\beta>0$ and $\gamma>0$ be fixed constants. There exists $r_2\in (0,r_1/4)$ such that if $r\in(0,r_2)$, $\lambda\in\mathbb{R}$ with $|\lambda|^2\leq r^{d-2+\frac{3n}{2}}$, and $\varphi\in C^{2,\alpha}(\mathbb{S}_r^{n-1})$ is $L^2-$orthogonal to the constant functions with $\|\varphi\|_{(2,\alpha),r}\leq \betaup r^{2+d-\frac{n}{2}-\delta_4}$, then there is a fixed point of the map $\mathcal{M}_r(\lambda,\varphi, \cdot)$ in the ball of radius $\gamma r^{2+d-\nu-\frac{n}{2}}$ in $C_{\nu}^{2,\alpha}(M_r)$.
\end{proping}
%
%%%%%%%%%%%%%%%%%%%%%%%%%%%%%%%%% Proof %%%%%%%%%%%%%%%%%%%%%%%%%%%%%%%%%%%%%
%
\begin{proof} As in Proposition \ref{propo03} we will show that
\begin{equation}\label{eq114}
\|\mathcal{M}_r(\lambda,\varphi,0)\|_{C_\nu^{2,\alpha}(M_r)}\leq \frac{1}{2}\gamma r_\varepsilon^{2+d-\nu-\frac{n}{2}}
\end{equation}
and
\begin{equation}\label{eq115}
\|\mathcal{M}_r(\lambda,\varphi,v_1)- \mathcal{M}_r(\lambda,\varphi,v_2)\|_{C_\nu^{2,\alpha}(M_r)}\leq \frac{1}{2}\|v_1-v_2\|_{C_\nu^{2,\alpha}(M_r)},
\end{equation}
for all $v_i\in C^{2,\alpha}_\nu(M_r)$ with $\|v_i\|_{C^{2,\alpha}_\nu(M_r)}\leq \gamma r_\varepsilon^{2+d-\nu-\frac{n}{2}}$, $i=1$ and $2$.

From (\ref{eq60}) and Proposition \ref{propo04} it follows that
$$\|\mathcal{M}_r(\lambda,\varphi,0)\|_{C^{2,\alpha}_\nu(M_r)}\leq c\left(\|Q^1(\lambda G_{p}+u_{\varphi}) \|_{C_{\nu-2}^{0,\alpha}(M_r)}+\|L_{{g_0}}^1(u_{\varphi}) \|_{C_{\nu-2}^{0,\alpha}(\Omega_{r,r_1})}\right),$$
for some constant $c>0$ independent of $r$.

Note that $|\lambda G_p|\leq cr^{1+\frac{d}{2}-\frac{n}{4}},$ with $1+d/2-n/4>0$ and $c>0$ independent of $r$, and from (\ref{eq44})
\begin{equation}\label{eq83}
Q^1(u)=\displaystyle\frac{n(n+2)}{n-2}u^2 \int_0^1\int_0^1(1+stu)^{\frac{6-n}{n-2}}sdsdt
\end{equation}
for $1+stu>0$. Since $0<c<1+st\lambda G_p<C$ in $M_{r_1}$ for small enough $r$, then
$$\max_{t\in[0,1]}\|(1+st\lambda G_p)^{\frac{6-n}{n-2}}\|_{C^{0,\alpha}(M_{\frac{1}{2}r_1})}\leq c,$$
and
\begin{equation}\label{eq84}
\|G_p\|_{C^{0,\alpha}(M_{\frac{1}{2}r_1})}\leq c,
\end{equation}
where $c>0$ is a constant independent of $r$. Thus, by (\ref{eq83}) and (\ref{eq84}) we have
%
%%%%%%%%%%%%%%%%%%%%%%%%%%%%%%%%% Equation %%%%%%%%%%%%%%%%%%%%%%%%%%%%%%%%%%%%%
%
\begin{equation}\label{eq40}
\|Q^1(\lambda G_{p})\|_{C^{0,\alpha}(M_{\frac{1}{2}r_1})} \leq C\displaystyle|\lambda|^2\leq Cr^{\delta'} r^{2+d-\nu-\frac{n}{2}},
\end{equation}
where the constant $C>0$ does not depend on $r$ and $\delta'=2n-4+\nu>0$ since $\nu>3/2-n$.

Now, observe that (\ref{eq111}) implies 
$|u_\varphi(x)|\leq c\beta r^{2+d-\frac{n}{2}-\delta_4},\;\; \forall x\in M_r,$
with $2+d-n/2-\delta_4>0$. From this and $|\lambda G_p(x)|\leq cr^{1+\frac{d}{2}-\frac{n}{4}}$ for all $x\in\Omega_{r,r_1}$, we get $0<c<1+t(\lambda G_p+u_{\varphi})<C$ for every $0\leq t\leq 1$. Again, using (\ref{eq111}) and $|\lambda \nabla G_p|\leq cr^{\frac{d}{2}-\frac{n}{4}}$, we conclude that the H\"older norm of $(1+t(\lambda G_{p}+u_{\varphi}))^{\frac{6-n}{n-2}}$ is bounded independently of $r$ and $t$. Therefore,
$$\max_{0\leq t\leq 1} \|(1+t(\lambda G_{p}+u_{\varphi}))^{\frac{6-n}{n-2}} \|_{(0,\alpha),[\sigma,2\sigma]}\leq C.$$

From (\ref{eq83}) we obtain
$$ \sigma^{2-\nu}\|Q^1(\lambda G_{p}+u_{\varphi})\|_{(0,\alpha),[\sigma,2\sigma]} \leq C\sigma^{2-\nu}\|\lambda G_{p}+u_{\varphi}\|_{(0,\alpha),[\sigma,2\sigma]}^2 \leq C_\beta r^{\frac{5}{2}+d-\nu-\frac{n}{2}},$$
since $n\geq 3$, $\delta_4<1/2$, $r\leq\sigma$ and $\nu>3/2-n$ implies that $6+2d-\nu-n-2\delta_4>5/2+d-\nu-n/2$ and  $9/2-\nu-2n<0$.

Therefore
%
%%%%%%%%%%%%%%%%%%%%%%%%%%%%%%%%% Equation %%%%%%%%%%%%%%%%%%%%%%%%%%%%%%%%%%%%%
%
\begin{equation}\label{eq41}
\|Q^1(\lambda G_{p}+u_{\varphi}) \|_{C_{\nu-2}^{0,\alpha}(\Omega_{r,r_1})} \leq C_\beta r^{\frac{1}{2}} r^{2+d-\nu-\frac{n}{2}},
\end{equation}
and from (\ref{eq40}) and (\ref{eq41}), we get
%
%%%%%%%%%%%%%%%%%%%%%%%%%%%%%%%%% Equation %%%%%%%%%%%%%%%%%%%%%%%%%%%%%%%%%%%%%
%
\begin{equation}\label{eq42}
\|Q^1(\lambda G_{p}+u_{\varphi}) \|_{C_{\nu-2}^{0,\alpha}(M_r)}\leq C_\beta r^{\delta''} r^{2+d-\nu-\frac{n}{2}},
\end{equation}
for some constant $\delta''>0$ independent of $r$.

From $\Delta\mathcal{Q}_r= \delta^{ij}\partial_i\partial_j\mathcal{Q}_r=0$, $(g_0)_{ij}=\delta_{ij}+O(|x|^2)$ and $\det g_0=1+O(|x|^2)$, we obtain
$$ \|\Delta_{g_0}(u_{\varphi})\|_{(0,\alpha),[\sigma,2\sigma]} \leq C \left(\|\mathcal{Q}_r(\varphi)\|_{(2,\alpha),[\sigma,2\sigma]}+ \sigma^{-2}\|\eta\|_{(0,\alpha),[\sigma,2\sigma]} \|\mathcal{Q}_r(\varphi)\|_{(2,\alpha),[\sigma,2\sigma]}\right),$$
where the term with $\sigma^{-2}$ appears only for $\sigma>\frac{1}{4}r_1$ , since $\partial_i\eta\equiv 0$ in $B_{\frac{1}{2}r_1}(0)$.

Therefore, using that $3-n-\nu>0$ and $\delta_4\in(0,1/2)$ we get
\begin{equation*}\label{eq87}
\sigma^{2-\nu}\|L_{{g_0}}^1(u_\varphi) \|_{(0,\alpha),[\sigma,2\sigma]} 
\leq C_{r_1}r^{n-1} \|\varphi \|_{(2,\alpha),r}= C_{r_1}\beta r^{n-1+\nu-\delta_4}r^{2+d-\nu-\frac{n}{2}},
\end{equation*}
with $n-1+\nu-\delta_4>0$.

This implies
%
%%%%%%%%%%%%%%%%%%%%%%%%%%%%%%%%% Equation %%%%%%%%%%%%%%%%%%%%%%%%%%%%%%%%%%%%%
%
\begin{equation}\label{eq61}
\|L_{{g_0}}^1(u_\varphi) \|_{C^{0,\alpha}_{\nu-2}(\Omega_{r,r_1})} \leq C_{r_1}\beta r^{n-1+\nu-\delta_4}r^{2+d-\nu-\frac{n}{2}},
\end{equation}
with $n-1+\nu-\delta_4>0$.

Therefore, by (\ref{eq42}) and (\ref{eq61}) we obtain (\ref{eq114}) for $r>0$ small enough.

For the same reason as before, we have
$$\|\mathcal{M}_r(\lambda,\varphi,v_1)- \mathcal{M}_r(\lambda,\varphi,v_0)\|_{C^{2,\alpha}_\nu(M_r)} \leq c\|Q^1(\lambda G_{p}+u_\varphi+v_1)-Q^1(\lambda G_{p}+u_\varphi+v_0)\|_{C^{0,\alpha}_{\nu-2}(M_r)}.$$

Furthermore,
$$Q^1(\lambda G_{p}+u_\varphi+v_1)-Q^1(\lambda G_{p}+u_\varphi+v_0)= \displaystyle\frac{n(n+2)}{n-2} (v_1-v_0)\int^1_0\int_0^1(1+sz_t)^{\frac{6-n}{n-2}} z_tdsdt,$$
where $z_t=\lambda G_{p}+u_\varphi+v_0+t(v_1-v_0)$, since for small enough $r>0$ we have $0<c<1+sz_t<C$. This implies
$$\|(1+sz_t)^{\frac{6-n}{n-2}}\|_{C^{0,\alpha}(M_{\frac{1}{2}r_1})}\leq C\;\;\;\mbox{ and }\;\;\;\|(1+sz_t)^{\frac{6-n}{n-2}}\|_{(0,\alpha),[\sigma,2\sigma]}\leq C,$$
with the constant $C>0$ independent of $r$. Then, by (\ref{eq84}), we have
$$\|Q^1(\lambda G_{p}+v_1)- Q^1(\lambda G_{p}+v_0)\|_{C^{0,\alpha}(M_{r_1})} \leq C(r^{\frac{d}{2}-1+\frac{3n}{4}}+r^{2+d-\nu-\frac{n}{2}}) \|v_1-v_0\|_{C^{2,\alpha}_\nu(M_r)}$$
and
\begin{equation*}\label{eq88}
\sigma^{2-\nu}\|Q^1(\lambda G_{p}+u_{\varphi}+v_1)-Q^1(\lambda G_{p}+u_{\varphi}+v_0)\|_{(0,\alpha),[\sigma,2\sigma]}\leq$$ $$
\begin{array}{cl}
\hspace{1,5cm}\leq &  C\left(|\lambda|\sigma^{4-n} +\sigma^2\|u_{\varphi}\|_{(2,\alpha),[\sigma,2\sigma]}+ \sigma^2\|v_1\|_{(2,\alpha),[\sigma,2\sigma]}\right.\\
&\hspace{0,5cm}+ \left.\sigma^2\|v_0\|_{(2,\alpha), [\sigma,2\sigma]}\right)\sigma^{-\nu}\|v_1-v_0\| _{(0,\alpha),[\sigma,2\sigma]}\\
\hspace{1,5cm}\leq & C_{r_1,\beta}r^{2+\frac{d}{2}-\frac{n}{4}} \|v_1-v_0\| _{C^{2,\alpha}_\nu(M_r)}
\end{array}
\end{equation*}
since $1+\nu<0$, $2+d/2-n/4<3+d-n/2 <4+d-n/2-\delta_4$ and $0<\delta_4<1/2$. Notice that $2+d/2-n/4>0$.

Therefore, we deduce (\ref{eq115}) for small enough $r>0$.
\end{proof}
From Proposition \ref{propo07} we get the main result of this section.
%
%%%%%%%%%%%%%%%%%%%%%%%%%%%%%%%%% Theorem %%%%%%%%%%%%%%%%%%%%%%%%%%%%%%%%%%%%%
%
\begin{theorem}\label{teo02}
Let $\nu\in(3/2-n,2-n)$, $\delta_4\in(0,1/2)$, $\beta>0$ and $\gamma>0$ be fixed constants. There is $r_2\in(0,r_1/2)$ such that  if $r\in(0,r_2)$, $\lambda\in\mathbb{R}$ with $|\lambda|^2\leq r^{d-2+\frac{3n}{2}}$, and $\varphi\in C^{2,\alpha}(\mathbb{S}_r^{n-1})$ is $L^2-$orthogonal to the constant functions with $\|\varphi\|_{(2,\alpha),r}\leq \beta r^{2+d-\frac{n}{2}-\delta_4}$, then there is a solution $V_{\lambda,\varphi}\in C^{2,\alpha}_\nu(M_r)$ to the problem
$$\left\{\begin{array}{l}
H_{{g_0}} (1+\lambda G_{p}+ u_\varphi+V_{\lambda,\varphi})=0\;\;\;\mbox{ in }\;\;\;M_r\\
(u_{\varphi}+V_{\lambda,\varphi})\circ\Psi|_{\partial B_r(0)}-\varphi\in\mathbb{R}\;\;\;\mbox{ on }\;\;\;\partial M_r
\end{array}\right..$$

Moreover,
\begin{equation}\label{eq72}
\|V_{\lambda,\varphi}\|_{C^{2,\alpha}_\nu(M_r)}\leq \gamma r^{2+d-\nu-\frac{n}{2}},
\end{equation}
and
\begin{equation}\label{eq86}
\|V_{\lambda,\varphi_1}- V_{\lambda,\varphi_2}\|_{C^{2,\alpha}_\nu(M_r)} \leq Cr^{\delta_5-\nu}\|\varphi_1-\varphi_2\|_{(2,\alpha),r},
\end{equation}
for some constant $\delta_5>0$ small enough independent of $r$.
\end{theorem}
%
%%%%%%%%%%%%%%%%%%%%%%%%%%%%%%%%% Proof %%%%%%%%%%%%%%%%%%%%%%%%%%%%%%%%%%%%%
%
\begin{proof} The solution $V_{\lambda,\varphi}$ is the fixed point of $\mathcal{M}_r(\lambda,\varphi,\cdot)$ given by Proposition \ref{propo07} with the estimate (\ref{eq72}). The inequality (\ref{eq86}) follows similarly to (\ref{eq82}).
\end{proof}
Define $f:=1/\mathcal{F}$, where $\mathcal{F}$ is the function defined in Section \ref{sec07}. We have $g_0=f^{\frac{4}{n-2}}g$ with $f=1+O(|x|^2)$ in conformal normal coordinates centered at $p$. We will denote the full conformal factor of the resulting constant scalar curvature metric in $M_r$ with respect to the metric $g$ as $\mathcal{B}_r(\lambda,\varphi)$, that is, the metric
$$\tilde{g}=\mathcal{B}_r(\lambda,\varphi)^{\frac{4}{n-2}}g$$
has constant scalar curvature $R_{\tilde{g}}=n(n-1)$, where
$$\mathcal{B}_r(\lambda,\varphi):=f+\lambda fG_p+fu_{\varphi}+fV_{\lambda,\varphi}.$$

%
%%%%%%%%%%%%%%%%%%%%%%%%%%%%%%%%% CHAPTER %%%%%%%%%%%%%%%%%%%%%%%%%%%%%%%%%%%%%
%

\section{Constant Scalar Curvature on $M\backslash\{p\}$}\label{sec05}
%
%%%%%%%%%%%%%%%%%%%%%%%%%%%%%%%%% Section %%%%%%%%%%%%%%%%%%%%%%%%%%%%%%%%%%%%%
%
%\section{Introduction}
The main task of this section is to prove the following theorem:
%
%%%%%%%%%%%%%%%%%%%%%%%%%%%%%%%%% Theorem %%%%%%%%%%%%%%%%%%%%%%%%%%%%%%%%%%%%%
%
\begin{theorem}\label{teo04}
Let $(M^n,{g_0})$ be an $n-$dimensional compact Riemannian manifold of scalar curvature $R_{g_0}=n(n-1)$, nondegenerate about 1, and let $p\in M$ be such that $\nabla^kW_{g_0}(p)=0$ for $k=0,\ldots,d-2$, where $W_{g_0}$ is the Weyl tensor. Then there exist a constant $\varepsilon_0$ and a one-parameter family of complete metrics $g_\varepsilon$ on $M\backslash\{p\}$ defined for $\varepsilon\in(0,\varepsilon_0)$ such that:
\begin{enumerate}
\item[i)] each $g_\varepsilon$ is conformal to $g_0$ and has constant scalar curvature $R_{g_\varepsilon}=n(n-1)$;
\item[ii)] $g_\varepsilon$ is asymptotically Delaunay;
\item[iii)] $g_\varepsilon\rightarrow g_0$ uniformly on compact sets in $M\backslash\{p\}$ as $\varepsilon\rightarrow 0$.
\end{enumerate}
\end{theorem}

If the dimension is at most 5, no condition on the Weyl tensor is needed. Let us give some examples of non locally conformally flat manifolds for which the theorem applies.
\vspace{0.3cm}

\noindent{\bf Example:} The spectrum of the Laplacian on the $n-$sphere $\mathbb{S}^n(k)$ of constant curvature $k>0$ is given by Spec$(\Delta_g)=\{i(n+i-1)k; i=0,1,\ldots\}$. Consider the product manifolds $\mathbb{S}^2(k_1)\times \mathbb{S}^2(k_2)$ and $\mathbb{S}^2(k_3)\times \mathbb{S}^3(k_4)$. If we normalize so that the curvatures satisfy the conditions $k_1+k_2=6$ and $k_3+3k_4=10$, then the operator given in definition \ref{def3} with $u=1$ is equal to $L_{g_{12}}^1=\Delta_{g_{12}}+4$ and $L_{g_{34}}^1=\Delta_{g_{34}}+5$, where $g_{12}$ and $g_{34}$ are the standard metrics on $\mathbb{S}^2(k_1)\times \mathbb{S}^2(k_2)$ and $\mathbb{S}^2(k_3)\times \mathbb{S}^3(k_4)$, respectively. Notice that we have $R_{g_{12}}=12$ and $R_{g_{23}}=20$. 

It is not difficult to show that the spectra satisfy
$$\mbox{Spec}(L_{g_{12}}^1)\subseteq \{i(i+1)k_m-4;m=1,2\mbox{ and }i=0,1,\ldots\}\cup [8,\infty)$$
and
$$\mbox{Spec}(L_{g_{34}}^1)\subseteq \{i(i+1)k_3-4,i(i+2)k_4-4;i=0,1,\ldots\}\cup[6,\infty).$$

The product $\mathbb{S}^2(k_1)\times \mathbb{S}^2(k_2)$ with normalized constant scalar curvature equal to 12, is degenerate if and only if $k_1=4/(i(i+1))$ or $k_2=4/(i(i+1))$ for some $i=1,2,\ldots$ For the product $\mathbb{S}^2(k_3)\times \mathbb{S}^3(k_4)$ with normalized constant scalar curvature equal to 20, we conclude that it is degenerate if and only if $k_3=4/(i(i+1))$ or $k_4=4/(i(i+2))$, for some $i=1,2,\ldots$

Therefore we conclude that only countably many of these products are degenerate.
\vspace{0.3cm}

In previous sections we have constructed a family of constant scalar curvature metrics on $\overline{B_{r_\varepsilon}(p)}$, conformal to ${g_0}$ and singular at $p$, with  parameters $\varepsilon\in(0,\varepsilon_0)$ for some $\varepsilon_0>0$, $R>0$, $a\in\mathbb{R}^n$ and high eigenmode boundary data $\phi$. We have also constructed a family of constant scalar curvature metrics on $M_r=M\backslash B_r(p)$ conformal to ${g_0}$ with parameters $r\in(0,r_2)$ for some $r_2>0$, $\lambda\in\mathbb{R}$ and boundary data $\varphi$ $L^2-$orthogonal to the constant functions.

In this section we examine suitable choices of the parameter sets on each piece so that the Cauchy data can be made to match up to be $C^1$ at the boundary of $B_{r_\varepsilon}(p)$. In this way we obtain a weak solution to $H_{g_0}(u)=0$ on $M\backslash\{p\}$. It follows from elliptic regularity theory and the ellipticity of $H_{g_0}$ that the glued solutions are smooth metric.

To do this we will split the equation that the Cauchy data must  satisfy in an equation corresponding to the high eigenmode,  another one  corresponding to the space spanned by the constant functions, and $n$  equations corresponding to the space spanned by the coordinate functions.

%Finally, in the last section we explain briefly how to proceed with more than one end.
%
%%%%%%%%%%%%%%%%%%%%%%%%%%%%%%%%% Section %%%%%%%%%%%%%%%%%%%%%%%%%%%%%%%%%%%%%%
%
\subsection{Matching the Cauchy data}\label{sec14}
From Theorem \ref{teo01} there is a family of constant scalar curvature metrics in $\overline{B_{r_\varepsilon}(p)}\backslash\{p\}$, for small enough $\varepsilon>0$, satisfying the following:
$$\hat{g}=\mathcal{A}_\varepsilon(R,a,\phi)^{\frac{4}{n-2}}g,$$
with $R_{\hat{g}}=n(n-1)$,
%
%%%%%%%%%%%%%%%%%%%%%%%%%%%%%%%%% Equation %%%%%%%%%%%%%%%%%%%%%%%%%%%%%%%%%%%%%
%
\begin{equation*}\label{eq64}
\mathcal{A}_\varepsilon(R,a,\phi)=u_{\varepsilon,R,a}+w_{\varepsilon,R}+ v_\phi+U_{\varepsilon,R,a,\phi},
\end{equation*}
in conformal normal coordinates centered at $p$, and with
\begin{enumerate}
\item[I1)] $R^{\frac{2-n}{2}}=2(1+b)\varepsilon^{-1}$ and $|b|\leq 1/2$;
\item[I2)] $\phi\in\pi''(C^{2,\alpha} (\mathbb{S}^{n-1}_{r_\varepsilon}))$ with $\|\phi\|_{(2,\alpha),r_\varepsilon}\leq \kappa r_\varepsilon^{2+d-\frac{n}{2}-\delta_1}$, $\delta_1\in(0,(8n-16)^{-1})$ and $\kappa>0$ is some constant to be chosen later;
\item[I3)] $|a|r_\varepsilon^{1-\delta_2}\leq 1$ with $\delta_2>\delta_1$;
\item[I4)] $w_{\varepsilon,R}\equiv 0$ for $3\leq n\leq 7$, $w_{\varepsilon,R}\in\pi''(C^{2,\alpha}_{2+d-\frac{n}{2}} (B_{r_\varepsilon}(0)\backslash\{0\}))$ is the solution of the equation (\ref{eq35}) for $n\geq 8$;
\item[I5)] $U_{\varepsilon,R,a,\phi}\in C^{2,\alpha}_\mu(B_{r_\varepsilon}(0)\backslash\{0\})$ with $\pi''_{r_\varepsilon}(U_{\varepsilon,R,a,\phi}|_{\partial B_{r_\varepsilon}(0)})=0$, satisfies the inequality (\ref{eq82}) and has norm bounded by $\tau r_\varepsilon^{2+d-\mu-\frac{n}{2}}$, with $\mu\in(1,5/4)$ and $\tau>0$ is independent of $\varepsilon$ and $\kappa$.
\end{enumerate}

Also, from Theorem \ref{teo02} there is a family of constant scalar curvature metrics  in $M_{r_\varepsilon}=M\backslash B_{r_\varepsilon}(p)$, for small enough $\varepsilon>0$, satisfying the following:
$$\tilde{g}=\mathcal{B}_{r_\varepsilon}(\lambda,\varphi)^ {\frac{4}{n-2}}g,$$
with $R_{\tilde{g}}=n(n-1)$,
%
%%%%%%%%%%%%%%%%%%%%%%%%%%%%%%%%% Equation %%%%%%%%%%%%%%%%%%%%%%%%%%%%%%%%%%%%%
%
\begin{equation*}\label{eq65}
\mathcal{B}_{r_\varepsilon}(\lambda,\varphi)=f+\lambda f G_p+fu_\varphi+ fV_{\lambda,\varphi},
\end{equation*}
 in conformal normal coordinates centered at $p$, with
\begin{enumerate}
\item[E1)] $f=1+\overline{f}$ with $\overline{f}=O(|x|^2)$;
\item[E2)] $\lambda\in\mathbb{R}$ with $|\lambda|^2\leq r_\varepsilon^{d-2+\frac{3n}{2}}$;
\item[E3)] $\varphi\in C^{2,\alpha}(\mathbb{S}^{n-1}_{r_\varepsilon})$ is $L^2-$orthogonal to the constant functions and belongs to the ball of radius $\beta r_\varepsilon^{2+d-\frac{n}{2}-\delta_4}$, $\delta_4\in(0,1/2)$ and $\beta>0$ is a constant to be chosen later;
\item[E4)] $V_{\lambda,\varphi}\in C^{2,\alpha}_\nu(M_{r_\varepsilon})$ is constant on $\partial M_{r_\varepsilon}$, satisfies the inequality (\ref{eq86}) and has norm bounded by $\gamma r_\varepsilon^{2+d-\nu-\frac{n}{2}}$, with $\nu\in(3/2-n,2-n)$ and $\gamma>0$ is a constant independent of $\varepsilon$ and $\beta$.
\end{enumerate}

Recall that $r_\varepsilon=\varepsilon^s$ with $(d+1-\delta_1)^{-1}<s<4(d-2+3n/2)^{-1}$, see Remark \ref{rem01}. For example, we can choose $\delta_1=1/8n$ and $s=2(n-1-1/2n)^{-1}$.

We want to show that there are parameters, $R\in\mathbb{R}_+$, $a\in\mathbb{R}^n$, $\lambda\in\mathbb{R}$ and $\varphi,\phi\in C^{2,\alpha}(\mathbb{S}^{n-1}_{r_\varepsilon})$ such that
\begin{equation}\label{eq47}
\left\{\begin{array}{lcl}
\mathcal{A}_\varepsilon(R,a,\phi) & = & \mathcal{B}_{r_\varepsilon}(\lambda,\varphi)\\
\partial_r\mathcal{A}_\varepsilon(R,a,\phi) & = & \partial_r\mathcal{B}_{r_\varepsilon}(\lambda,\varphi)
\end{array}\right.
\end{equation}
on $\partial B_{r_\varepsilon}(p)$.

First, let $\delta_1\in(0,(8n-16)^{-1})$ be fixed. If we take $\omega$ and $\vartheta$ in the ball of radius $r_\varepsilon^{2+d-\frac{n}{2}-\delta_1}$ in $C^{2,\alpha}(\mathbb{S}^{n-1}_{r_\varepsilon})$, with $\omega$ belonging to the space spanned by the coordinate functions, $\vartheta$ belonging to the high eigenmode, and we define $\varphi:=\omega+\vartheta$, then we can apply Theorem \ref{teo02} with $\beta= 2$ and $\delta_4=\delta_1$, to define $\mathcal{B}_{r_\varepsilon}( \lambda,\omega+\vartheta)$, since $\|\varphi\|_{(2,\alpha),r_\varepsilon}\leq 2r_\varepsilon^{2+d-\frac{n}{2}-\delta_1}$.

Now define
\begin{equation}\label{eq66}
\begin{array}{lcl}
\phi_\vartheta & := & \pi''_{r_\varepsilon}((\mathcal{B}_{r_\varepsilon}( \lambda,\omega+\vartheta)-u_{\varepsilon,R,a}-w_{\varepsilon,R}) |_{\mathbb{S}_{r_\varepsilon}^{n-1}})\\
\\
& = & \pi''_{r_\varepsilon}((\overline{f}+ \lambda fG_p+\overline{f}u_{\omega+\vartheta}+\overline{f} V_{\lambda,\omega+\vartheta}- u_{\varepsilon,R,a} - w_{\varepsilon,R})|_{\mathbb{S}_{r_\varepsilon}^{n-1}})+\vartheta,
\end{array}
\end{equation}
where in the second equality we use that $\pi''_{r_\varepsilon}(u_{\omega+\vartheta}|_ {\mathbb{S}^{n-1}_{r_\varepsilon}})=\vartheta$, $\pi''_{r_\varepsilon}(V_{\lambda,\omega+\vartheta} |_{\mathbb{S}^{n-1}_{r_\varepsilon}})=0$ and $f=1+\overline{f}$, with $\overline{f}=O(|x|^2)$. 

We have to derive an estimate for $\|\phi_\vartheta \|_{(2,\alpha),r_\varepsilon}$. To do this, we will use the inequality (\ref{eq43}) in Lemma \ref{lem07}. But before, from (\ref{eq93}) in Corollary \ref{cor02}, we obtain
\begin{equation}\label{eq92}
\pi''_{r_\varepsilon}(u_{\varepsilon,R,a} |_{\mathbb{S}_{r_\varepsilon}^{n-1}})=O(|a|^2r_\varepsilon^2),
\end{equation}
since $r_\varepsilon=\varepsilon^s$ and $R^{\frac{2-n}{2}}=2(1+b)\varepsilon^{-1}$ with $s<4(d-2+3n/2)^{-1}<2(n-2)^{-1}$ and $|b|\leq 1/2$ implies that $R<r_\varepsilon$ for small enough $\varepsilon>0$. 

Let $1+d/2-n/4>\delta_2>\delta_1$ and let $a\in\mathbb{R}^n$ with $|a|^2\leq r_\varepsilon^{d-\frac{n}{2}}$ ($\delta_2=1/8$, for example). Hence we have that $|a|r_\varepsilon^{1-\delta_2}\leq r_\varepsilon^{1+\frac{d}{2}-\frac{n}{4}-\delta_2}$ tends to zero when $\varepsilon$ goes to zero, and I3) is satisfied for $\varepsilon>0$ small enough. Furthermore, since $|a|^2r_\varepsilon^2\leq r_\varepsilon^{2+d-\frac{n}{2}}$, we can show that
\begin{equation}\label{eq90}
\|\pi''_{r_\varepsilon}(u_{\varepsilon,R,a} |_{\mathbb{S}_{r_\varepsilon}^{n-1}})\|_{(2,\alpha),r_\varepsilon}\leq Cr_\varepsilon^{2+d-\frac{n}{2}},
\end{equation}
for some constant $C>0$ independent of $\varepsilon$, $R$ and $a$.

Observe that $(fG_p)(x)=|x|^{2-n}+O(|x|^{3-n})$ and $|\lambda|^2\leq r_\varepsilon^{d-2+\frac{3n}{2}}$ imply $\pi''_{r_\varepsilon}(\lambda (fG_p)|_{\mathbb{S}^{n-1}_{r_\varepsilon}})= O(r_\varepsilon^{2+\frac{d}{2}-\frac{n}{4}}),$
with $2+d/2-n/4>2+d-n/2$. Thus
\begin{equation}\label{eq91}
\|\pi''_{r_\varepsilon}(\lambda (fG_p)|_{\mathbb{S}^{n-1}_{r_\varepsilon}})\|_{(2,\alpha),r_\varepsilon}\leq Cr_\varepsilon^{2+d-\frac{n}{2}}.
\end{equation}
Now, using (\ref{eq74}), (\ref{eq57}), (\ref{eq72}), (\ref{eq66}), Lemma \ref{lem07} and the fact that $\overline{f}=O(|x|^2)$, we deduce that
\begin{equation}\label{eq99}
\|\phi_\vartheta-\vartheta \|_{(2,\alpha),r_\varepsilon}\leq cr_\varepsilon^{2+d-\frac{n}{2}},
\end{equation}
and
$$\|\phi_\vartheta\|_{(2,\alpha),r_\varepsilon}\leq cr_\varepsilon^{2+d-\frac{n}{2}-\delta_1},$$
for every $\vartheta\in \pi''(C^{2,\alpha}(\mathbb{S}^{n-1}_{r_\varepsilon}))$ in the ball of radius $r_\varepsilon^{2+d-\frac{n}{2}-\delta_1}$, for some constant $c>0$ that does not depend on $\varepsilon$.  Therefore we can apply Theorem \ref{teo01} with $\kappa$ equal to this constant $c$ and $\mathcal{A}_\varepsilon(R,a,\phi_\vartheta)$ is well defined. The definition (\ref{eq66}) immediately yields
$$\pi''_{r_\varepsilon}(\mathcal{A}_\varepsilon(R,a,\phi_\vartheta) |_{\mathbb{S}_{r_\varepsilon}^{n-1}}) =\pi''_{r_\varepsilon}(\mathcal{B}_{r_\varepsilon} (\lambda,\omega+\vartheta) |_{\mathbb{S}_{r_\varepsilon}^{n-1}}).$$

We project the second equation of the system (\ref{eq47}) on the high  eigenmode, the space of functions which are $L^2(\mathbb{S}^{n-1})-$orthogonal to $e_0$ $,\ldots,e_n$. This yields a nonlinear equation which can be written as
\begin{equation}\label{eq89}
r_\varepsilon \partial_r(v_{\vartheta}- u_{\vartheta})+ \mathcal{S}_\varepsilon(a,b,\lambda,\omega,\vartheta)=0,
\end{equation}
on $\partial_rB_{r_\varepsilon}(0)$, where
$$ \mathcal{S}_\varepsilon(a,b,\lambda,\omega,\vartheta) = r_\varepsilon\partial_r v_{\phi_\vartheta-\vartheta}+   r_\varepsilon\partial_r \pi''_{r_\varepsilon}(u_{\varepsilon,R,a}|_{\mathbb{S}^{n-1}_{r_\varepsilon}}) +r_\varepsilon\partial_r w_{\varepsilon,R}$$
$$+ r_\varepsilon \partial_r \pi''_{r_\varepsilon}((U_{\varepsilon,R,a,\phi_\vartheta} -  \overline{f}-\lambda fG_p-\overline{f}u_{\omega+\vartheta})|_{ \mathbb{S}^{n-1}_{r_\varepsilon}})
- r_\varepsilon \partial_r \pi''_{r_\varepsilon}((fV_{\lambda,\omega+\vartheta}) |_{ \mathbb{S}^{n-1}_{r_\varepsilon}}).$$

Since $v_\vartheta=\mathcal{P}_r(\vartheta)$ and $u_\vartheta=\mathcal{Q}_r(\vartheta)$ in $\Omega_{r_\varepsilon,\frac{1}{2}r_1}\subset M_{r_\varepsilon}$ for some $r_1>0$, see Section \ref{sec13}, from (\ref{eq62}) and (\ref{eq75}), we conclude that
$$r_\varepsilon\partial_r (v_\vartheta- u_\vartheta) (r_\varepsilon\cdot) =\partial_r(\mathcal{P}_1(\vartheta_1)-\mathcal{Q}_1(\vartheta_1)),$$
where $\vartheta_1\in C^{2,\alpha}(\mathbb{S}^{n-1})$ is defined by $\vartheta_1(\theta):=\vartheta(r\theta)$. Define an isomorphism $\mathcal{Z}:\pi''(C^{2,\alpha}(\mathcal{S}^{n-1}))\rightarrow \pi''(C^{1,\alpha}(\mathcal{S}^{n-1}))$ by
$$\mathcal{Z}(\vartheta):= \partial_r(\mathcal{P}_1(\vartheta)-\mathcal{Q}_1(\vartheta)),$$
(see \cite{M}, proof of Proposition 8 in \cite{MP} and proof of Proposition 2.6 in \cite{PR}).

To solve the equation (\ref{eq89}) it is enough to show that the map $\mathcal{H}_\varepsilon(a,b,\lambda,\omega,\cdot): \mathcal{D}_\varepsilon \rightarrow \pi''(C^{2,\alpha}(\mathcal{S}^{n-1}))$ given by
$$\mathcal{H}_\varepsilon(a,b,\lambda,\omega,\vartheta)= -\mathcal{Z}^{-1}(\mathcal{S}_\varepsilon(a,b,\lambda,\omega, \vartheta_{r_\varepsilon}) (r_\varepsilon\cdot)),$$
has a fixed point, where $\mathcal{D}_\varepsilon:=\{\vartheta\in \pi''(C^{2,\alpha}(\mathbb{S}^{n-1}));\|\vartheta\|_{(2,\alpha),1}\leq r_\varepsilon^{2+d-\frac{n}{2}-\delta_1}\}$ and $\vartheta_{r_\varepsilon}(x):=\vartheta(r_\varepsilon^{-1}x)$.

\begin{lemma}\label{lem08}
There is a constant $\varepsilon_0>0$ such that if $\varepsilon\in(0,\varepsilon_0)$, $a\in\mathbb{R}^n$ with $|a|^2\leq r_\varepsilon^{d-\frac{n}{2}}$, $b$ and $\lambda$ in $\mathbb{R}$ with $|b|\leq 1/2$ and $|\lambda|^2\leq r_\varepsilon^{d-2+\frac{3n}{2}}$, and $\omega\in C^{2,\alpha}(\mathbb{S}^{n-1}_{r_\varepsilon})$ belongs to the space spanned by the coordinate functions and with norm bounded by $r_\varepsilon^{2+d-\frac{n}{2}-\delta_1}$, then the map $\mathcal{H}_\varepsilon(a,b,\lambda,\omega,\cdot)$ has a fixed point in $\mathcal{D}_\varepsilon$.
\end{lemma}
\begin{proof}
As before, in Proposition \ref{propo03} and \ref{propo07} it is enough to show that
\begin{equation}\label{eq97}
\|\mathcal{H}_\varepsilon(a,b,\lambda,\omega,0)\|_{(2,\alpha),1}\leq \frac{1}{2}r_\varepsilon^{2+d-\frac{n}{2}-\delta_1}
\end{equation}
and
\begin{equation}\label{eq98}
\|\mathcal{H}_\varepsilon(a,b,\lambda,\omega,\vartheta_1)- \mathcal{H}_\varepsilon(a,b,\lambda,\omega,\vartheta_2)\|_{(2,\alpha),1}\leq \frac{1}{2} \|\vartheta_1-\vartheta_2\|_{(2,\alpha),1},
\end{equation}
for all $\vartheta_1,\vartheta_2\in\mathcal{D}_\varepsilon$.

Since $\mathcal{Z}$ is an isomorphism, we have that
$$\|\mathcal{H}_\varepsilon(a,b,\lambda,\omega,0)\|_{(2,\alpha),1}\leq C\|\mathcal{S}_\varepsilon (a,b,\lambda,\omega,0)\|_{(1,\alpha),r_\varepsilon}$$
where by (\ref{eq99}), $\phi_0$ satisfies
$$\|\phi_0\|_{(2,\alpha),r_\varepsilon}\leq cr_\varepsilon^{2+d-\frac{n}{2}},$$
where the constant $C>0$ and $c>0$ are independent of $\varepsilon$.

From (\ref{eq48}), (\ref{eq73}), (\ref{eq74}), (\ref{eq57}), (\ref{eq71}), (\ref{eq72}), (\ref{eq90}) and (\ref{eq91}) and the fact that $\overline{f}=O(|x|^2)$ we obtain
$$\|\mathcal{S}_\varepsilon (a,b,\lambda,\omega,0)\|_{(1,\alpha),r_\varepsilon}\leq cr_\varepsilon^{2+d-\frac{n}{2}}.$$
for some constant $c>0$ independent of $\varepsilon$. 

Therefore we get (\ref{eq97}) for small enough $\varepsilon$.

Now, we have
$$\|\mathcal{H}_\varepsilon(a,b,\lambda,\omega,\vartheta_1)- \mathcal{H}_\varepsilon(a,b,\lambda,\omega,\vartheta_2)\|_{(2,\alpha),1} \leq C\left(\|r_\varepsilon\partial_r v_{\phi_{\vartheta_{r_\varepsilon,1}}-\vartheta_{r_\varepsilon,1} -(\phi_{\vartheta_{r_\varepsilon,2}} -\vartheta_{r_\varepsilon,2})}\|_{(1,\alpha),r_\varepsilon}\right.$$
$$\hspace{-2cm}+ \|r_\varepsilon \partial_r\pi''_{r_\varepsilon} ((U_{\varepsilon,R,a,\phi_{\vartheta_{r_\varepsilon,1}}}- U_{\varepsilon,R,a,\phi_{\vartheta_{r_\varepsilon,2}}}) |_{\mathbb{S}^{n-1}_{r_\varepsilon}}) \|_{(1,\alpha),r_\varepsilon}$$
$$+ \|r_\varepsilon \partial_r\pi''_{r_\varepsilon}( (f(V_{\lambda,\omega+\vartheta_{r_\varepsilon,1}}- V_{\lambda,\omega+\vartheta_{r_\varepsilon,2}})) |_{\mathbb{S}^{n-1}_{r_\varepsilon}})\|_{(1,\alpha),r_\varepsilon}$$
$$\left.+\|r_\varepsilon \partial_r \pi''_{r_\varepsilon}((\overline{f}u_{\vartheta_{r_\varepsilon,1} -\vartheta_{r_\varepsilon,2}})|_{\mathbb{S}^{n-1}_{r_\varepsilon}}) \|_{(1,\alpha),r_\varepsilon}\right),$$
where, by (\ref{eq66})
$$\phi_{\vartheta_{r_\varepsilon,1}}-\vartheta_{r_\varepsilon,1} -(\phi_{\vartheta_{r_\varepsilon,2}} -\vartheta_{r_\varepsilon,2})=\pi''_{r_\varepsilon} ((\overline{f}u_{\vartheta_{r_\varepsilon,1}-\vartheta_{r_\varepsilon,2}} +\overline{f} (V_{\lambda,\omega+\vartheta_{r_\varepsilon,1}}- V_{\lambda,\omega+\vartheta_{r_\varepsilon,2}})) |_{\mathbb{S}_{r_\varepsilon}^{n-1}}).$$

Using the inequality (\ref{eq43}) of Lemma \ref{lem07}, (\ref{eq74}), (\ref{eq86}) and the fact that $\overline{f}=O(|x|^2)$, we obtain
$$\|\phi_{\vartheta_{r_\varepsilon,1}}-\vartheta_{r_\varepsilon,1}- (\phi_{\vartheta_{r_\varepsilon,2}}-\vartheta_{r_\varepsilon,2}) \|_{(2,\alpha),r_\varepsilon}\leq cr_\varepsilon^{\delta_6} \|\vartheta_{r_\varepsilon,1}- \vartheta_{r_\varepsilon,2}\|_{(2,\alpha),r_\varepsilon},$$
for some constants $\delta_6>0$ and $c>0$ that does not depend on $\varepsilon$. This implies
\begin{equation}\label{eq108}
\|r_\varepsilon\partial_r v_{\phi_{\vartheta_{r_\varepsilon,1}}-\vartheta_{r_\varepsilon,1} -(\phi_{\vartheta_{r_\varepsilon,2}} -\vartheta_{r_\varepsilon,2})}\|_{(1,\alpha),r_\varepsilon}\leq cr_\varepsilon^{\delta_6}\|\vartheta_1- \vartheta_2\|_{(2,\alpha),1}.
\end{equation}

From (\ref{eq82}) and (\ref{eq86}) we conclude that
$$\|U_{\varepsilon,R,a,\phi_{\vartheta_{r_\varepsilon,1}}}- U_{\varepsilon,R,a,\phi_{\vartheta_{r_\varepsilon,2}}}\|_{(2,\alpha), [\frac{1}{2}r_\varepsilon,r_\varepsilon]}\leq Cr_\varepsilon^{\delta_1} \|\vartheta_{r_\varepsilon,1}- \vartheta_{r_\varepsilon,2}\|_{(2,\alpha),r_\varepsilon}$$
and
$$\|V_{\lambda,\omega+\vartheta_{r_\varepsilon,1}}- V_{\lambda,\omega+\vartheta_{r_\varepsilon,2}}\|_{(2,\alpha), [r_\varepsilon,2r_\varepsilon]}\leq Cr_\varepsilon^{\delta_5} \|\vartheta_{r_\varepsilon,1}- \vartheta_{r_\varepsilon,2}\|_{(2,\alpha),r_\varepsilon},$$
for some $\delta_1>0$ and $\delta_5>0$ independent of $\varepsilon$. From this, (\ref{eq74}) and the fact that $f=1+\overline{f}$, we derive an estimate as (\ref{eq108}) for the other terms, and from this the inequality (\ref{eq98}) follows, since $\varepsilon$ is small enough.
\end{proof}

Therefore there exists a unique solution of (\ref{eq89}) in the ball of radius $r_\varepsilon^{2+d-\frac{n}{2}-\delta_1}$ in $C^{2,\alpha}(\mathbb{S}^{n-1}_{r_\varepsilon})$. We denote by $\vartheta_{\varepsilon, a,b,\lambda,\omega}$ this solution given by Lemma \ref{lem08}. Since this solution is obtained through the application of fixed point theorems for contraction mappings, it is continuous with respect to the parameters $\varepsilon$, $a$, $b$, $\lambda$ and $\omega$.

Recall that $R^{\frac{2-n}{2}}=2(1+b)\varepsilon^{-1}$ with $|b|\leq 1/2$. Hence, using (\ref{eq92}) and Corollary \ref{cor02} and \ref{cor04} we show that
$$\begin{array}{rcl}
u_{\varepsilon,R,a}(r_\varepsilon\theta) & = & 1+b+\displaystyle \frac{\varepsilon^2}{4(1+b)} {r_\varepsilon}^{2-n}+ ((n-2)u_{\varepsilon,R}(r_\varepsilon\theta)+r\partial_r u_{\varepsilon,R}(r_\varepsilon\theta))a\cdot x\\
& + & O(|a|^2r_\varepsilon^2) +O(\varepsilon^{2\frac{n+2}{n-2}}{r_\varepsilon} ^{-n}),
\end{array}$$
where the last term, $O(\varepsilon^{2\frac{n+2}{n-2}}r_\varepsilon^{-n})$, does not depend on $\theta$. Hence, we have
%
%%%%%%%%%%%%%%%%%%%%%%%%%%%%%%%%% Equation %%%%%%%%%%%%%%%%%%%%%%%%%%%%%%%%%%%%%
%
$$\mathcal{A}_\varepsilon(R,a,\phi_{\vartheta_ {\varepsilon,a,b,\lambda,\omega}})(r_\varepsilon\theta) =1+b+ \displaystyle\frac{\varepsilon^2}{4(1+b)}{r_\varepsilon}^{2-n} +v_{\phi_{\vartheta_{\varepsilon, a,b,\lambda,\omega}}}(r_\varepsilon\theta) +w_{\varepsilon,R}(r_\varepsilon\theta)$$
$$+ ((n-2)u_{\varepsilon,R}(r_\varepsilon\theta)+ r_\varepsilon \partial_ru_{\varepsilon,R}(r_\varepsilon\theta))r_\varepsilon a\cdot \theta $$
$$\hspace{2cm}+ U_{\varepsilon,R,a,\phi_{\vartheta_{\varepsilon, a,b,\lambda,\omega}}}(r_\varepsilon\theta)+ O(|a|^2r_\varepsilon^2)+ O(\varepsilon^{2\frac{n+2}{n-2}}r_\varepsilon^{-n}).$$

In the exterior manifold $M_{r_\varepsilon}$, in conformal normal coordinate system in the neighborhood of $\partial M_{r_\varepsilon}$, namely $\Omega_{r_\varepsilon,\frac{1}{2}r_1}$, we have
%
%%%%%%%%%%%%%%%%%%%%%%%%%%%%%%%%% Equation %%%%%%%%%%%%%%%%%%%%%%%%%%%%%%%%%%%%%
%
$$\mathcal{B}_{r_\varepsilon}(\lambda,\omega+\vartheta_{\varepsilon, a,b,\lambda,\omega})(r_\varepsilon\theta) = 1+\lambda r_\varepsilon^{2-n}+u_{\omega+\vartheta_{\varepsilon, a,b,\lambda,\omega}}(r_\varepsilon\theta) + \overline{f}(r_\varepsilon\theta)$$
$$+ (\overline{f}u_{\omega+\vartheta_{\varepsilon, a,b,\lambda,\omega}})(r_\varepsilon\theta) + (fV_{\lambda,\omega+\vartheta_{\varepsilon, a,b,\lambda,\omega}})(r_\varepsilon\theta) + O(|\lambda| r_\varepsilon^{3-n}).$$
Using that $w_{\varepsilon,R}\in \pi''(C_{2+d-\frac{n}{2}}^{2,\alpha}( B_{r_\varepsilon}(0)\backslash\{0\}))$, we now project the system (\ref{eq47}) on the set of functions spanned by the constant function. This yields the equations
\begin{equation}\label{eq03}
\left\{\begin{array}{rcl}
b+\left( \displaystyle\frac{\varepsilon^2}{4(1+b)}- \lambda\right)r_\varepsilon^{2-n} & = & \mathcal{H}_{0,\varepsilon}(a,b,\lambda,\omega)\\
(2-n)\left( \displaystyle\frac{\varepsilon^2}{4(1+b)}- \lambda\right)r_\varepsilon^{2-n} & = & r_\varepsilon\partial_r\mathcal{H}_{0,\varepsilon}(a,b,\lambda,\omega)
\end{array}\right.,
\end{equation}
where $\mathcal{H}_{0,\varepsilon}$ and $\partial_r \mathcal{H}_{0,\varepsilon}$ are continuous maps and satisfy
\begin{equation}\label{eq00}
\mathcal{H}_{0,\varepsilon}(a,b,\lambda,\omega)= O(r_\varepsilon^{2+d-\frac{n}{2}})\;\;\;\mbox{ and }\;\;\;r_\varepsilon\partial_r \mathcal{H}_{0,\varepsilon}(a,b,\lambda,\omega)= O(r_\varepsilon^{2+d-\frac{n}{2}}).
\end{equation}
\begin{lemma}\label{lem12}
There is a constant $\varepsilon_1>0$ such that if $\varepsilon\in(0,\varepsilon_1)$, $a\in\mathbb{R}^n$ with $|a|^2\leq r_\varepsilon^{d-\frac{n}{2}}$ and $\omega\in C^{2,\alpha}(\mathbb{S}^{n-1}_{r_\varepsilon})$ belongs to the space spanned by the coordinate functions and has norm bounded by $r_\varepsilon^{2+d-\frac{n}{2}-\delta_1}$, then the system (\ref{eq03}) has a solution $(b,\lambda)\in\mathbb{R}^2$, with $|b|\leq 1/2$ and $|\lambda|^2\leq r_\varepsilon^{d-2+\frac{3n}{2}}$.
\end{lemma}
\begin{proof} Define a continuous map $\mathcal{G}_{\varepsilon,a,\omega}: \mathcal{D}_{0,\varepsilon}\rightarrow \mathbb{R}^2$ by
$$\begin{array}{rcl}
\mathcal{G}_{\varepsilon,a,\omega}(b,\lambda) & := & \displaystyle \left(\frac{r_\varepsilon}{n-2}\partial_r\mathcal{H}_{0,\varepsilon} (a,b,\lambda,\omega)+\mathcal{H}_{0,\varepsilon} (a,b,\lambda,\omega),\right.\\
\\
& & \displaystyle \left.\frac{\varepsilon^2}{4(1+b)}+\frac{r_\varepsilon^{n-1}}{n-2} \partial_r\mathcal{H}_{0,\varepsilon} (a,b,\lambda,\omega)\right),
\end{array}$$
where $\mathcal{D}_{0,\varepsilon}:= \{(b,\lambda)\in\mathbb{R}^2; |b|\leq 1/2\mbox{ and } |\lambda|\leq r_\varepsilon^{\frac{d}{2}-1+\frac{3n}{4}}\}$.

Then, using (\ref{eq00}) and the fact that $2> s(d/2-1+3n/4)$, we can show that
$\mathcal{G}_{\varepsilon,a,\omega}(\mathcal{D}_{0,\varepsilon})\subset \mathcal{D}_{0,\varepsilon},$ for small enough $\varepsilon>0$. By the Brouwer's fixed point theorem it follows that there exists a fixed point of the map $\mathcal{G}_{\varepsilon,a,\omega}$. Obviously, this fixed point is a solution of the system (\ref{eq03}).
\end{proof}

With further work, one can also show that the mapping is a contraction, and hence that the fixed point is unique and depends continuously on the parameter $\varepsilon$, $a$ and $\omega$. 

From now on we will work with the fixed point given by Lemma \ref{lem12} and we will write simply as $(b,\lambda)$.

Finally, we project the system (\ref{eq47}) over the space of functions spanned by the coordinate functions. It will be convenient to decompose $\omega$ in
\begin{equation}\label{eq102}
\omega=\sum_{i=1}^n\omega_ie_i,\;\;\;\mbox{ where }\;\;\;\omega_i=\int_{\mathbb{S}^{n-1}}\omega(r_\varepsilon\cdot)e_i.
\end{equation}

Hence, $|\omega_i|\leq c_n\sup_{\mathbb{S}^{n-1}_{r_\varepsilon}}|\omega|.$ From this and Remark \ref{remark01} we get the system
\begin{equation}\label{eq101}
\left\{\begin{array}{rcl}
F(r_\varepsilon)r_\varepsilon a_i-\omega_i & = & \mathcal{H}_{i,\varepsilon}(a,\omega)\\
G(r_\varepsilon)r_\varepsilon a_i-(1-n)\omega_i & = & r_\varepsilon\partial_r \mathcal{H}_{i,\varepsilon}(a,\omega),
\end{array}\right.
\end{equation}
$i=1,\ldots,n$, where
$$F(r_\varepsilon):= (n-2)u_{\varepsilon,R}(r_\varepsilon\theta)+ r_\varepsilon\partial_r u_{\varepsilon,R}(r_\varepsilon\theta),$$
$$G(r_\varepsilon):= (n-2)u_{\varepsilon,R}(r_\varepsilon\theta)+ nr_\varepsilon\partial_r u_{\varepsilon,R}(r_\varepsilon\theta) +r_\varepsilon^2\partial_r^2 u_{\varepsilon,R}(r_\varepsilon\theta),$$
\begin{equation}\label{eq103}
\mathcal{H}_{i,\varepsilon}(a,\omega)= O(r_\varepsilon^{2+d-\frac{n}{2}})
\;\;\;\mbox{ and }\;\;\;r_\varepsilon\partial_r\mathcal{H}_{i,\varepsilon}(a,\omega)= O(r_\varepsilon^{2+d-\frac{n}{2}}).
\end{equation}
The maps $\mathcal{H}_{i,\varepsilon}$ and $\partial_r\mathcal{H}_{i,\varepsilon}$ are continuous.

\begin{lemma}\label{lem13}
There is a constant $\varepsilon_2>0$ such that if $\varepsilon\in(0,\varepsilon_2)$ then the system (\ref{eq101}) has a solution $(a,\omega)\in\mathbb{R}^n\times C^{2,\alpha}(\mathbb{S}^{n-1}_{r_\varepsilon})$ with $|a|^2\leq r_\varepsilon^{d-\frac{n}{2}}$ and $\omega$ given by (\ref{eq102}) of norm bounded by $r_{\varepsilon}^{2+d-\frac{n}{2}-\delta_1}$.
\end{lemma}
\begin{proof}
Define a continuous map $\mathcal{K}_{i,\varepsilon}:\mathcal{D}_{i,\varepsilon}\rightarrow \mathbb{R}^2$ by
$$\mathcal{K}_{i,\varepsilon}(a_i,\omega_i) := \left( (G(r_\varepsilon)+ (n-1)F(r_\varepsilon))^{-1}r_\varepsilon^{-1} (r_\varepsilon\partial_r\mathcal{H}_{i,\varepsilon}(a,\omega)+ (n-1)\mathcal{H}_{i,\varepsilon}),\right.$$
$$
\hspace{1,5cm}\left. (G(r_\varepsilon)+ (n-1)F(r_\varepsilon))^{-1}F(r_\varepsilon) (r_\varepsilon\partial_r\mathcal{H}_{i,\varepsilon}(a,\omega) +(n-1)\mathcal{H}_{i,\varepsilon})-\mathcal{H}_{i,\varepsilon}\right),$$
where $\mathcal{D}_{i,\varepsilon}:= \{(a_i,\omega_i)\in\mathbb{R}^2;|a_i|^2\leq n^{-1}r_\varepsilon^{d-\frac{n}{2}}\mbox{ and }|\omega_i|\leq n^{-1}k_{i,n}^{-1}r_\varepsilon^{2+d-\frac{n}{2}-\delta_1}\}$, $k_{i,n}=\|e_i\|_{(2,\alpha),1}$, $F(r_\varepsilon) = (n-2)(1+b)+ O(\varepsilon^{2-s(n-2)})$ and $G(r_\varepsilon)+ (n-1)F(r_\varepsilon)= n(n-2)(1+b)+ O(\varepsilon^{2-s(n-2)})$ with $2-s(n-2)>0$. 

From (\ref{eq103}) we obtain that $\mathcal{K}_{i,\varepsilon}(\mathcal{D}_{i,\varepsilon})\subset \mathcal{D}_{i,\varepsilon},$ for small enough $\varepsilon>0$. Again, by the Brouwer's fixed point theorem there exists a fixed point of the map $\mathcal{K}_{i,\varepsilon}$ and this fixed point is a solution of the system (\ref{eq101}).
\end{proof}

Now we are ready to prove the main theorem of this paper.
\begin{proof}[\bf Proof of Theorem \ref{teo04}] We keep the notation of the last section. Using Theorem \ref{teo01} we find a family of constant scalar curvature metrics in $\overline{B_{r_\varepsilon}(p)}\subset M$, for small enough $\varepsilon>0$, given by
$$\hat{g}=\mathcal{A}_\varepsilon (R,a,\phi)^{\frac{4}{n-2}}g,$$
with the parameters $R\in\mathbb{R}^+$, $a\in\mathbb{R}^n$ and $\phi\in\pi''(C^{2,\alpha}(\mathbb{S}_{r_\varepsilon}^{n-1}))$ satisfying the conditions I1--I5.

From Theorem \ref{teo02} we obtain a family of constant scalar curvature metrics in $M\backslash B_{r_\varepsilon}(p)$, for small enough $\varepsilon>0$, given by
$$\tilde{g}=\mathcal{B}_{r_\varepsilon}(\lambda,\varphi)^{\frac{4}{n-2}}g,$$
with the parameters $\lambda\in\mathbb{R}$ and $\varphi\in C^{2,\alpha}(\mathbb{S}_{r_\varepsilon}^{n-1})$ satisfying the conditions E1--E4. As before, the metric $g$ is conformal to the metric $g_0$.

From Lemmas \ref{lem08}, \ref{lem12} and \ref{lem13} we conclude that there is $\varepsilon_0>0$ such that for all $\varepsilon\in(0,\varepsilon_0)$ there are parameters $R_\varepsilon$, $a_\varepsilon$, $\phi_\varepsilon$, $\lambda_\varepsilon$ and $\varphi_\varepsilon$ for which the functions $\mathcal{A}_\varepsilon (R_\varepsilon,a_\varepsilon,\phi_\varepsilon)$ and $\mathcal{B}_{r_\varepsilon}(\lambda_\varepsilon,\varphi_\varepsilon)$ coincide up to order one in $\partial B_{r_\varepsilon}(p)$. Hence using elliptic regularity we show that the function $\mathcal{W}_\varepsilon$ defined by $\mathcal{W}_\varepsilon:=\mathcal{A}_\varepsilon (R_\varepsilon,a_\varepsilon,\phi_\varepsilon)$ in $B_{r_\varepsilon}(p)\backslash\{p\}$ and $\mathcal{W}_\varepsilon:= \mathcal{B}_{r_\varepsilon}(\lambda_\varepsilon,\varphi_\varepsilon)$ in $M\backslash B_{r_\varepsilon}(p)$ is a positive smooth function in $M\backslash\{p\}$. Moreover, $\mathcal{W}_\varepsilon$ tends to infinity on approach to $p$.

Therefore, the metric $g_\varepsilon:=\mathcal{W}_\varepsilon^{\frac{4}{n-2}}g$ is a complete smooth metric defined in $M\backslash\{p\}$ and by Theorem \ref{teo01} and \ref{teo02} it satisfies i), ii) and iii).
\end{proof}
%
%%%%%%%%%%%%%%%%%%%%%%%%%%%%%% Section 6 %%%%%%%%%%%%%%%%%%%%%%%%%%%%%%%%%%%
%
%
\section{Multiple point gluing}\label{sec16}
In this final section we discuss the minor changes that need to be made in order to deal with more than one singular point. Let $X=\{p_1,\ldots,p_k\}$ so that at each point we have $\nabla^l W_{g_0}(p_i)=0$, for $l=0,\dots,d-2$. 

As in the previous case, there are three steps. In Section \ref{sec02} we do not need to make any changes, since the analysis is done at each point $p_i$. Here, we find a family of metrics defined in $B_{r_{\varepsilon_i}}(p)\backslash\{p\}$, with $\varepsilon_i=t_i\varepsilon$, $\varepsilon>0$, $t_i\in(\delta,\delta^{-1})$ and $\delta>0$ fixed, $i=1,\ldots,k$.

In order to get a family of metrics as in Section \ref{sec04} we need to make some changes. Let $\Psi_i:B_{2r_0}(0)\rightarrow M$ be a normal coordinate system with respect to $g_i=\mathcal{F}_i^{\frac{4}{n-2}}g_0$ on $M$ centered at $p_i$. Here, $\mathcal{F}_i$ is such that as in Section \ref{sec04}. Therefore, each metric $g_i$ gives us conformal normal coordinates centered at $p_i$. Recall that $\mathcal{F}_i=1+O(|x|^2)$ in the coordinate system $\Psi_i$. Denote by $G_{p_i}$ the Green's function for $L_{g_0}^1$ with pole at $p_i$ and assume that $\displaystyle\lim_{x\rightarrow 0} |x|^{n-2}G_{p_i}(x)=1$ in the coordinate system $\Psi_i$. Let $G_{p_1,\ldots,p_k}\in C^\infty(M\backslash\{p_1,\ldots,p_k\})$ be such that
$$G_{p_1,\ldots,p_k}=\sum_{i=1}^k\lambda_i G_{p_i},$$
where $\lambda_i\in\mathbb{R}$.

Let $r=(r_{\varepsilon_1},\ldots,r_{\varepsilon_k})$. Denote by $M_{r}$ the complement in $M$ of the union of $\Psi_i(B_{r_{\varepsilon_i}}(0))$ and define the space $C_\nu^{l,\alpha} (M\backslash\{p_1,\ldots,p_k\})$ as in Definition \ref{def4}, with the following norm
$$\|v\|_{C_\nu^{l,\alpha}(M\backslash\{p\})}:= \|v\|_{C^{l,\alpha}(M_{\frac{1}{2}r_0})} +\sum_{i=1}^k\|v\circ\Psi_i\|_{(l,\alpha),\nu,r_0}.$$
The space $C_\nu^{l,\alpha} (M_{r})$ is defined similarly. 

It is possible to show an analogue of Proposition \ref{propo04} in this context, with $w\in\mathbb{R}$ constant on any component of $\partial M_{r}$.

Let $\varphi=(\varphi_1,\ldots,\varphi_k)$, with $\varphi_i\in C^{2,\alpha}(\mathbb{S}_r^{n-1})$ $L^2-$orthogonal to the constant functions. Let $u_\varphi\in C_\nu^{2,\alpha}(M_{r})$ be such that $u_\varphi\circ\Psi_i=\eta\mathcal{Q}_{r_{\varepsilon_i}}(\varphi_i)$, where $\eta$ is a smooth, radial function equal to 1 in $B_{r_0}(0)$, vanishing in $\mathbb{R}^n\backslash B_{2r_0}(0)$, and satisfying $|\partial_r\eta(x)|\leq c|x|^{-1}$ and $|\partial_r^2\eta(x)|\leq c|x|^{-2}$ for all $x\in B_{2r_0}(0)$.

Finally, in the same way that we showed the existence of solutions to the equation (\ref{eq25}), we solve the equation
$$H_{g_0}(1+G_{p_1,\ldots,p_k}+u_\varphi+u)=0.$$

The result reads as follows:
\begin{theorem}
Let $(M^n,g_0)$ be an $n-$dimensional compact Riemannian manifold of scalar curvature $n(n-1)$, nondegenerate about 1. Let $\{p_1,\ldots,p_k\}$ a set of points in $M$ so that $\nabla^j_{g_0}W_{g_0}(p_i)=0$ for $j=0,\ldots,\left[\frac{n-6}{2}\right]$ and $i=1,\ldots,k$, where $W_{g_0}$ is the Weyl tensor of the metric $g_0$. There exists a complete metric $g$ on $M\backslash\{p_1,\ldots,p_k\}$ conformal to $g_0$, with constant scalar curvature $n(n-1)$, obtained by attaching Delaunay-type ends to the points $p_1,\ldots,p_k$.
\end{theorem}

%\backmatter
%
%%%%%%%%%%%%%%%%%%%%%%%%%%%%%% BIBLIOGRAFIA %%%%%%%%%%%%%%%%%%%%%%%%%%%%%%%
%

\end{document}